\documentclass[11pt,a4paper]{amsart}

\usepackage[utf8]{inputenc}
\usepackage[T2A,T1]{fontenc}
\usepackage[russian,english]{babel}
\usepackage{microtype}
\usepackage{amsmath,amssymb,amsthm,mathtools,mathrsfs}
\usepackage{enumitem}
\usepackage{booktabs}
\usepackage{xcolor}
\usepackage[
  colorlinks=true,
  linkcolor=black,
  citecolor=black,
  urlcolor=black
]{hyperref}

\allowdisplaybreaks
\numberwithin{equation}{section}
\theoremstyle{plain}
\newtheorem{theorem}{Theorem}[section]
\newtheorem{proposition}[theorem]{Proposition}
\newtheorem{lemma}[theorem]{Lemma}
\newtheorem{corollary}[theorem]{Corollary}
\newtheorem{introtheorem}{Theorem}

\theoremstyle{definition}
\newtheorem{definition}[theorem]{Definition}
\newtheorem{example}[theorem]{Example}
\newtheorem{question}[theorem]{Question}
\theoremstyle{remark}
\newtheorem{remark}[theorem]{Remark}

\DeclareMathOperator{\adj}{adj}
\DeclareMathOperator{\Area}{Area}
\DeclareMathOperator{\Span}{span}
\DeclareMathOperator{\tr}{tr}
\DeclareMathOperator{\Gr}{Gr}
\DeclareMathOperator{\gr}{gr}
\DeclareMathOperator{\Tors}{Tors}
\DeclareMathOperator{\tors}{tors}
\DeclareMathOperator{\QGr}{QGr}
\DeclareMathOperator{\qgr}{qgr}
\DeclareMathOperator{\qprf}{qprf}
\DeclareMathOperator{\Perf}{Perf}
\DeclareMathOperator{\coh}{coh}
\DeclareMathOperator{\Qcoh}{Qcoh}
\newcommand{\PerfDG}[1]{\operatorname{Perf}_{\mathrm{dg}}(#1)}
\newcommand{\kk}{\Bbbk}

\newcommand{\GL}{\operatorname{GL}}

\newcommand{\PSL}{\operatorname{PSL}}
\newcommand{\SL}{\operatorname{SL}}

\title[Non-commutative Fano varieties and minifolds]
{Non-commutative Fano varieties and three-dimensional minifolds}
\author{Konstantin Loginov}
\address{Konstantin Loginov\\
Steklov Mathematical Institute, Moscow, Russia}
\email{loginov@mi-ras.ru}
\author{Dmitri Orlov}
\address{Dmitri Orlov\\
Steklov Mathematical Institute, Moscow, Russia}
\email{orlov@mi-ras.ru}
\date{}

\begin{document}

\begin{abstract}
We introduce notions of non-commutative Fano varieties and
non-commutative minifolds, and study the latter in dimension three.
We prove that the Euler matrices of full exceptional collections
generating geometric helices on these minifolds form five
mutation classes, distinguished by their Euler degrees
\(22,40,54,64,72.\)
The first four classes are realized by (commutative) Fano threefolds
\(X_{22},V_5,Q^3,\mathbb P^3,\) respectively.
The fifth is realized by the non-commutative family \(M_{72}\)
and cannot be realized by a full exceptional collection on
a smooth projective threefold.\end{abstract}


\maketitle

\tableofcontents

\section*{Introduction}
\begingroup
\raggedbottom
\hypersetup{linkcolor=black,citecolor=black,urlcolor=black}

Throughout the paper, \(\kk\) denotes an algebraically closed field of
characteristic zero, unless otherwise specified.
A smooth projective variety \(X\) over \(\kk\) can be studied
through the bounded derived category \(D^b(\operatorname{coh}X)\) of
coherent sheaves. Since \(X\) is smooth, this category is equivalent to
the category \(\Perf(X)\) of perfect complexes. It admits a differential
graded (DG) enhancement \(\PerfDG{X},\) in which morphisms form complexes
and composition is compatible with their differentials. Taking the
zeroth cohomology of these morphism complexes gives the morphisms of
\(\Perf(X).\)

Non-commutative algebraic geometry extends this approach to DG categories
that need not arise from commutative varieties. In this paper, a
non-commutative scheme over \(\kk\) is a \(\kk\)-linear DG category of
the form \(\PerfDG{\mathscr R},\) the DG category of perfect right
modules over a DG algebra \(\mathscr R.\) We require \(\mathscr R\) to be
cohomologically bounded, meaning that \(H^m(\mathscr R)=0\) for all but
finitely many \(m.\)
Some notions of commutative algebraic geometry, such as smoothness
and properness, extend to non-commutative schemes,
see Definition~\ref{def:smooth-proper-nc-scheme}.

Let \(X\) be a smooth Fano variety of dimension \(n\) over
\(\kk.\)  If \(S_X\) is its Serre
functor, then \(T_X=S_X^{-1}[n]\) is tensoring with
the anticanonical line bundle \(\omega_X^{-1}.\) One can consider its anticanonical algebra
\(R(X,\mathcal O_X)=\bigoplus_{p\geq0}
\operatorname{Hom}(\mathcal O_X,T_X^p\mathcal O_X).\)  Kodaira vanishing gives
\(\operatorname{Hom}(\mathcal O_X,T_X^p\mathcal O_X[k])=0\) for
\(p\geq0\) and \(k\ne0.\)
More generally, for any object \(E\) of a \(\kk\)-linear triangulated
category \(\mathcal D\) and any \(\kk\)-linear autoequivalence \(T\) of
\(\mathcal D,\) one can consider the graded algebra
\(R=\bigoplus_{p\geq0}\operatorname{Hom}_{\mathcal D}(E,T^pE).\)
In the commutative case, Serre's theorem gives
\(\Qcoh X\simeq\QGr R(X,\mathcal O_X),\) see \cite{Serre1955}.
Here \(\QGr R\) is the quotient of the category of graded right
\(R\)-modules by the subcategory of torsion modules.

We propose a definition of a non-commutative Fano variety that
generalizes the above properties of smooth Fano varieties,
see Definition~\ref{def:orlov-fano-variety}. Let \(\mathsf X\) be a smooth proper
non-commutative scheme. Then the category \(\Perf(\mathsf X)\)
admits a Serre functor \(S.\) We fix an integer \(n\geq0\) and put
\(T=S^{-1}[n].\) We require an exceptional object
\(E\in\Perf(\mathsf X)\) satisfying
\(\operatorname{Hom}(E,T^pE[k])=0\) for \(p\geq0\) and \(k\ne0.\)
We also require an equivalence between the unbounded derived category
\(\mathcal D(\mathsf X)\) of quasi-coherent sheaves on \(\mathsf X\) and the unbounded derived category
\(\mathcal D(\QGr R).\) On \(\Perf(\mathsf X),\) this equivalence must be given by
the graded-Hom functor specified in
Definition~\ref{def:orlov-fano-variety}\textup{(ii)}.
A non-commutative Fano variety is \emph{of geometric type} if,
for some such \(E,\) we have that
\(\dim_{\kk}\operatorname{Hom}(E,T^pE)\) is given by a polynomial in \(p\) for every
\(p\geq0,\) see Definition~\ref{def:geometric-fano}.  The degree of this polynomial is called the Hilbert dimension of the non-commutative Fano variety with respect to \(E.\)
In general, in the non-commutative case, polynomial growth does not follow
from the other conditions, see
Example~\ref{ex:three-kronecker}.
Every smooth Fano variety over \(\kk\) is a non-commutative
Fano variety of geometric type,
see Proposition~\ref{prop:commutative-fano-is-geometric}.

Let \(\mathcal D=\Perf(\mathsf X),\) where \(\mathsf X\) is a smooth proper
non-commutative scheme. An object \(E\in\mathcal D\) is \emph{exceptional}
if its endomorphism complex has cohomology \(\kk\) in degree zero and
zero in every other degree. An \emph{exceptional collection} in
\(\mathcal D\) is an ordered sequence \(\mathcal E=(E_1,\ldots,E_s)\) of exceptional
objects such that \(\operatorname{Hom}_{\mathcal D}(E_j,E_i[m])=0\)
for \(j>i\) and every \(m\in\mathbb Z.\) It is \emph{full} if its objects
generate \(\mathcal D\) under taking shifts, cones, and direct summands, see
Definition~\ref{def:exceptional-strong}.

Left and right \emph{mutations} replace a pair of exceptional objects
\((E,F)\) by \((L_EF,E)\) and \((F,R_FE),\) respectively.
Mutations satisfy the braid relations. Thus, the
braid group \(B_s\) on \(s\) strands acts on exceptional collections of length \(s\) in
\(\mathcal D.\)

We study minifolds and their non-commutative analogues. The notion of a
minifold was introduced in \cite[Definition~2.4]{GKMS2013}. Recall that
an \(n\)-dimensional commutative minifold is a smooth projective variety
\(X\) such that \(D^b(\operatorname{coh}X)\) admits a full exceptional
collection of the minimal possible length \(n+1,\) see Definition~\ref{def:usual-minifold}.

Projective spaces \(\mathbb P^n\) and smooth odd-dimensional quadrics
over \(\kk\) are examples of minifolds \cite{Beilinson1978,Kapranov1988}.
In dimension three, every minifold is isomorphic to
\(\mathbb P^3,\) \(Q^3,\) \(V_5,\) or a member of the family \(X_{22},\) see
\cite[Theorem~1.1(2)]{GKMS2013}.
Here \(Q^3\subset\mathbb P^4\) is a smooth quadric hypersurface,
\(V_5=\operatorname{Gr}(2,5)\cap\mathbb P^6\subset\mathbb P^9\)
is a smooth linear section of the Grassmannian in its Pl\"ucker embedding, and \(X_{22}\)
denotes the Fano threefolds of Picard rank one and
anticanonical degree \(22.\)

We propose the following definition. An \emph{\(n\)-dimensional
non-commutative minifold} is a smooth proper non-commutative scheme
\(\mathsf X\) such that
\(\mathcal D=\Perf(\mathsf X)\) admits a full exceptional collection of
length \(n+1\) generating a geometric helix, and \(T=S^{-1}[n],\) where
\(S\) is the Serre functor of \(\mathcal D,\) acts maximally unipotently
on \(K_0(\mathcal D)\otimes\mathbb Q,\)
see Definition~\ref{def:orlov-minifold}. Here a \emph{helix} is an
infinite sequence in \(\mathcal D\) obtained by extending an exceptional
collection through mutations in both directions. It is \emph{geometric}
if the morphism complexes from earlier to later terms have cohomology
only in degree zero, see Definition~\ref{def:geometric-helix} and
\cite{BondalPolishchuk1994}. Such a scheme is a non-commutative
Fano variety of geometric type, and its Hilbert and Serre dimensions are
both equal to \(n,\) see Propositions~\ref{prop:geometric-collection-fano}
and~\ref{prop:minifold-hilbert-degree}.
If a commutative \(n\)-dimensional minifold \(X\) admits a full exceptional
collection of length \(n+1\) generating a geometric helix, then
\(\PerfDG{X}\) satisfies our definition, see
Proposition~\ref{prop:strong-is-noncommutative-minifold}.
In particular, this applies to all three-dimensional minifolds,
see the discussion following
Theorem~\ref{thm:threefold-minifolds}.\par

Let \(\mathsf X\) be a three-dimensional non-commutative minifold
and put \(\mathcal D=\Perf(\mathsf X).\) For objects \(E,F\in\mathcal D,\)
put \(\chi(E,F)=\sum_{m\in\mathbb Z}(-1)^m
\dim_{\kk}\operatorname{Hom}_{\mathcal D}(E,F[m]),\)
where the sum is finite by properness. Additivity in distinguished
triangles makes \(\chi\) an integral bilinear form on the Grothendieck
group \(K_0(\mathcal D),\) called the \emph{Euler form}, see
Definition~\ref{def:euler-form}. The group \(K_0(\mathcal D)\) together
with this form is preserved by \(\kk\)-linear exact equivalences and
provides a numerical invariant of \(\mathcal D.\)

For a full exceptional collection \(\mathcal E=(E_1,E_2,E_3,E_4)\) in \(\mathcal D,\)
the classes \([E_1],\ldots,[E_4]\) form a basis of \(K_0(\mathcal D).\) Mutations change this basis
and hence the matrix representing the Euler form.
The \emph{Euler matrix} of the exceptional collection \(\mathcal E\) has the form
\[
 A=A(a,b,c,d,e,f)=
 \begin{pmatrix}
 1&a&e&d\\
 0&1&b&f\\
 0&0&1&c\\
 0&0&0&1
 \end{pmatrix},
 \qquad A_{ij}=\chi(E_i,E_j).
\]

We say that an Euler matrix is \emph{positive} if its six parameters $a,b,c,d,e,f$ are at least three,
every \(3\times3\) minor of \(A+A^{\mathsf t}\) is positive, and
\(C=-A^{-1}A^{\mathsf t}\) is maximally unipotent,
see Definition~\ref{def:positive-euler-matrix}.
The following positivity result follows from
Theorem~\ref{thm:geometric-euler-positivity}.

\begin{introtheorem}\label{thm:intro-positivity}
Let \(\mathsf X\) be a three-dimensional non-commutative minifold, and let
\((E_i)_{i\in\mathbb Z}\) be a full geometric helix of period four in
\(\Perf(\mathsf X).\) Then the Euler matrix of the exceptional collection
\((E_i,E_{i+1},E_{i+2},E_{i+3})\) is positive for every \(i\in\mathbb Z.\)
\end{introtheorem}

The braid group \(B_4\) acts on Euler matrices by mutations, see
Lemma~\ref{lem:coordinate-mutations}. This action preserves positivity, see
Proposition~\ref{prop:mutation-positivity}.

For a positive Euler matrix, put \(B=A+A^{\mathsf t}.\)  Then \(B\) has
rank three.  If \(p\) is a primitive generator of
\(\ker B\cap\mathbb Z^4,\) then
\(\adj(B)=D(A)pp^{\mathsf t}\) for a positive even integer \(D(A)\) where \(\adj(B)\) denotes the adjoint of \(B.\)
We call \(D(A)\) the \emph{Euler degree}.  It is invariant under mutations, see Lemma~\ref{lem:formal-mutation-invariants}.
For the four commutative minifolds \(X,\) it equals the anticanonical
degree \((-K_X)^3,\) see Proposition~\ref{prop:strong-minifold-euler-degree}.
\begin{samepage}
Our main result follows from Theorem~\ref{thm:main-reduction},
Proposition~\ref{prop:strong-minifold-euler-degree}, and the realizations in
Subsections~\ref{subsec:commutative-minifolds} and~\ref{subsec:m72-construction}.

\begin{introtheorem}\label{thm:intro-classification}
The action of the braid group \(B_4\) on positive \(4\times4\) Euler
matrices has five orbits. The Euler degree distinguishes these orbits.
The table lists the Euler degree, a representative, and a family of
three-dimensional non-commutative minifolds realizing each orbit:
\[
\begin{array}{c|c|c}
D(A)&(a,b,c,d,e,f)&\textnormal{family of minifolds}\\
\hline
22&(4,4,7,7,3,20)&X_{22}\\
40&(5,5,5,5,3,18)&V_5\\
54&(4,5,5,4,4,11)&Q^3\\
64&(4,4,4,4,6,6)&\mathbb P^3\\
72&(3,6,3,6,7,7)&M_{72}.
\end{array}
\]
\end{introtheorem}
\par
\end{samepage}

The first four orbits are realized by commutative minifolds.
The orbit of Euler degree \(72\) cannot be realized by a full exceptional
collection of length four on a smooth projective threefold,
see Theorem~\ref{thm:threefold-minifolds} and
Propositions~\ref{prop:euler-degree-arithmetic}
and~\ref{prop:strong-minifold-euler-degree}.
It is realized by the family of non-commutative minifolds \(M_{72}\)
constructed in \cite{Orlov2026}, see Subsection~\ref{subsec:m72-construction}.
Theorem \ref{thm:intro-classification} classifies Euler matrices up to mutations; it does not
classify the categories having these
matrices. Nogin
\cite{Nogin1994} studies semiorthonormal
bases of \(K_0(X),\) equipped with its Euler form, for each of the four
commutative minifolds \(X.\) A basis \((v_1,\ldots,v_s)\) of
\(K_0(X)\) is \emph{semiorthonormal} if \(\chi(v_i,v_i)=1\) and
\(\chi(v_j,v_i)=0\) for \(j>i.\)

This classification of Euler matrices suggests a broader
boundedness question. Smooth Fano varieties
form a bounded family in each fixed dimension
\cite{KollarMiyaokaMori1992}.
It is natural to ask whether non-commutative Fano varieties of geometric
type are bounded in each fixed Serre dimension.  A precise version of this
question requires a definition of families of non-commutative schemes, which
we do not give here.

We prove Theorem \ref{thm:intro-classification} by passing from Euler matrices to character varieties,
see Subsection~\ref{subsec:trace-construction}.
For every positive Euler matrix \(A,\) the trace correspondence of
\cite[Propositions~5.9 and~6.8]{FanWhang2024} determines a unique
\(\SL_2\)-character \([\rho_A]\) of the twice-punctured torus,
see Proposition~\ref{prop:euler-fricke-dictionary}.
The map \(A\mapsto[\rho_A]\) is equivariant for the braid-group actions
specified in \eqref{eq:braid-action-euler-bases} and
\eqref{eq:character-braid-action}.  Then we realize \([\rho_A]\) by a representation
\(F_3\to\SL_2(\mathbb Z),\)
see Lemma~\ref{lem:integral-traces-modular}. By the \hyperref[app:half-turn-realization]{Appendix}, the
projective image \(\Gamma\subset\PSL_2(\mathbb Z)\) is torsion-free of
index twelve, and \(\Gamma\backslash\mathbb H\) is a torus with two cusps.
There are five such \(\PSL_2(\mathbb Z)\)-conjugacy classes, see
\cite[Section~6]{SebbarBesrour2021}.
Positive Euler matrices with \(\PSL_2(\mathbb Z)\)-conjugate modular images
are mutation-equivalent, see Proposition~\ref{prop:modular-orbit-injectivity}.
The five matrices in the theorem realize all five classes.
\par\endgroup

We thank Anton Fonarev and Alexander Kuznetsov for useful discussions.
OpenAI's ChatGPT was used to explore proof ideas, test preliminary
arguments, check computations, and identify possible references.  The authors
independently checked every argument, wrote the final text, and take
full responsibility for its content.

\section{Preliminaries}\label{sec:preliminaries}

\subsection{DG algebras, DG categories, and non-commutative schemes}

\begingroup\begingroup
All categories are \(\kk\)-linear, and modules are right modules.  We use standard
conventions; see \cite{Drinfeld2004,Keller1994,Keller2006,LuntsOrlov2010,Toen2007}.

\begin{definition}\label{def:dg-algebra-category}
A \emph{DG algebra} \(\mathscr R\) over \(\kk\) is a unital graded algebra
\(\mathscr R=\bigoplus_{m\in\mathbb Z}\mathscr R^m\) equipped with a
differential of degree one such that \(d^2=0\) and
$d(xy)=d(x)y+(-1)^{\deg x}x d(y)$
for homogeneous elements \(x,y\in\mathscr R.\)
A \emph{DG category} \(\mathscr C\) over \(\kk\) is a category whose
morphism spaces \(\mathsf{Hom}_{\mathscr C}(E,F)\) are complexes of
\(\kk\)-vector spaces and whose composition maps are morphisms of
complexes.
\end{definition}

Thus, a DG algebra is a DG category with one object.  The
homotopy category \(H^0(\mathscr C)\) has the same objects as
\(\mathscr C\) and morphisms
\(H^0(\mathsf{Hom}_{\mathscr C}(E,F)).\)
A \emph{DG enhancement} of a triangulated category \(\mathcal D\) is a
pretriangulated DG category \(\mathcal D_{\mathrm{dg}},\) together with
an exact equivalence \(H^0(\mathcal D_{\mathrm{dg}})\simeq\mathcal D.\)
If \(\mathscr C\) is pretriangulated, then \(H^0(\mathscr C)\) is
triangulated.

A right DG \(\mathscr R\)-module is a graded right \(\mathscr R\)-module with a
compatible differential. The unbounded derived category
\(\mathcal D(\mathscr R)\) is obtained from the category of right DG
\(\mathscr R\)-modules by inverting quasi-isomorphisms.
The perfect derived category is
\[
 \Perf(\mathscr R)=\mathcal D(\mathscr R)^c
 =\operatorname{thick}(\mathscr R).
\]
Here \(\mathcal D(\mathscr R)^c\) is the full subcategory of compact objects.
An object \(P\in\mathcal D(\mathscr R)\) is \emph{compact} if the functor
\(\operatorname{Hom}_{\mathcal D(\mathscr R)}(P,-)\) commutes with
arbitrary direct sums. Here \(\operatorname{thick}(\mathscr R)\)
is the smallest full triangulated subcategory of \(\mathcal D(\mathscr R)\)
containing the right DG module \(\mathscr R\) and closed under taking
direct summands. We denote its standard DG model by
\(\PerfDG{\mathscr R},\) so that
\[
 H^0\bigl(\PerfDG{\mathscr R}\bigr)\simeq\Perf(\mathscr R).
\]
The algebra \(\mathscr R\) is \emph{cohomologically bounded} if
\(H^m(\mathscr R)=0\) for all but finitely many \(m.\)
\endgroup\endgroup

\begin{definition}
\label{def:noncommutative-scheme}
A \emph{(derived) non-commutative scheme} \(\mathsf X\) over
\(\kk\) is a \(\kk\)-linear DG category of the form
\(\PerfDG{\mathscr R},\) where \(\mathscr R\) is a cohomologically bounded
DG algebra over \(\kk.\)  The derived category \(\mathcal D(\mathscr R)\)
is called the derived category of quasi-coherent sheaves on this
non-commutative scheme, while the triangulated category
\(\Perf(\mathscr R)\) is called the category of perfect complexes on it.
\end{definition}

In what follows, we omit the word ``derived'' and use the shorter term
\emph{non-commutative scheme}.  We write
\[
    \mathcal D(\mathsf X):=  \mathcal D(\mathscr R),\qquad   \Perf(\mathsf X):=\Perf(\mathscr R).
\]

\begin{definition}\label{def:smooth-proper-nc-scheme}
Let \(\mathsf X=\PerfDG{\mathscr R}\) be a non-commutative scheme. It is called
\emph{proper} if, for all \(M,N\in\Perf(\mathsf X),\) the  \(\kk\)-vector
space \(\bigoplus_{m\in\mathbb Z}
\operatorname{Hom}_{\Perf(\mathsf X)}(M,N[m])\) is finite-dimensional.
Equivalently, the total cohomology space
\(\bigoplus_{m\in\mathbb Z}H^m(\mathscr R)\) is finite-dimensional.

The scheme is \emph{smooth} if its diagonal bimodule is perfect, that is, if
\(\mathscr R\in
\Perf\bigl(\mathscr R^{\mathrm{op}}\otimes_{\kk}\mathscr R\bigr).\)
\end{definition}

Smoothness is invariant under Morita equivalence
\cite{LuntsSchnurer2014}, and hence does not depend on the presentation by
\(\mathscr R.\)  If \(X\) is a separated scheme of finite type over \(\kk,\)
then \(X\) is proper if and only if its standard DG category of perfect
complexes is proper, and \(X\) is smooth if and only if this DG
category is smooth {\cite[Propositions~3.30 and~3.31]{Orlov2016}}.
We denote the latter
DG category by \(\PerfDG{X}.\)

\begin{definition}\label{def:serre-functor} Let \(\mathcal D\) be a \(\kk\)-linear Hom-finite triangulated category.
  A
\emph{Serre functor} on \(\mathcal D\) is a \(\kk\)-linear autoequivalence
\(S\colon\mathcal D\to\mathcal D\) together with functorial isomorphisms
\[
 \operatorname{Hom}_{\mathcal D}(Y,SX)
 \simeq
 \operatorname{Hom}_{\mathcal D}(X,Y)^{\vee}.
\]
\end{definition}

\begingroup\begingroup
A Serre functor is exact and unique up to natural isomorphism
\cite{BondalKapranov1990}.  For every smooth proper non-commutative scheme
\(\mathsf X=\PerfDG{\mathscr R},\) it exists and is given by
\[
 S_{\mathsf X}(-)=-\overset{\mathbf L}{\otimes}_{\mathscr R}\mathscr R^\vee,
 \qquad \mathscr R^\vee=\mathsf{Hom}_{\kk}(\mathscr R,\kk).
\]
\endgroup\endgroup

\subsection{Graded algebras, quotient categories, and Serre's theorem}
\label{sec:graded-quotient-categories}

A graded algebra \(R=\bigoplus_{m\ge0}R_m\) is \emph{connected} if
\(R_0=\kk,\) and \emph{locally finite} if each \(R_m\) is finite-dimensional.
All connected graded algebras considered below are locally finite.

\begin{definition}\label{def:z-algebra}
A \emph{\(\mathbb Z\)-algebra} over \(\kk\) is a \(\kk\)-linear category with
object set \(\mathbb Z.\)  Set
\(\mathcal A_i^j=\operatorname{Hom}_{\mathcal A}(i,j).\)  Equivalently, it
is the locally unital algebra
\(\mathcal A=\bigoplus_{i,j\in\mathbb Z}\mathcal A_i^j,\) whose multiplication
is induced by composition:
\[
 \mathcal A_j^k\otimes\mathcal A_i^j\longrightarrow\mathcal A_i^k,
 \qquad b\otimes a\longmapsto b\circ a.
\]
The identities \(e_i=\operatorname{id}_i\) are
pairwise orthogonal idempotents, and every element of \(\mathcal A\) has a
two-sided identity of the form \(\sum_{i\in F}e_i\) for some finite subset
\(F\subset\mathbb Z.\) In particular,
\(e_j\mathcal A e_i=\mathcal A_i^j.\)
\end{definition}

\begingroup\begingroup
Right modules over a \(\mathbb Z\)-algebra \(\mathcal A\) are
understood to be unitary: \(M=\bigoplus_{i\in\mathbb Z}Me_i.\)
Any graded algebra \(R=\bigoplus_{p\in\mathbb Z}R_p\) determines a
\(\mathbb Z\)-algebra by \(\mathcal A_i^j=R_{i-j}.\)
A graded right \(R\)-module \(N\) corresponds to the right
\(\mathcal A\)-module \(M\) with \(Me_i=N_i,\) with the action
\(Me_j\otimes\mathcal A_i^j\to Me_i\) induced by multiplication in
\(N.\) This gives an equivalence of the
two module categories.
\endgroup\endgroup
{This \(\mathbb Z\)-algebra} is
\emph{1-periodic}: it has an automorphism \(\tau\) satisfying
\(\tau(\mathcal A_i^j)=\mathcal A_{i+1}^{j+1}.\)  More generally, a
\(\mathbb Z\)-algebra is called \emph{\(l\)-periodic} if it admits an
automorphism \(\tau\) satisfying
\(\tau(\mathcal A_i^j)=\mathcal A_{i+l}^{j+l}.\)
A \(\mathbb Z\)-algebra \(\mathcal A\) is called \emph{connected} if
all spaces \(\mathcal A_i^j\) are finite-dimensional,
\(\mathcal A_i^i\simeq\kk,\) and \(\mathcal A_i^j=0\) for \(j>i.\)

For a right module \(M\) over a connected \(\mathbb Z\)-algebra
\(\mathcal A,\) a homogeneous element
\(x\in Me_i,\) where \(e_i\) is the identity of the object \(i,\) is called
\emph{torsion} if \(x\mathcal A_j^i=0\) for every \(j\ge i+p,\) for some
\(p=p(x)\ge1.\)
An arbitrary element is torsion if each of its finitely many
homogeneous components is torsion.  The~torsion~elements form a submodule
\(t(M)\subset M,\) and \(M\) is called a \emph{torsion module} if
\(t(M)=M.\)  \begingroup
For the quotient construction below, we require the torsion modules to
form a localizing subcategory.  A sufficient condition is that every tail
\(\bigoplus_{j\ge q}\mathcal A_j^i\) of a representable right module be
finitely generated.  This holds for right noetherian algebras and for
algebras generated in degree one.  Indeed, in an extension
of two torsion modules, a tail of a homogeneous cyclic module lies in the
torsion submodule; finite generation of that tail gives a uniform bound
beyond which it vanishes.  Closure under submodules, quotients, and direct
   sums is immediate.
\endgroup

Let \(\mathcal A\) be a connected \(\mathbb Z\)-algebra whose torsion right
modules form a localizing subcategory of its category of right modules.  Denote its category
of right modules by \(\Gr\!\operatorname{-}\mathcal A.\)
Denote its localizing subcategory of torsion modules by
\(\Tors\!\operatorname{-}\mathcal A\) and put
\[
 \QGr\mathcal A
 =\Gr\!\operatorname{-}\mathcal A/
   \Tors\!\operatorname{-}\mathcal A.
\]
The category \(\qprf\!\operatorname{-}\mathcal A\) is the full subcategory of
compact objects in the unbounded derived category
\(\mathcal D(\QGr\mathcal A).\)
If \(\mathcal A\) is {right coherent}, let
\(\gr\!\operatorname{-}\mathcal A\) be its category of finitely presented
right modules.  Writing
\(\tors\!\operatorname{-}\mathcal A
=\gr\!\operatorname{-}\mathcal A\cap
  \Tors\!\operatorname{-}\mathcal A,\) define
\[
 \qgr\mathcal A
 =\gr\!\operatorname{-}\mathcal A/
   \tors\!\operatorname{-}\mathcal A.
\]

{We use the following form of Serre's theorem}
\cite[Section~59]{Serre1955}; see also \cite[Section~2]{ArtinZhang1994}.

\begin{theorem}
\label{thm:serre-qgr}
Let \(X\) be a projective scheme over \(\kk\) such that
\(H^0(X,\mathcal O_X)=\kk,\) let \(L\) be an ample line bundle, and put
\[
        R(X,L)=\bigoplus_{m\ge0}H^0(X,L^m).
\]
Then sheafification induces equivalences
\(
 \qgr R(X,L)\simeq\coh X
 \) and \(
 \QGr R(X,L)\simeq\Qcoh X.
\)
If \(X\) is smooth, it also induces
\(
 \qprf\!\operatorname{-}R(X,L)\simeq\Perf(X)=D^b(\coh X).
\)
\end{theorem}

\begingroup\begingroup
The assertion for perfect complexes follows by passing to compact
objects in the induced equivalence of unbounded derived categories:
\(\mathcal D(\Qcoh X)^c=\Perf(X),\) see
\cite[Theorem~3.1.1]{BondalVanDenBergh2003}.
\endgroup\endgroup

\subsection{Exceptional collections, Euler form, and Euler matrix}
\label{sec:exceptional-euler-matrix}

Let \(\mathcal D=\Perf(\mathsf X),\) where \(\mathsf X\) is a smooth proper
non-commutative scheme.

\begin{definition}\label{def:exceptional-strong}
An object \(E\in\mathcal D\) is \emph{exceptional} if
\(\operatorname{Hom}_{\mathcal D}(E,E[m])=0\) for \(m\ne0,\) and
\(\operatorname{Hom}_{\mathcal D}(E,E)\simeq\kk.\)  An \emph{exceptional
collection} is a sequence \(\mathcal E=(E_1,\ldots,E_s)\) of exceptional
objects satisfying
\[
\operatorname{Hom}_{\mathcal D}(E_i,E_j[m])=0
\qquad\text{for all }i>j\text{ and }m\in\mathbb Z.
\]
An exceptional
collection \(\mathcal E\) is called \emph{full} if
\(\mathcal D=\operatorname{thick}(E_1,\ldots,E_s),\) and
\emph{strong} if
\[
\operatorname{Hom}_{\mathcal D}(E_i,E_j[m])=0
\qquad\text{for all }1\le i,j\le s
\text{ and }m\ne0.
\]
\end{definition}

\begin{example}\label{ex:endomorphism-algebra}
\begingroup\begingroup
\postdisplaypenalty=10000
For a full exceptional collection \(\mathcal E=(E_1,\ldots,E_s)\) in
\(\mathcal D,\) put
\(G=\bigoplus_iE_i\) and
\(\mathscr A_{\mathcal E}=\mathsf{Hom}_{\mathsf X}(G,G).\)
Derived Morita equivalence gives
\[
 \Perf(\mathsf X)\simeq\Perf(\mathscr A_{\mathcal E}),
 \qquad F\longmapsto\mathsf{Hom}_{\mathsf X}(G,F),
\]
see \cite[Theorem~3.8]{Keller2006}.  If \(\mathcal E\) is strong, then
\(\mathscr A_{\mathcal E}\) is quasi-isomorphic to the finite-dimensional
algebra \(\Lambda_{\mathcal E}=\operatorname{End}_{\mathcal D}(G).\)
The object \(E_i\) corresponds to the projective right module
\(e_i\Lambda_{\mathcal E},\) where \(e_i\) is the projector onto \(E_i.\)
\endgroup\endgroup
\end{example}

\begingroup\begingroup
We use the Grothendieck group \(K_0(\mathcal D),\) with its relations
\([B]=[A]+[C]\) for distinguished triangles \(A\to B\to C\to A[1],\)
and write \(K_0(X)=K_0(D^b(\coh X))\) for a smooth projective variety.
\endgroup\endgroup

\begin{definition}\label{def:euler-form}
The \emph{Euler form} of \(\mathcal D\) is the bilinear form on
\(K_0(\mathcal D)\) determined by
\[
 \chi([E],[F])
 =\sum_{i\in\mathbb Z}(-1)^i
   \dim_{\kk}\operatorname{Hom}_{\mathcal D}(E,F[i]).
\]
\end{definition}

We write \(\chi(E,F)=\chi([E],[F])\) for objects of \(\mathcal D.\)  The sum
is finite by properness, and additivity in distinguished triangles makes the
form well-defined on \(K_0(\mathcal D).\)

\begin{lemma}\label{lem:full-collection-k0-basis}
If \(\mathcal E=(E_1,\ldots,E_s)\) is a full exceptional collection{ in
\(\mathcal D\)}, then
\([E_1],\ldots,[E_s]\) is a basis of \(K_0(\mathcal D).\)  In particular,
\(
        K_0(\mathcal D)\simeq\mathbb Z^s.
\)
\end{lemma}

\begin{definition}\label{def:euler-matrix}
Let \(\mathcal E=(E_1,\ldots,E_s)\) be an exceptional collection{ in
\(\mathcal D\)}.  Its
\emph{Euler matrix} is
\[
        A_{\mathcal E}
        :=\bigl(\chi(E_i,E_j)\bigr)_{i,j=1}^s.
\]
\end{definition}

The matrix \(A_{\mathcal E}\) is upper triangular with diagonal entries equal
to \(1.\)  If \(\mathcal E\) is full, it is the matrix of the Euler form in the
basis of Lemma~\ref{lem:full-collection-k0-basis}.  In particular, the Euler
form is non-degenerate.  If \(\mathcal E\) is strong, then
\(\chi(E_i,E_j)=\dim_{\kk}\operatorname{Hom}_{\mathcal D}(E_i,E_j)\ge 0\) for all
\(i<j.\)

\subsection{Mutations, the braid group, helices, and dual {exceptional} collections}
\label{sec:mutations-helices}

\begingroup\begingroup
Throughout this subsection, let \(\mathsf X=\PerfDG{\mathscr R}\) be a smooth proper
non-commutative scheme and put
\(\mathcal D=\Perf(\mathsf X)=H^0(\mathsf X).\)  We identify the objects of
\(\mathcal D\) with those of its DG enhancement \(\mathsf X.\)  For
\(E,F\in\PerfDG{\mathscr R},\) let \(\mathsf{Hom}_{\mathsf X}(E,F)\) denote the
complex of morphisms.  This complex belongs to \(\Perf(\kk),\) and
\(H^m\bigl(\mathsf{Hom}_{\mathsf X}(E,F)\bigr)
\cong\operatorname{Hom}_{\mathcal D}(E,F[m]).\)    For
\(V\in\Perf(\kk),\) we write \(V\otimes E\) for the DG tensor product and
\(V^\vee=\mathsf{Hom}_{\kk}(V,\kk).\)
\endgroup\endgroup

\begin{definition}\label{def:mutations}
For a pair of exceptional objects \((E,F)\) in \(\mathcal D\) with
\(\operatorname{Hom}_{\mathcal D}(F,E[m])=0\) for every \(m\in\mathbb Z,\) define \(L_EF\) and \(R_FE\) by
the distinguished triangles whose middle arrows are the canonical
evaluation and coevaluation morphisms
\[
\begin{gathered}
 L_EF\longrightarrow \mathsf{Hom}_{\mathsf X}(E,F)\otimes E
 \xrightarrow{\mathrm{ev}} F\longrightarrow L_EF[1],
\\
 R_FE[-1]\longrightarrow E\xrightarrow{\mathrm{coev}}
 \mathsf{Hom}_{\mathsf X}(E,F)^{\vee}\otimes F
 \longrightarrow R_FE.
\end{gathered}
\]
{The left and right mutations} of the pair of {exceptional} objects \((E,F)\) are
the pairs of exceptional objects \((L_EF,E)\) and \((F,R_FE),\) respectively.  On the Grothendieck group \(K_0(\mathcal D),\) we have
\[
 [L_EF]=\chi(E,F)[E]-[F],
 \qquad
 [R_FE]=\chi(E,F)[F]-[E].
\]
\end{definition}

For \(s\ge2,\) the braid group on \(s\) strands is
\begin{equation}
\label{def-braid-group}
 B_s=\left\langle \sigma_1,\ldots,\sigma_{s-1}\ \middle|\
 \begin{aligned}
  \sigma_i\sigma_j&=\sigma_j\sigma_i
     && \text{for }1\le i,j\le s-1\text{ and }|i-j|>1,\\
  \sigma_i\sigma_{i+1}\sigma_i
   &=\sigma_{i+1}\sigma_i\sigma_{i+1} && \text{for }1\le i\le s-2
 \end{aligned}
 \right\rangle.
\end{equation}
\begingroup\begingroup
The group \(B_s\) acts on exceptional collections of length $s$ in \(\mathcal D\)
 by
\[
\begin{aligned}
 \sigma_i(\ldots,E_i,E_{i+1},\ldots)
   &=(\ldots,L_{E_i}E_{i+1},E_i,\ldots),\\
 \sigma_i^{-1}(\ldots,E_i,E_{i+1},\ldots)
   &=(\ldots,E_{i+1},R_{E_{i+1}}E_i,\ldots).
\end{aligned}
\]
These operations are inverse and satisfy \eqref{def-braid-group}, see
\cite{BondalKapranov1990}.  Exceptional collections in the same orbit are called
\emph{mutation-equivalent}.
\endgroup\endgroup

\begin{definition}\label{def:geometric-helix}
A \emph{helix} of period \(s\in{\mathbb N}\) in a triangulated category
\(\mathcal D\) is a sequence \((E_i)_{i\in\mathbb Z}\) such that
\((E_i,E_{i+1},\ldots,E_{i+s-1})\) is an exceptional collection
for every \(i\in\mathbb Z\) and
\begin{equation}
\label{eq-mutations-generate-helix}
 E_{i-s}=L_{E_{i-s+1}}\ldots L_{E_{i-1}}E_i,
 \qquad
 E_{i+s}=R_{E_{i+s-1}}\ldots R_{E_{i+1}}E_i.
\end{equation}
The helix is \emph{full} if
\((E_i,E_{i+1},\ldots,E_{i+s-1})\) is full for one, equivalently every,
\(i\in\mathbb Z.\)
The helix is \emph{geometric} if
\[
 \operatorname{Hom}_{\mathcal D}(E_i,E_j[m])=0
 \qquad\text{for all }i\le j
 \text{ and }m\ne0.
\]
\end{definition}

Let \(\mathcal E=(E_1,\ldots,E_s)\) be a full exceptional
collection{ in \(\mathcal D\)}, and set
\[
 F_i=R_{E_s}\ldots R_{E_{i+1}}E_i,
 \qquad
 G_i=L_{E_1}\ldots L_{E_{i-1}}E_i.
\]
{We call \(\mathcal E^\vee=(F_s,\ldots,F_1)\) the
\emph{right dual exceptional collection} of \(\mathcal E,\) and
\({}^\vee\mathcal E=(G_s,\ldots,G_1)\) its
\emph{left dual exceptional collection}.}  If
\(\widehat F_i=F_i[i-s]\) and \(\widehat G_i=G_i[i-1],\) then
\begin{equation}\label{eq:normalized-duality}
 \mathsf{Hom}_{\mathsf X}(\widehat F_i,E_j)
 \simeq
 \mathsf{Hom}_{\mathsf X}(E_j,\widehat G_i)
 \simeq
 \begin{cases}
 \kk,&i=j,\\
 0,&i\ne j.
\end{cases}
\end{equation}
Here \(\kk\) is concentrated in degree zero.
In particular,
\(\chi(\widehat F_i,E_j)=\chi(E_j,\widehat G_i)=\delta_{ij}.\)

\begin{remark}\label{rem:serre-period}
\begingroup\begingroup
With our convention, a full helix of period \(s\) satisfies
\(E_{i+s}\simeq S^{-1}(E_i)[s-1],\) see
\cite{BondalKapranov1990}.  This also follows from
\eqref{eq:normalized-duality} and Serre duality.  On a smooth projective
variety \(X\) of dimension \(s-1,\) where
\(S=(-)\otimes\omega_X[s-1],\) it becomes
\(E_{i+s}\simeq E_i\otimes\omega_X^{-1}.\)
\endgroup\endgroup
\end{remark}

\subsection{The \texorpdfstring{\(\mathbb Z\)-algebra}{Z-algebra} of a helix}
Let \((E_i)_{i\in\mathbb Z}\) be a full geometric helix of period
\(m.\)
Its directed \(\mathbb Z\)-algebra is defined by
\(\mathcal A_i^j=\operatorname{Hom}_{\mathcal D}(E_{-i},E_{-j})\) for \(j\le i,\)
and \(\mathcal A_i^j=0\) otherwise, with multiplication given by composition.  Since the helix is geometric,
\(\operatorname{Hom}_{\mathcal D}(E_{-i},E_{-j}[q])=0\) for \(j\le i\) and
\(q\ne0,\) and hence
\(\chi(E_{-i},E_{-j})=\dim_{\kk}\mathcal A_i^j.\)
\begingroup\begingroup
Put \(T_m=S^{-1}[m-1]\) and choose isomorphisms
\(u_i:T_m^{-1}(E_i)\xrightarrow{\sim}E_{i-m}.\)
The maps
\[
 \mathcal A_i^j\longrightarrow\mathcal A_{i+m}^{j+m},
 \qquad f\longmapsto u_{-j}\circ T_m^{-1}(f)\circ u_{-i}^{-1},
\]
preserve composition and define an \(m\)-periodicity automorphism
of \(\mathcal A.\)
\endgroup\endgroup

Bondal and Polishchuk proved in
\cite[Theorem~4.2]{BondalPolishchuk1994} that the \(\mathbb Z\)-algebra of a
geometric helix is Koszul and co-Koszul and has global dimension equal to the
period.  In particular, it is quadratic and generated in degree one, that
is, by the components
\(\mathcal A_{i+1}^i=\operatorname{Hom}_{\mathcal D}(E_{-i-1},E_{-i})\): for every
\(i<j,\) the composition map
\(\mathcal A_{i+1}^i\otimes\cdots\otimes\mathcal A_j^{j-1}
\longrightarrow\mathcal A_j^i\) is surjective.

\begin{lemma}\label{lem:nonzero-forward-homs}
Let \(\mathcal H=(E_i)_{i\in\mathbb Z}\) be a full geometric helix of period
\(s.\)  Then
\[
        \operatorname{Hom}_{\mathcal D}(E_i,E_j)\ne0
        \qquad\text{for all } \; i\le j.
\]
\end{lemma}

\begin{proof}
The assertion is clear for \(i=j.\)  Suppose that \(i<j\) and
\(\operatorname{Hom}_{\mathcal D}(E_i,E_j)=0.\)  Since \(\mathcal A\) is
generated in degree one, the composition map
\[
 \operatorname{Hom}_{\mathcal D}(E_j,E_k)\otimes
 \operatorname{Hom}_{\mathcal D}(E_i,E_j)
 \longrightarrow\operatorname{Hom}_{\mathcal D}(E_i,E_k)
\]
is surjective for every \(k\ge j,\) and its source is zero.
Hence, \(\operatorname{Hom}_{\mathcal D}(E_i,E_k)=0\) for
all \(k\ge j.\)  Geometricity then gives
\(\operatorname{Hom}_{\mathcal D}(E_i,E_k[m])=0\) for all such \(k\) and
all \(m\in\mathbb Z.\)  The exceptional collection
\((E_j,E_{j+1},\ldots,E_{j+s-1})\) is full, so
\(\operatorname{Hom}_{\mathcal D}(E_i,-[m])\) vanishes on \(\mathcal D\)
for every \(m\in\mathbb Z.\)  In particular,
\(\operatorname{Hom}_{\mathcal D}(E_i,E_i)=0,\) contrary to the
exceptionality of \(E_i.\)
\end{proof}

{Van den Bergh proves the noetherian case of the following result in
\cite[Lemma~3.5]{VanDenBergh2011}.  Here we assume instead that the algebra
is generated in degree one.}

\begin{lemma}\label{lem:periodic-z-subalgebra}
Let \(\mathcal A\) be a connected \(\mathbb Z\)-algebra
generated in degree one.  Fix
\(m\ge1\) and \(r\in\mathbb Z,\) put \(J=r+m\mathbb Z,\) and let
\(\mathcal B\) be the full subalgebra of \(\mathcal A\) on the indices in
\(J.\)  If \(e_i\) is the identity at the object \(i,\) write formally
\(e=\sum_{i\in J}e_i,\) so that \(\mathcal B=e\mathcal A e.\)  The functors
\[
 \operatorname{Res}\colon
 \Gr\!\operatorname{-}\mathcal A\longrightarrow
 \Gr\!\operatorname{-}\mathcal B,
 \; M\longmapsto Me=\bigoplus_{i\in J}Me_i,
\quad
\text{and}
\quad
 \operatorname{Ind}\colon
 \Gr\!\operatorname{-}\mathcal B\longrightarrow
 \Gr\!\operatorname{-}\mathcal A,
 \; N\longmapsto N\otimes_{\mathcal B}e\mathcal A,
\]
descend to mutually inverse exact equivalences
\(\QGr\mathcal A\simeq\QGr\mathcal B.\)  The induced equivalence of derived
categories restricts to an equivalence
\(
 \qprf\!\operatorname{-}\mathcal A
 \simeq
 \qprf\!\operatorname{-}\mathcal B.
\)
\end{lemma}

\begin{proof}
The algebra \(\mathcal B,\) after reindexing, is also generated in
degree one.  Thus, both quotient categories are defined.
Suppose first that \(Me=0.\)  If \(x\in M e_i,\) choose
\(\ell\in r+m\mathbb Z\) such that \(\ell-m<i\le\ell.\)
For every \(j\ge\ell,\) the multiplication map
\(
 \mathcal A_\ell^i\otimes\mathcal A_j^\ell
 \rightarrow\mathcal A_j^i
\)
is surjective because \(\mathcal A\) is generated in degree one.  Since
\(M e_\ell=0,\) it follows that
\(x\mathcal A_j^i=0\) for all \(j\ge\ell.\)  Thus, \(M\) is torsion.

\begingroup\begingroup
Restriction sends torsion \(\mathcal A\)-modules to torsion
\(\mathcal B\)-modules. To prove the corresponding assertion for
induction, let \(N\) be a torsion right \(\mathcal B\)-module and let
\(n\in Ne_i\) be homogeneous. Choose \(q\in J,\) with \(q\ge i,\)
such that \(n\mathcal B_q^i=0.\) For \(j\ge q,\) factor
\(\mathcal A_j^i=\mathcal A_q^i\mathcal A_j^q.\)
The tensor relation then gives
\((n\otimes e_i)\mathcal A_j^i=0.\)
The elements \(n\otimes e_i\) generate
\(N\otimes_{\mathcal B}e\mathcal A\) as a right
\(\mathcal A\)-module; hence, this induced module is torsion.
\endgroup\endgroup
Moreover,
\(
 \bigl(N\otimes_{\mathcal B}e\mathcal A\bigr)e\simeq N.
\)
For an \(\mathcal A\)-module \(M,\) the kernel and cokernel of the
adjunction morphism
\(
        (Me)\otimes_{\mathcal B}e\mathcal A\rightarrow M
\)
have zero restriction to \(\mathcal B,\) and hence are torsion.  If
\(f\colon N'\hookrightarrow N\) is a monomorphism and
\(K=\ker(\operatorname{Ind}(f)),\) then exactness of restriction and the
natural isomorphism
\(\bigl(N'\otimes_{\mathcal B}e\mathcal A\bigr)e\simeq N'\) give
\(Ke\simeq\ker(f)=0.\)  Thus, \(K\) is torsion.
Since induction is right exact, it therefore becomes exact on the quotient
categories.  Restriction and induction induce mutually inverse
exact equivalences
\(\QGr\mathcal A\simeq\QGr\mathcal B.\)
The induced equivalence of derived categories preserves arbitrary direct
sums and {therefore compact objects, which gives the second equivalence.}
\end{proof}

\section{Non-commutative Fano varieties and minifolds}
\label{sec:minifolds}

\subsection{Non-commutative Fano varieties}

Let \(\mathsf X\) be a smooth proper non-commutative scheme and put
\(\mathcal D=\Perf(\mathsf X).\)  Let \(S\) be the Serre functor
of \(\mathcal D.\)

\begin{definition}\label{def:orlov-fano-variety}
Let $n\geq 0$ be an integer.
An \emph{\(n\)-dimensional non-commutative Fano variety} is a smooth proper
non-commutative scheme \(\mathsf X\) such that
\begin{enumerate}[label=\textup{(\roman*)}]
\item there exists an exceptional object \(E\in\Perf(\mathsf X)\) satisfying
\begin{equation}\label{eq:fano-vanishing}
 \operatorname{Hom}_{\mathcal D}(E,T^pE[k])=0
 \quad\text{for }p\ge0\text{ and }k\ne0,
 \quad\text{where}\quad T=S^{-1}[n];
\end{equation}
\item For the graded algebra
\(R:=R(\mathsf X,E)
 =\bigoplus_{p\ge0}\operatorname{Hom}_{\mathcal D}(E,T^pE),\) with
multiplication \(f\cdot g=T^q(f)\circ g\) for homogeneous elements \(f\) and
\(g\) of degrees \(p\) and \(q,\) respectively, there is an equivalence
\(\Phi_E\colon\mathcal D(\mathsf X)
 \xrightarrow{\ \sim\ }\mathcal D(\QGr R)\)
whose restriction to compact objects is
\begin{equation}\label{eq:fano-equivalence}
 \Phi_E|_{\mathcal D}\colon\mathcal D=\Perf(\mathsf X)
 \xrightarrow{\ \sim\ }\qprf\!\operatorname{-}R,\qquad
 F\longmapsto q_R\!\left(\bigoplus_{p\ge0}
 \mathsf{Hom}_{\mathsf X}(E,T^pF)\right).
\end{equation}
Here \(q_R\colon\mathcal D(\Gr\!\operatorname{-}R)
\to\mathcal D(\QGr R)\) is the derived quotient functor.
\end{enumerate}
\end{definition}

\begin{remark}\label{rem:fano-serre-dimension}
The integer \(n\) is called the \emph{Serre dimension} of \(\mathsf X.\)
This terminology agrees with the Serre dimension introduced in
\cite[Definition~5.3 and Proposition~5.5]{ElaginLunts2021}.
\end{remark}

Definition~\ref{def:orlov-fano-variety} is a more detailed version
of the definition given in \cite{Orlov2026}.
\begingroup\begingroup
Condition~\textup{(ii)} implies
\(\operatorname{thick}\{T^aE\mid a\in\mathbb Z\}=\mathcal D.\)
Use the grading convention \(R(a)_q=R_{a+q}.\)
For every \(a\in\mathbb Z,\) condition~\textup{(i)} identifies the
graded complex defining \(\Phi_E(T^aE)\) in
\eqref{eq:fano-equivalence} with \(R(a)\) in all sufficiently large
internal degrees, up to quasi-isomorphism. The discarded components
have torsion cohomology, so \(\Phi_E(T^aE)\simeq q_R(R(a)).\)
The modules \(R(a),\) for \(a\in\mathbb Z,\) generate
\(\mathcal D(\Gr\!\operatorname{-}R)\) as a localizing subcategory;
their images therefore generate \(\mathcal D(\QGr R)\) as a
localizing subcategory.
These images are compact by~\textup{(ii)} and hence generate
\(\qprf\!\operatorname{-}R\) as a thick subcategory.
\endgroup\endgroup

\begin{definition}\label{def:geometric-fano}
\begingroup\begingroup
A non-commutative Fano variety \(\mathsf X\) is \emph{of geometric type}
if there exist an exceptional object \(E\) satisfying conditions
\textup{(i)} and \textup{(ii)} of Definition~\ref{def:orlov-fano-variety}
and a polynomial
\(h\in\mathbb Q[x]\) such that
\(\dim_\kk R(\mathsf X,E)_p=h(p)\) for every integer \(p\ge0.\)
The degree of \(h\) is the \emph{Hilbert dimension} with respect to \(E.\)
\endgroup\endgroup
\end{definition}

{Example~\ref{ex:three-kronecker} shows that polynomial growth
is not automatic.}

\begin{example}\label{ex:three-kronecker}
For \(N\ge3,\) let \(Q_N\) be the quiver with two vertices and \(N\) arrows
from \(0\) to \(1,\) and put \(\Lambda_N=\kk Q_N.\)
\begingroup\begingroup
The scheme \(\mathsf X_N=\PerfDG{\Lambda_N}\) is smooth and proper.
Let \((P_1,P_0)\) be its indecomposable projective right
\(\Lambda_N\)-modules, ordered so that
\(\operatorname{Hom}_{\Lambda_N}(P_1,P_0)\simeq\kk^N,\) and put
\(T=S_{\mathsf X_N}^{-1}[1].\) The pair
\(\mathcal E_N=(P_1,P_0)\) is a full strong exceptional collection in
\(\Perf(\mathsf X_N).\) In the basis \(([P_1],[P_0])\) of
\(K_0(\Perf(\mathsf X_N)),\) the Euler matrix and the matrix of
\(T_*\) are
\endgroup\endgroup
\[
 A_N=\begin{pmatrix}1&N\\0&1\end{pmatrix},
 \qquad
 T_*=-A_N^{-t}A_N
 =\begin{pmatrix}-1&-N\\N&N^2-1\end{pmatrix}.
\]
By \cite[Section~2]{Minamoto2008}, the scheme \(\mathsf X_N\) is a
one-dimensional non-commutative Fano variety; see also
\cite[Proposition~2.5]{Polishchuk2005}. If
\(h_N(p)=\dim_{\kk}\operatorname{Hom}(P_1,T^pP_1),\) then
\(h_N(0)=1,\) \(h_N(1)=N^2-1,\) and
\begin{equation}\label{eq:kronecker-hilbert-recurrence}
 h_N(p+2)=(N^2-2)h_N(p+1)-h_N(p),\qquad p\ge0.
\end{equation}
\begingroup
The characteristic roots are \(\lambda,\lambda^{-1},\) with
\(\lambda>1,\) so \(h_N\) grows exponentially.  More generally, for an
exceptional object \(E\) satisfying \eqref{eq:fano-vanishing}, the sequence
\(h_E(p)=\chi(E,T^pE)\) satisfies the recurrence
\eqref{eq:kronecker-hilbert-recurrence} and has
\(h_E(0)=1.\)  Its coefficient of \(\lambda^p\) cannot vanish: otherwise
this nonnegative integer sequence would tend to zero and hence, be
eventually zero, contrary to the recurrence and \(h_E(0)=1.\)
That coefficient is therefore positive.  {Thus, no exceptional object
satisfying Definition~\ref{def:orlov-fano-variety} has polynomial growth,}
and \(\mathsf X_N\) is not of geometric type.
\endgroup
\end{example}

\begin{proposition}\label{prop:commutative-fano-is-geometric}
Let \(X\) be an \(n\)\!-dimensional smooth Fano variety over
\(\kk.\)  Then  \(\PerfDG{X}\) is an \(n\)\!-dimensional
non-commutative Fano variety of geometric type.
\end{proposition}

\begin{proof}[Proof]
Put \(\mathcal D=D^b(\coh X)=\Perf(X).\)  Its Serre functor is
\[
        S(F)=F\otimes\omega_X[n],
        \qquad
        T=S^{-1}[n]=(-)\otimes\omega_X^{-1}.
\]
For \(E=\mathcal O_X,\) the algebra \(R(X,E)\) is the anticanonical section {algebra}
\(
        R(X,\mathcal O_X)
        =\bigoplus_{p\ge0}H^0(X,\omega_X^{-p}).
\)
Kodaira vanishing gives \eqref{eq:fano-vanishing}.  Since
\(\omega_X^{-1}\) is ample, {this algebra is finitely
generated}.  Theorem~\ref{thm:serre-qgr} gives
\[
 \QGr R(X,\mathcal O_X)\simeq\Qcoh X,
 \qquad
 \mathcal D\bigl(\QGr R(X,\mathcal O_X)\bigr)
 \simeq\mathcal D(\Qcoh X),
\]
and, on compact objects,
\(\qprf\!\operatorname{-}R(X,\mathcal O_X)\simeq\Perf(X)=\mathcal D.\)
Finally, Riemann--Roch and Kodaira vanishing show that
\(h_{\mathcal O_X}(p)=\chi(X,\omega_X^{-p})\) is a polynomial in \(p.\)
\end{proof}

\subsection{Minifolds}

In this subsection, we define non-commutative minifolds.

\begin{definition}\label{def:orlov-minifold}
Let \(n\geq0.\)  A smooth proper non-commutative scheme
\(\mathsf X\) is called a \emph{non-commutative minifold of dimension \(n\)} if
\(\mathcal D=\Perf(\mathsf X)\) admits a full exceptional collection \(
        \mathcal E=(E_1,\ldots,E_{n+1})
\)
such that the helix generated by \(\mathcal E\) is geometric,
and \(T=S^{-1}[n]\) acts maximally unipotently on
\(K_0(\mathcal D)\otimes\mathbb Q,\) that is, it has a single Jordan block
with eigenvalue \(1.\)
\end{definition}

\begin{proposition}\label{prop:geometric-collection-fano}
Let \(\mathsf X\) be a smooth proper non-commutative scheme such that
\(\mathcal D=\Perf(\mathsf X)\) admits a full exceptional collection
\(
        \mathcal E=(E_1,\ldots,E_{n+1})
\)
with
a geometric helix.
Then \(\mathsf X\) is an \(n\)-dimensional non-commutative Fano variety,
and every \(E_i\) can be used as the exceptional object in
Definition~\ref{def:orlov-fano-variety}.
{Moreover, if \(T=S^{-1}[n]\) acts unipotently on
\(K_0(\mathcal D)\otimes\mathbb Q,\) then \(\mathsf X\) is of geometric type.}
\end{proposition}

\begin{proof}
Put \(m=n+1,\) and let
\(\mathcal H=(E_i)_{i\in\mathbb Z}\) be the full geometric helix generated
by \(\mathcal E.\)  By Remark~\ref{rem:serre-period},
\begin{equation}\label{eq:geometric-fano-period}
        E_{i+m}\simeq T(E_i).
\end{equation}
Fix \(r\in\mathbb Z,\) put \(E=E_r,\) and denote by \(\mathcal A\) the
directed \(\mathbb Z\)-algebra of \(\mathcal H.\)  For \(p\ge0,\)
geometricity and \eqref{eq:geometric-fano-period} give
\[
 \operatorname{Hom}_{\mathcal D}(E,T^pE[k])
 =
 \operatorname{Hom}_{\mathcal D}(E_r,E_{r+pm}[k])
 =0
 \qquad\text{for }k\ne0.
\]
Thus, \(E\) satisfies \eqref{eq:fano-vanishing}.
By \cite[Theorem~4.2]{BondalPolishchuk1994}, the algebra
\(\mathcal A\) is a geometric \(\mathbb Z\)-algebra of period \(m.\)
\begingroup
Put
\[
        G=\bigoplus_{j=r}^{r+m-1}E_j,
        \qquad B=\operatorname{End}_{\mathcal D}(G).
\]
The block \((E_r,\ldots,E_{r+m-1})\) is full and strong.  Hence, its DG
endomorphism algebra is quasi-isomorphic to \(B,\) and derived Morita
equivalence \cite[Theorem~3.8]{Keller2006} gives
\[
        \mathcal D(\mathsf X)\simeq\mathcal D(B),
        \qquad \mathcal D\simeq\Perf(B).
\]
Under the definition of \(\mathcal A,\) this algebra is
\(B=\bigoplus_{-r-m<i,j\le-r}\mathcal A_i^j.\)  {This is the finite
block algebra in \cite[Theorem~8.14]{EfimovLuntsOrlov2011}, which gives}
\(\mathcal D(\QGr\mathcal A)\simeq\mathcal D(B).\)  It preserves
coproducts and therefore restricts to compact objects; {this restriction}
is also stated in
\cite[Proposition~8.21]{EfimovLuntsOrlov2011}.  Composing these equivalences gives
\endgroup
\begin{equation}\label{eq:helix-equivalence}
  \mathcal D(\mathsf X) \simeq\mathcal D(\QGr \mathcal A),\qquad      \mathcal D
        \simeq
        \qprf\!\operatorname{-}\mathcal A.
\end{equation}

Let \(\mathcal B\) be the full subalgebra of \(\mathcal A\) on the
indices in \(-r+m\mathbb Z.\)
\begingroup\begingroup
After reindexing \(-r+pm\) by \(p,\) compatible period isomorphisms
\(E_{r-pm}\simeq T^{-p}E\) identify \(\mathcal B\) with the
\(\mathbb Z\)-algebra of
\endgroup\endgroup
\[
        R=R(\mathsf X,E)
        =\bigoplus_{p\ge0}\operatorname{Hom}_{\mathcal D}(E,T^pE).
\]
{Indeed, for \(q\le p,\) applying \(T^p\) to morphisms gives}
\[
 \mathcal B_p^q
 =\mathcal A_{-r+pm}^{-r+qm}
 =\operatorname{Hom}_{\mathcal D}
 \bigl(E_{r-pm},E_{r-qm}\bigr)
 \simeq
 \operatorname{Hom}_{\mathcal D}\bigl(E,T^{p-q}E\bigr)
 =R_{p-q}.
\]
{These identifications preserve multiplication.}
Lemma~\ref{lem:periodic-z-subalgebra} and
\eqref{eq:helix-equivalence} now give
\(
        \mathcal D
        \simeq
        \qprf\!\operatorname{-}\mathcal A
        \simeq
        \qprf\!\operatorname{-}R.
\)
\begingroup\begingroup
The first equivalence in \eqref{eq:helix-equivalence} is induced by DG
Yoneda and the derived quotient functor. Here we use derived change of
scalars along a chain of quasi-isomorphisms connecting the directed DG
endomorphism algebra of the helix to \(\mathcal A,\) which exists by
geometricity. For \(F\in\mathcal D,\)
restriction to \(-r+m\mathbb Z,\) with internal degree \(p\)
corresponding to the object \(-r+pm,\) gives
\(N_F\in\mathcal D(\Gr\!\operatorname{-}R)\) with
\[
 (N_F)_p\simeq\mathsf{Hom}_{\mathsf X}(E_{r-pm},F)
 \simeq\mathsf{Hom}_{\mathsf X}(E,T^pF)
 \qquad\text{in }\mathcal D(\kk).
\]
The period identifications respect composition and give the right
\(R\)-action in Definition~\ref{def:orlov-fano-variety}\textup{(ii)}.
Since \(N_F/(N_F)_{\ge0}\) has torsion cohomology, the composite
equivalence sends \(F\) to \(q_R((N_F)_{\ge0}),\) as in
\eqref{eq:fano-equivalence}. Thus, its restriction to compact objects
is \(\Phi_E|_{\mathcal D},\) proving the Fano condition.
\endgroup\endgroup

{Suppose now that \(T\) acts unipotently.}  Write
\(
        T_*=I+N
\)
on \(K_0(\mathcal D)\otimes\mathbb Q.\)  Fullness of \(\mathcal E\) gives
\(\dim_{\mathbb Q}K_0(\mathcal D)\otimes\mathbb Q=n+1,\) and hence
\(N^{n+1}=0.\)  The vanishing
\eqref{eq:fano-vanishing} yields
\[
 \dim_{\kk}\operatorname{Hom}_{\mathcal D}(E,T^pE)
 =\chi(E,T^pE)
 =\sum_{j=0}^{n}\binom pj\chi\bigl([E],N^j[E]\bigr)
\]
for every \(p\ge0.\)
Thus, the Hilbert function of \(R\) is polynomial in \(p,\) so
\(\mathsf X\) is of geometric type.
\end{proof}

\begin{proposition}\label{prop:minifold-hilbert-degree}
Let \(\mathsf X\) be an \(n\)-dimensional non-commutative minifold, and let
\(\mathcal H=(E_i)_{i\in\mathbb Z}\) be the helix generated by
{a full exceptional collection \(\mathcal E=(E_1,\ldots,E_{n+1})\)
in \(\mathcal D=\Perf(\mathsf X)\)}
as in Definition~\ref{def:orlov-minifold}.  Then, for every \(i,\) the
Hilbert function
\[
 h_{E_i}(p)
 =\dim_{\kk}\operatorname{Hom}_{\mathcal D}(E_i,T^pE_i),
 \qquad p\ge0,
\]
is a polynomial of degree \(n\) with positive leading coefficient.
For every exceptional object \(E\) satisfying conditions
\textup{(i)} and \textup{(ii)} of Definition~\ref{def:orlov-fano-variety},
its Hilbert function is also a
polynomial of degree \(n\) with positive leading coefficient.  Consequently, the Hilbert dimension
of \(\mathsf X\) is equal to \(n\) and is independent of the choice of \(E.\)
\end{proposition}

\begin{proof}
\begingroup
Put \(V=K_0(\mathcal D)\otimes\mathbb Q\) and \(N=T_*-I.\)
Maximal unipotence gives \(N^{n+1}=0\) and \(N^n\ne0.\)  For
\(F,G\in\mathcal D,\) the binomial expansion gives
\begin{equation}\label{eq:minifold-hilbert-leading-term}
 \chi(F,T^pG)=\sum_{k=0}^{n}\binom pk
                 \chi\bigl([F],N^k[G]\bigr),\qquad p\ge0.
\end{equation}
Let \(E\) be an exceptional object satisfying conditions
\textup{(i)} and \textup{(ii)} of Definition~\ref{def:orlov-fano-variety}.
By \eqref{eq:fano-vanishing}, \(h_E(p)=\chi(E,T^pE)\) is a polynomial
of degree \(d\le n,\) and \(h_E(0)=1.\)
For fixed \(F,G\in\mathcal D,\) put
\[
 b_p(F,G)=\sum_k\dim_\kk\operatorname{Hom}_{\mathcal D}(F,T^pG[k]).
\]
For fixed integers \(a,b\) and sufficiently large \(p,\)
the vanishing \eqref{eq:fano-vanishing} gives
\[
 b_p(T^aE,T^bE)=h_E(p+b-a)=O(p^d).
\]
Long exact sequences show that the estimate \(b_p(F,G)=O(p^d)\)
is preserved by finite sums, shifts, cones, and direct summands in
either argument.  {Condition~\textup{(ii)} of
Definition~\ref{def:orlov-fano-variety} gives}
\(\operatorname{thick}\{T^aE\mid a\in\mathbb Z\}=\mathcal D.\)
Extending the estimate first in the second argument and then in the
first therefore proves it for every fixed pair \(F,G\in\mathcal D.\)

The Euler form is non-degenerate, and the classes of
{the full exceptional collection \(\mathcal E\)} form a basis of \(V.\)
Since \(N^n\ne0,\) some pair \(i,j\) satisfies
\(\chi([E_i],N^n[E_j])\ne0.\)  Thus, \(\chi(E_i,T^pE_j)\) has degree
\(n\) by \eqref{eq:minifold-hilbert-leading-term}, while
\[
 |\chi(E_i,T^pE_j)|\le b_p(E_i,E_j)=O(p^d).
\]
Hence, \(n\le d,\) so \(d=n.\)  Its leading coefficient is positive
because \(h_E(p)\ge0\) for \(p\ge0.\)
By Proposition~\ref{prop:geometric-collection-fano}, every term of
\(\mathcal H\) {satisfies the conditions imposed on \(E.\)}
\endgroup
\end{proof}

\subsection{Commutative minifolds}
\label{subsec:commutative-minifolds}

In this subsection, we recall the definition and examples of commutative minifolds.

\begin{definition}\label{def:usual-minifold}
An
\emph{\(n\)-dimensional minifold} is a smooth projective variety \(X\) of
dimension \(n\) such that \(D^b(\coh X)\) admits a full exceptional
collection of length \(n+1.\)
\end{definition}
\begin{definition}\label{def:strong-minifold}
{An \(n\)-dimensional minifold \(X\) is called \emph{strong} if
\(D^b(\coh X)\) admits a full exceptional collection of length \(n+1\)
whose helix is geometric.}
\end{definition}

If \(X\) is a smooth projective variety of dimension \(n,\) then
\(\dim_{\mathbb Q}(K_0(X)\otimes\mathbb Q)\ge n+1.\)  Indeed, the Chern
character identifies this space with the rational Chow ring, in which the
powers \(1,H,\ldots,H^n\) of an ample class are non-zero and belong to
distinct graded pieces.

\begin{theorem}
\label{thm:minifold-strong-geometric}
{Let \(X\) be an \(n\)-dimensional minifold, and let}
\(\mathcal E=(E_1,\ldots,E_{n+1})\) be a full exceptional collection in
\(D^b(\coh X).\)  Then the helix generated by
\(\mathcal E\) is geometric if and only if the {exceptional} collection \(\mathcal E\) is
strong.  If these conditions hold, there are vector bundles
\(V_1,\ldots,V_{n+1}\) on \(X\) and an integer \(q\) such that
\[
        E_i\simeq V_i[q]
        \qquad\text{for }1\le i\le n+1.
\]
If \(\mathcal E\) is strong, \(X\) is a Fano variety.
\end{theorem}

\begin{proof}
{If the helix generated by \(\mathcal E\) is geometric, then
\(\operatorname{Hom}_{D^b(\coh X)}(E_i,E_j[m])=0\) for \(m\ne0\): for
\(i\le j\) this is geometricity, and for \(i>j\) it is exceptionality.
Hence, \(\mathcal E\) is strong.}

Conversely, suppose that \(\mathcal E\) is strong.  Since
\(\mathcal E\) is full,
its classes form a basis of \(K_0(X),\) so
\(\operatorname{rk}K_0(X)=n+1.\)  By
\cite[Theorem and its proof]{Positselski1995}, there are vector bundles
\(V_i\) and an integer \(q\) such that \(E_i\simeq V_i[q];\) moreover, the
helix is geometric and \(X\) is Fano.
\end{proof}

By Theorem~\ref{thm:minifold-strong-geometric}, the condition in
Definition~\ref{def:strong-minifold} is equivalent to the existence of a
full strong exceptional collection of length \(n+1\)
{in \(D^b(\coh X)\)}.  It is also
equivalent to the existence of a full exceptional collection of \(n+1\)
coherent sheaves {in \(D^b(\coh X),\) see}
\cite[Proposition~3.3]{BondalPolishchuk1994}.

\begin{question}\label{ques:all-minifolds-strong}
{For every \(n\)-dimensional minifold \(X,\) does \(D^b(\coh X)\)
admit a full exceptional collection \(\mathcal E=(E_1,\ldots,E_{n+1})\)
such that the helix generated by \(\mathcal E\) is geometric?}
\end{question}

An affirmative answer would imply, by
Theorem~\ref{thm:minifold-strong-geometric}, that every minifold is Fano.
It is not known in arbitrary dimension whether every minifold is
Fano.  For complex minifolds,
Question~\ref{ques:all-minifolds-strong} has an affirmative answer in
dimensions at most three.  In dimension four it is known under the
additional assumption that the minifold is Fano
\cite[Theorem~1.1(3)]{GKMS2013}.

\begin{proposition}\label{prop:strong-is-noncommutative-minifold}
Let \(X\) be a strong \(n\)-dimensional minifold.  Then
\(\PerfDG{X}\) is an \(n\)-dimensional non-commutative minifold.  Moreover,
\(T=(-)\otimes\omega_X^{-1}\) acts maximally unipotently on
\(K_0(X)\otimes\mathbb Q.\)
\end{proposition}

\begin{proof}
Choose a full exceptional collection
{\(\mathcal E=(E_1,\ldots,E_{n+1})\) in \(D^b(\coh X)\)}
with a geometric helix as in
Definition~\ref{def:strong-minifold}.  Theorem~\ref{thm:minifold-strong-geometric}
shows that \(X\) is Fano.  It remains to verify maximal unipotence.  {The
classes of the exceptional collection \(\mathcal E\) give}
\(K_0(X)\simeq\mathbb Z^{n+1}.\)  Put
\(H=c_1(\omega_X^{-1}).\)  The Chern character gives an
isomorphism
\(
 \operatorname{ch}\colon K_0(X)\otimes\mathbb Q
 \xrightarrow{\sim}\mathrm{CH}^*(X)_{\mathbb Q}
\)
and identifies the action of \(T\) with multiplication by \(e^H.\)  Thus
for every \(\alpha\in K_0(X)\otimes\mathbb Q,\) we have
\[
 \operatorname{ch}\bigl((T-I)^{n+1}\alpha\bigr)=0,
 \qquad
 \operatorname{ch}\bigl((T-I)^n[\mathcal O_X]\bigr)=H^n\ne0.
\]
The nilpotency index of \(T-I\) is therefore \(n+1.\)  Since
\(K_0(X)\otimes\mathbb Q\) has dimension \(n+1,\) the operator \(T\) has
one unipotent Jordan block.
\end{proof}

{Beilinson's exceptional collection in \(D^b(\coh\mathbb P^n)\)
shows that projective spaces are strong minifolds}
over \(\kk\) \cite{Beilinson1978}.  {If
\(Q\) is a smooth odd-dimensional
quadric over \(\kk,\) Kapranov's exceptional collection in \(D^b(\coh Q)\) shows
that \(Q\) is a strong minifold}
\cite{Kapranov1988}.  By
Proposition~\ref{prop:strong-is-noncommutative-minifold}, the DG categories
of projective spaces and these quadrics are non-commutative minifolds.
{A smooth}
even-dimensional quadric {\(Q\) over \(\kk\)} has two spinor objects in a full exceptional
collection {in \(D^b(\coh Q)\)} and its Grothendieck group has rank
\(\dim Q+2;\) hence, it is not a minifold.

{The classification of complex three-dimensional minifolds appears in
\cite[Theorem~1.1(2)]{GKMS2013}.}

\begin{theorem}
\label{thm:threefold-minifolds}
Let \(X\) be a complex three-dimensional minifold.  Then
\(X\) is isomorphic to \(\mathbb P^3,\) \(Q^3,\) \(V_5,\) or a member of
the family \(X_{22}.\)
\end{theorem}

The variety \(Q^3\subset\mathbb P^4\) is a smooth quadric
hypersurface.
The variety \(V_5\) is a smooth codimension-three linear section
of \(\operatorname{Gr}(2,5),\) i.e.
\(
        V_5=\operatorname{Gr}(2,5)\cap\mathbb P^6
        \subset\mathbb P^9.
\)
If \(H\) is the
restriction of the Pl\"ucker hyperplane class, then
\(\operatorname{Pic}(V_5)=\mathbb ZH,\) \(-K_{V_5}=2H,\) and \(H^3=5.\)

The varieties \(X_{22}\) can be constructed as follows; see
\cite[Section~4]{Kuznetsov2009}.  Let \(W\) be a seven-dimensional vector
space and let
\(\mathcal U\) be the tautological rank-three subbundle on
\(\operatorname{Gr}(3,W).\)  Such a threefold is the smooth zero locus of a
regular section of \((\Lambda^2\mathcal U^\vee)^{\oplus3}.\)
Equivalently,
for a three-dimensional subspace \(B\subset\Lambda^2W^\vee,\) it
parametrizes the three-dimensional subspaces of \(W\) isotropic for every
form in \(B.\)  These threefolds form a six-dimensional family and satisfy
\(\operatorname{Pic}(X_{22})=\mathbb Z\cdot(-K_{X_{22}})\) and
\((-K_{X_{22}})^3=22.\)

{We use the following exceptional collections and denote them by
\(\mathcal E_X.\) The exceptional collections
\(\mathcal E_{\mathbb P^3},\mathcal E_{Q^3},\mathcal E_{V_5}\)
are given, respectively, in}
\cite{Beilinson1978,Kapranov1988,Orlov1991}.  Kuznetsov constructed the
four-term exceptional collection {\(\mathcal E_{X_{22}},\) see} \cite{Kuznetsov1996}; its
fullness follows from \cite[Theorem~4.1]{Kuznetsov2009}.
{By \cite[Proposition~3.3]{BondalPolishchuk1994}, each of these
four-term {exceptional} collections generates a geometric helix.  Together with
Theorem~\ref{thm:threefold-minifolds}, this shows that every
complex three-dimensional minifold is strong.}

{For each of these four exceptional collections \(\mathcal E_X,\)
let \(A_X\) be its Euler matrix and let \(A_X^\vee\) be the Euler matrix
of its right dual exceptional collection \(\mathcal E_X^\vee.\)}

{\small
\[
\begin{gathered}
\mathbb P^3:\
(\mathcal O,\mathcal O(1),\mathcal O(2),\mathcal O(3))\\[0.4em]
A_{\mathbb P^3}=
\begin{pmatrix}
1&4&10&20\\0&1&4&10\\0&0&1&4\\0&0&0&1
\end{pmatrix}
\qquad
A_{\mathbb P^3}^\vee=
\begin{pmatrix}
1&4&6&4\\0&1&4&6\\0&0&1&4\\0&0&0&1
\end{pmatrix}
\end{gathered}
\]

\[
\begin{gathered}
Q^3:\
(\mathcal O,\mathcal S_Q^\vee,\mathcal O(1),\mathcal O(2))\\[0.4em]
A_{Q^3}=
\begin{pmatrix}
1&4&5&14\\0&1&4&16\\0&0&1&5\\0&0&0&1
\end{pmatrix}
\qquad
A_{Q^3}^\vee=
\begin{pmatrix}
1&5&4&5\\0&1&4&11\\0&0&1&4\\0&0&0&1
\end{pmatrix}
\end{gathered}
\]

\[
\begin{gathered}
V_5:\
(\mathcal O,\mathcal Q_5,\mathcal S_5^\vee,\mathcal O(1))\\[0.4em]
A_{V_5}=
\begin{pmatrix}
1&5&5&7\\0&1&3&10\\0&0&1&5\\0&0&0&1
\end{pmatrix}
\qquad
A_{V_5}^\vee=
\begin{pmatrix}
1&5&5&7\\0&1&3&10\\0&0&1&5\\0&0&0&1
\end{pmatrix}
\end{gathered}
\]
\[
\begin{gathered}
X_{22}:\
(\mathcal O,\mathcal S_{22}^\vee,\mathcal E_{22}^\vee,
  \Lambda^2\mathcal S_{22}^\vee)\\[0.4em]
A_{X_{22}}=
\begin{pmatrix}
1&7&8&18\\0&1&4&13\\0&0&1&4\\0&0&0&1
\end{pmatrix}
\qquad
A_{X_{22}}^\vee=
\begin{pmatrix}
1&4&3&7\\0&1&4&20\\0&0&1&7\\0&0&0&1
\end{pmatrix}
\end{gathered}
\]
}
\noindent The eight matrices \(A_X,A_X^\vee,\) for
\(X\in\{\mathbb P^3,Q^3,V_5,X_{22}\},\) also appear in
\cite[Sections~3.1 and~4]{Golyshev2008}.
Here \(\mathcal S_Q\) is the spinor bundle on \(Q^3,\)
\(\mathcal S_5\) and \(\mathcal Q_5\) are the restrictions of the
tautological subbundle and quotient bundle on \(\operatorname{Gr}(2,5),\)
and \(\mathcal S_{22},\) \(\mathcal E_{22}\) are the rank-three and
rank-two bundles in Kuznetsov's {exceptional collection \(\mathcal E_{X_{22}}\)}.

{Bondal and Polishchuk conjectured that the group
\(B_s\ltimes\mathbb Z^s,\) acting by mutations and independent shifts of
the objects, acts transitively on the full exceptional collections of
length \(s\) in a fixed triangulated category
\cite[Conjecture~2.2]{BondalPolishchuk1994}.}  {For each threefold \(X\) listed in
Theorem~\ref{thm:threefold-minifolds}, Nogin \cite{Nogin1994} proved
transitivity on semiorthonormal bases of \(K_0(X)\) under mutations and
independent changes of signs.}  {For each threefold \(X\)
listed in Theorem~\ref{thm:threefold-minifolds}},
\cite[Theorem~1.1]{Polishchuk2011} proves
that \(B_4\) acts transitively on full exceptional collections of vector
bundles {in \(D^b(\coh X)\)}.

\begin{theorem}
\label{thm:polishchuk-transitivity}
Let \(X\) be a complex three-dimensional minifold.  Then \(B_4\) acts
transitively by mutations on the full exceptional collections of vector
bundles in \(D^b(\coh X).\)
\end{theorem}

{The theorem below was proved by Nordskova and Van den Bergh
for arbitrary full exceptional collections}
\cite[Theorem~8.6 and Corollary~8.7]{NordskovaVdB2024}.

\begin{theorem}
\label{thm:nordskova-vdb-transitivity}
Let \(X\) be a complex three-dimensional minifold.  The group
\(B_4\ltimes\mathbb Z^4\) acts freely and transitively on the full
exceptional collections in \(D^b(\coh X),\) where \(B_4\) acts by mutations and
\(\mathbb Z^4\) by independent shifts.  {After shifting its objects
independently,} every such {exceptional collection is strong and consists of vector bundles}.
\end{theorem}

For complex varieties, it is conjectured that projective spaces are the only
even-dimensional minifolds.  This is known in dimension two, while in
dimension four it is known for Fano minifolds
\cite[Theorem~1.1(1),(3)]{GKMS2013}.  In dimension five the known examples
are \(\mathbb P^5,\) \(Q^5,\) \(X_{18}^5,\) \(X_{16}^5,\) and
\(X_{12}^5;\) their geometry and exceptional collections are discussed in
\cite{Kuznetsov2006Hyperplane,Kuznetsov2018Spinor,
Kuznetsov2018Minifolds}.  A classification in dimension five is not known.
\section{Euler matrices and positivity}
\label{sec:numerical-euler-data}

In this section we introduce positive Euler matrices
(Definition~\ref{def:positive-euler-matrix}) and prove that a full geometric
helix of period four with maximally unipotent $S^{-1}[3]$ has a positive
Euler matrix (Theorem~\ref{thm:geometric-euler-positivity}).

\begin{definition}
An \emph{Euler matrix} is an
upper-triangular integer matrix of the form
\begin{equation}\label{eq:A}
A=\begin{pmatrix}
1&a&e&d\\
0&1&b&f\\
0&0&1&c\\
0&0&0&1
\end{pmatrix},
\qquad a,b,c,d,e,f\in\mathbb Z.
\end{equation}
We call \(a,b,c,d,e,f\) the \emph{parameters} of \(A.\)
\end{definition}

We put
\begin{equation}\label{eq:euler-operators}
B=A+A^t,
\qquad \Omega=A-A^t,
\qquad T=-A^{-t}A,
\qquad C=T^{-1}=-A^{-1}A^t.
\end{equation}
Thus, \(B\) and \(\Omega\) are the symmetrization and skew-symmetrization of
\(A,\) respectively.  For any
invertible matrix \(M,\) we use the convention \(M^{-t}=(M^{-1})^t.\)

\begin{definition}
We say that an Euler matrix \(A\) satisfies condition
\(\mathrm{(MU)}\) if \(T\) is maximally unipotent; that is, \(C\) has a
single Jordan block with eigenvalue \(1.\)
\end{definition}

When \(A\) is the Euler matrix of a full exceptional collection
{\(\mathcal E=(E_1,E_2,E_3,E_4)\) in a triangulated category \(\mathcal D\) with
Serre functor \(S\)}, we use the same symbol \(T\) for the functor
\(S^{-1}[3]\) and for its induced action on
\(K_0(\mathcal D)\otimes\mathbb Q.\)
The identities
\begin{equation}\label{eq:fixed-subspace-radical}
        T-I=-A^{-t}B,
        \qquad
        C-I=-A^{-1}B
\end{equation}
identify the fixed subspaces of \(T\) and \(C\) with \(\ker B.\)  If \(A\)
is the Euler matrix of {such an exceptional collection \(\mathcal E\)
and \(\mathsf S\) is the matrix of \(S\)} in the basis \([E_1],\ldots,[E_4],\) then Serre
duality gives
\(A=\mathsf S^tA^t.\)  Hence, \(\mathsf S=A^{-1}A^t,\) and \(T\) represents
\(S^{-1}[3]\) on \(K_0(\mathcal D)\otimes\mathbb Q.\)

For \(u,v,t\in\mathbb R,\) put
\[
        \Delta(u,v,t)=uvt-u^2-v^2-t^2+4.
\]
Equivalently, for every \(\delta\in\mathbb R,\) the equation
\(\Delta(u,v,t)=\delta\) is the Markov cubic
\(u^2+v^2+t^2-uvt=4-\delta.\)
The four principal \(3\times3\) minors of \(B\) are twice the values of \(\Delta\)
on
\begin{equation}\label{eq:principal-triples}
        (a,b,e),\qquad (a,d,f),\qquad (c,d,e),\qquad (b,c,f).
\end{equation}
We call these the \emph{principal triples} of \(A.\)

\subsection{Consequences of geometricity}
\label{sec:consequences-geometricity}

Throughout Subsection~\ref{sec:consequences-geometricity},
\(\mathcal D=\Perf(\mathsf X),\) where
\(\mathsf X\) is a smooth proper non-commutative scheme.  Thus,
\(\mathcal D\) is a \(\kk\)-linear Hom-finite idempotent-complete
triangulated category with a DG enhancement and a Serre functor \(S.\)

\begin{proposition}\label{prop:geometric-bounds}
Let \(\mathcal H=(E_i)_{i\in\mathbb Z}\) be a full geometric
helix of period four, set \(\mathcal E=(E_1,E_2,E_3,E_4),\) and let \(A\)
be its Euler matrix.  Then
$
        a,b,c,d,e,f\ge3.
$
For each principal triple \((u,v,t)\) of \(A,\) one has
\(\Delta(u,v,t)\ge4.\)  Consequently, every principal
\(3\times3\) minor of \(B\) is at least \(8.\)
\end{proposition}

\begin{proof}
By \cite[Proposition~2.6]{BondalPolishchuk1994}, for every
\(1\le i<j<k\le4,\) the ordered triple \((E_i,E_j,E_k)\) is a full
geometric exceptional collection in \(\langle E_i,E_j,E_k\rangle,\) and
hence, generates a geometric helix of period three.  Consider one such triple
\((E_0,E_1,E_2),\) with Euler matrix
\[
 A_0
 =\begin{pmatrix}
  1&u&v\\
  0&1&t\\
  0&0&1
 \end{pmatrix}.
\]
By Lemma~\ref{lem:nonzero-forward-homs}, the integers
\(u,v,t\) are positive.  The generators \(\sigma_1,\sigma_2\)
of the braid group \(B_3,\) together with their inverses, act on the Euler
parameters by
\begin{equation}\label{eq:period-three-coordinate-mutations}
\begin{aligned}
 \sigma_1(u,v,t)&=(u,uv-t,v),&
 \sigma_1^{-1}(u,v,t)&=(u,t,ut-v),\\
 \sigma_2(u,v,t)&=(ut-v,u,t),&
 \sigma_2^{-1}(u,v,t)&=(v,vt-u,t).
\end{aligned}
\end{equation}
Mutations preserve geometricity \cite[Theorem~2.3]{BondalPolishchuk1994},
so every coordinate in the orbit is positive.  No coordinate is equal to \(1.\) Indeed, if,
for example, \(u=1,\) the formulas
\eqref{eq:period-three-coordinate-mutations} give both \(v-t>0\) and
\(t-v>0.\)  For \(v=1\) and \(t=1,\)
\eqref{eq:period-three-coordinate-mutations} gives, respectively,
\(u-t,t-u>0\) and \(u-v,v-u>0.\)  Thus, all three entries of every triple in the
\(B_3\)-orbit of \((u,v,t)\) are at least \(2.\)

Choose a triple \((u',v',t')\) in this orbit minimizing \(u'+v'+t';\)
such a triple exists because these sums are positive integers.
Write its entries in nondecreasing order as \(2\le x\le y\le z.\)  According to which original coordinate equals \(z,\)
one of the four formulas in
\eqref{eq:period-three-coordinate-mutations} replaces \(z\) by \(xy-z\)
and {leaves the other two values unchanged, up to order}.  Minimality therefore gives
\[
        xy-z\ge z,
        \qquad
        y\le z\le\frac{xy}{2}.
\]
If \(x=2,\) then \(y=z=:r.\)  {By applying the mutations in
\eqref{eq:period-three-coordinate-mutations}, we can move the entry \(2\)
to the first position}, so a mutation-equivalent triple
\(\mathcal F=(F_0,F_1,F_2)\) has Euler matrix
\[
 \begin{pmatrix}
  1&2&r\\
  0&1&r\\
  0&0&1
 \end{pmatrix},
 \qquad r\ge2.
\]
Let \((F_i)_{i\in\mathbb Z}\) be its geometric helix.  Repeatedly using
\([R_FE]=\chi(E,F)[F]-[E]\) gives
\begingroup
\[
\begin{aligned}
 [F_3]&=[F_0]-2[F_1]+r[F_2],\\
 [F_4]&=2[F_0]-3[F_1]+r[F_2],\\
 [F_5]&=r[F_0]-r[F_1]+[F_2].
\end{aligned}
\]
Applying \(\chi(F_0,-)\) to the last equality gives
\[
 \chi(F_0,F_5)=r\chi(F_0,F_0)-r\chi(F_0,F_1)
                  +\chi(F_0,F_2)=r-2r+r=0.
\]
\endgroup
Since the helix is geometric, we have
\(\chi(F_0,F_5)=\dim_{\kk}\operatorname{Hom}(F_0,F_5).\)  This contradicts
Lemma~\ref{lem:nonzero-forward-homs}.  Hence, \(x\ge3.\)

The function \(s\mapsto\Delta(x,y,s)\) is non-decreasing on
\([y,xy/2],\) and hence
\[
 \Delta(x,y,z)
 \ge \Delta(x,y,y)
 =(x-2)y^2-x^2+4
 \ge (x-3)x^2+4
 \ge4.
\]
{Each of the four mutations in
\eqref{eq:period-three-coordinate-mutations} preserves \(\Delta.\)
Therefore} \(\Delta(u,v,t)=\Delta(x,y,z)\ge4.\)  Moreover, no coordinate of
\((u,v,t)\) is \(2,\) because
\(\Delta(2,q,s)=-(q-s)^2\le0.\)  Hence, \(u,v,t\ge3.\)

{Applying \(u,v,t\ge3\) and
\(\Delta(u,v,t)\ge4\) to each of the four triples in
\eqref{eq:principal-triples} gives
\(a,b,c,d,e,f\ge3\) and \(\Delta(a,b,e),\Delta(a,d,f),
\Delta(b,c,f),\Delta(c,d,e)\ge4\)}, since each principal
minor is \(2\Delta(u,v,t).\)
\end{proof}

\subsection{Adding maximal unipotence}
\label{sec:minifold-positivity}

For an Euler matrix \(A,\) define
\begin{align}
 \Phi(A)&=\operatorname{Pf}(\Omega)=ac+bd-ef,
                                                        \label{eq:Phi-def}\\
 \Psi(A)&=a^2+b^2+c^2+d^2+e^2+f^2-abe-adf-bcf-cde+abcd. \label{eq:Psi-def}
\end{align}
The characteristic polynomial of \(C\) is
\begin{equation}\label{eq:coxeter-characteristic}
\det(\lambda I-C)=\lambda^4+(4-\Psi(A))\lambda^3+\bigl(\Phi(A)^2-2\Psi(A)+6\bigr)\lambda^2+(4-\Psi(A))\lambda+1.
\end{equation}
Thus, \(\mathrm{(MU)}\) implies
\begin{equation}\label{eq:maximal-unipotence-relations}
\Phi(A)^2=16,\qquad \Psi(A)=8.
\end{equation}
Moreover,
\begin{equation}\label{eq:detB}
        \det B=\Phi(A)^2-4\Psi(A)+16.
\end{equation}
Conversely, if at least one principal \(3\times3\) minor of \(B\) is
non-zero, equation \eqref{eq:maximal-unipotence-relations} implies
\(\mathrm{(MU)}.\)
Indeed, \eqref{eq:coxeter-characteristic} gives the characteristic polynomial
\((\lambda-1)^4,\) while \eqref{eq:detB} and
\eqref{eq:fixed-subspace-radical} give
\(\operatorname{rank}(C-I)=\operatorname{rank}B=3.\)

\begingroup\begingroup
\begin{lemma}\label{lem:vieta-complements-positive}
Let \(u,v,t\in\mathbb Z,\) with \(u,v,t\ge3,\) and suppose that
\(\Delta(u,v,t)>0.\)  Then
\(uv-t\ge3,\) \(ut-v\ge3,\) and \(vt-u\ge3.\)
\end{lemma}

\begin{proof}
By symmetry, it is enough to consider \(t'=uv-t.\)  If \(t'\le2,\) then
\[
 \Delta(u,v,t')+(u-v)^2
   =(2-t')(2+t'-uv)\le0,
\]
which contradicts \(\Delta(u,v,t')=\Delta(u,v,t)>0.\)  Thus, \(t'\ge3.\)
\end{proof}
\endgroup\endgroup

\begin{proposition}\label{prop:strict-principal-deltas}
Let \(A\) be an Euler matrix satisfying \(\mathrm{(MU)}.\)  Suppose
 its parameters are at least \(3\) and each principal
\(3\times3\) minor of \(B=A+A^t\) is positive.  Then, for every principal triple
\((u,v,t)\) of \(A,\)
\begin{equation}\label{eq:strict-principal-deltas}
        \Delta(u,v,t)>4.
\end{equation}
\end{proposition}

\begin{proof}
\begingroup
{Let \(u,v,t\ge3\) be integers with \(\Delta(u,v,t)>0.\)
By Lemma~\ref{lem:vieta-complements-positive}, replacing \(t\) by
\(uv-t\) preserves this inequality and keeps every coordinate at least
\(3.\)}
Permute the coordinates so that
\(x\le y\le z,\) and replace \(z\) by \(xy-z\) whenever \(xy-z<z.\)
Each step preserves \(\Delta,\) keeps all coordinates at least \(3,\) and
strictly decreases their sum.  The process therefore stops with
\(3\le x\le y\le z\le xy/2.\)  On this interval \(\Delta(x,y,z)\)
is nondecreasing in \(z,\) so
\[
 \Delta(u,v,t)=\Delta(x,y,z)
 \ge (x-2)y^2-x^2+4\ge(x-3)x^2+4\ge4.
\]
Apply this to each principal triple of \(A.\)
If \(\Delta(u,v,t)=4,\) reduction of
\(u^2+v^2+t^2=uvt\) modulo \(3\) gives \(3\mid u,v,t.\)
Each term \(ac,\) \(bd,\) and \(ef\) contains an entry of that principal
triple, so \(3\mid\Phi(A),\) contrary to \(\Phi(A)^2=16.\)
Thus, every principal triple satisfies \(\Delta>4.\)
\endgroup
\end{proof}

For the remainder of Subsection~\ref{sec:minifold-positivity}, let
\(\mathsf X\) be a smooth proper
non-commutative scheme, put \(\mathcal D=\Perf(\mathsf X),\) and let
\(\mathcal H=(E_i)_{i\in\mathbb Z}\) be a full geometric helix of period
four in \(\mathcal D.\)  Suppose that \(T=S^{-1}[3]\) acts maximally unipotently on
\(K_0(\mathcal D)\otimes\mathbb Q.\)  By
\cite[Theorem~2.3]{BondalPolishchuk1994}, every exceptional collection
{in the braid-group orbit of} \((E_i,E_{i+1},E_{i+2},E_{i+3}),\) for some
\(i\in\mathbb Z,\) extends to a geometric helix.  For the Euler matrix of
any such exceptional collection,
Proposition~\ref{prop:geometric-bounds} gives a positive principal
\(3\times3\) minor of \(B,\) whereas
\eqref{eq:maximal-unipotence-relations} and \eqref{eq:detB} give
\(\det B=0.\)  Hence, \(\operatorname{rank}B=3.\)

\begingroup\begingroup
For a square matrix \(M,\) let \(\adj(M)\) denote its adjoint, that is,
the transpose of its cofactor matrix.
\endgroup\endgroup

\begin{lemma}\label{lem:adjoint-leading-term}
Let \(A\) be the Euler matrix of the exceptional collection
\((E_i,E_{i+1},E_{i+2},E_{i+3})\) for some \(i\in\mathbb Z,\) and
let \(B,T\) be as in \eqref{eq:euler-operators}.  For a non-zero vector
\(p\in\ker(B\colon\mathbb Q^4\to\mathbb Q^4),\) put \(q=Ap.\)  There is a
unique \(\lambda\in\mathbb Q^\times\) such that
\(
\adj(B)=\lambda pp^t
.\)  Then
\begin{equation}\label{eq:adjoint-leading-term}
 (T-I)^3=\adj(B)A^t,
 \qquad
 A(T-I)^3=\lambda qq^t,
\end{equation}
and, for all \(u,v\in\mathbb Q^4,\)
\begin{equation}\label{eq:euler-matrix-asymptotic}
 u^tAT^mv
        =\frac{\lambda}{6}(u^tq)(v^tq)m^3+O(m^2).
\end{equation}
\end{lemma}

\begin{proof}
Since \(\operatorname{rank}B=3,\) we have $\adj(B)=\lambda pp^t$.  Put
\(U=T-I=-A^{-t}B.\)  Its characteristic polynomial is \(t^4,\) so
{the characteristic-polynomial formula for the adjoint gives}
\(\adj(U)=-U^3.\)  Since \(\det A=1,\)
\[
 \adj(U)=\adj(B)\adj(-A^{-t})
        =-\adj(B)A^t=-\lambda pq^t.
\]
Thus, \(U^3=\adj(B)A^t=\lambda pq^t,\) and multiplication by \(A\) gives
\eqref{eq:adjoint-leading-term}.  Finally, the binomial expansion
\[
        T^m=(I+U)^m
        =I+mU+\binom m2U^2+\binom m3U^3
\]
and \(U^3=\lambda pq^t\) give
\eqref{eq:euler-matrix-asymptotic}.
\end{proof}

For \(1\le i,j\le4,\) let \(D_{ij}\) be the minor of \(B\) obtained by
deleting row \(i\) and column \(j.\)  Since \(B\) is symmetric,
\(\adj(B)_{ij}=(-1)^{i+j}D_{ij}.\)

\begin{proposition}\label{prop:geometric-minors}
Let \(\mathcal H=(E_i)_{i\in\mathbb Z}\) be a full geometric
helix of period four in \(\mathcal D,\)
and suppose that \(T=S^{-1}[3]\) acts maximally unipotently on
\(K_0(\mathcal D)\otimes\mathbb Q.\)  Let
\(\mathcal E=(E_1,E_2,E_3,E_4),\) let \(A\) be its
Euler matrix, and put \(B=A+A^t.\)  Then every \(3\times3\) minor of \(B\)
is positive.
\end{proposition}

\begin{proof}
\begingroup
Choose \(0\ne p\in\ker B,\) and {write
\(\pi=\sum_i p_i[E_i]\in K_0(\mathcal D)\otimes\mathbb Q\)},
and write \(\adj(B)=\lambda pp^t.\)
By Proposition~\ref{prop:geometric-bounds},
\(D_{ii}=\lambda p_i^2\ge8,\) so \(\lambda>0\) and every \(p_i\ne0.\)
For the {right dual exceptional collection
\(\mathcal E^\vee=(F_4,F_3,F_2,F_1)\) of \(\mathcal E\)}, duality
\eqref{eq:normalized-duality} gives
\(\chi(F_i,\pi)=(-1)^ip_i.\)
{The exceptional collection \(\mathcal E^\vee\) also generates a geometric helix.}  Hence, for \(m\ge1,\)
Lemma~\ref{lem:nonzero-forward-homs} and
\eqref{eq:euler-matrix-asymptotic} give
\[
 0<\chi(F_i,T^mF_j)
   =\frac{\lambda(-1)^{i+j}p_ip_j}{6}m^3+O(m^2)
   =\frac{D_{ij}}6m^3+O(m^2).
\]
Thus, \(D_{ij}\ge0.\)  Since
\(D_{ij}=(-1)^{i+j}\lambda p_ip_j\ne0,\) every \(D_{ij}\) is positive.
\endgroup
\end{proof}

\subsection{Positive Euler matrices and the Euler degree}

In this subsection, we define positive Euler matrices and their Euler degree.

\begingroup\begingroup
\begin{definition}\label{def:positive-euler-matrix}
An Euler matrix \(A\) is \emph{positive} if
\begin{enumerate}[label=\textup{(\roman*)}]
\item its parameters satisfy \(a,b,c,d,e,f\ge3;\)
\item every \(3\times3\) minor of \(B=A+A^t\) is positive;
\item it satisfies \(\mathrm{(MU)};\)
\item for every principal triple \((u,v,t)\) in
\eqref{eq:principal-triples}, one has \(\Delta(u,v,t)>4.\)
\end{enumerate}
\end{definition}
In fact, condition \textup{(iv)} follows from \textup{(i)}--\textup{(iii)} by
Proposition~\ref{prop:strict-principal-deltas}.
\endgroup\endgroup

\begin{theorem}\label{thm:geometric-euler-positivity}
Let \(\mathsf X\) be a three-dimensional non-commutative minifold, and let
\((E_i)_{i\in\mathbb Z}\) be a full geometric helix of period four in
\(\Perf(\mathsf X).\) Then the Euler matrix of the exceptional collection
\((E_i,E_{i+1},E_{i+2},E_{i+3})\) is positive for every \(i\in\mathbb Z.\)
\end{theorem}

\begingroup\begingroup
\begin{proof}
Conditions \textup{(i)} and \textup{(ii)} follow from
Propositions~\ref{prop:geometric-bounds} and~\ref{prop:geometric-minors}.
Condition \textup{(iii)} holds by Definition~\ref{def:orlov-minifold}.
Condition \textup{(iv)} then follows from
Proposition~\ref{prop:strict-principal-deltas}.
This proves Theorem~\ref{thm:intro-positivity} stated in the Introduction.
\end{proof}
\endgroup\endgroup

Since \(B\) is symmetric of rank \(3\) and
\(\adj(B)=\lambda pp^t,\) we have
\begin{equation}\label{eq:deleted-minor-factorization}
        D_{ij}=(-1)^{i+j}\lambda p_ip_j.
\end{equation}

\begin{proposition}\label{prop:positive-radical-vector}
Let \(A\) be a positive Euler matrix.  Then
\(\ker B\cap\mathbb Z^4\) has a unique primitive generator of the form
\(
        p=(-x,y,-z,w)^t,
\) where
\(  x,y,z,w\in\mathbb Z_{>0}.
\)
\end{proposition}

\begin{proof}
Maximal unipotence and
\eqref{eq:fixed-subspace-radical} give \(\operatorname{rank}B=3.\)  For a
primitive generator \(p\) of \(\ker B\cap\mathbb Z^4,\) write
$
        \adj(B)=\lambda pp^t.
$
The principal \(3\times3\) minors show that \(\lambda>0\) and that no coordinate of
\(p\) vanishes.  Applying \eqref{eq:deleted-minor-factorization} with
\((i,j)=(1,2),(2,3),(3,4)\) shows that consecutive coordinates of \(p\) have
opposite signs.  {Changing the common sign if necessary gives
\(p_1,p_3<0<p_2,p_4.\)}  Primitivity determines it up to sign, and the condition
\(p_1,p_3<0<p_2,p_4\) fixes the sign.
\end{proof}

For a positive Euler matrix \(A,\) let \(p\) denote the primitive
radical vector given by Proposition~\ref{prop:positive-radical-vector}.
Since \(B\) has rank \(3,\) there is a unique non-zero rational number
\(D(A)\) such that
\begin{equation}\label{eq:euler-degree}
        \adj(B)=D(A)pp^t.
\end{equation}
The positivity of the principal \(3\times3\) minors gives \(D(A)>0.\)  We call
\(D(A)\) the \emph{Euler degree} of \(A.\)

\begin{proposition}\label{prop:euler-degree-arithmetic}
Let \(A\) be a positive Euler matrix, and let \(p\) be the
primitive radical vector given by
Proposition~\ref{prop:positive-radical-vector}.  The degree \(D(A)\) defined
by \eqref{eq:euler-degree} is a positive even integer.  It is the
greatest common divisor of the entries of \(\adj(B).\)  If
\(U\in\operatorname{GL}_4(\mathbb Z),\)
then, for
\(
        B'=U^tBU,\; p'=U^{-1}p,
\)
one has \(\adj(B')=D(A)p'(p')^t.\)  In particular, whenever
\(A'=U^tAU\) is again a positive Euler matrix, \(D(A')=D(A).\)
\end{proposition}

\begin{proof}
Since \(p\) is primitive, the integers \(p_ip_j,\)
\(1\le i,j\le4,\) have greatest common divisor \(1.\)  Choose integers
\(c_{ij}\) such that \(\sum_{i,j}c_{ij}p_ip_j=1.\)  Equation
\eqref{eq:euler-degree} gives
\(D(A)=\sum_{i,j}c_{ij}\adj(B)_{ij}\in\mathbb Z.\)
Since \(\adj(B)_{ij}=D(A)p_ip_j,\) the greatest common divisor of the
entries of \(\adj(B)\) is \(D(A).\)

{With \(B'\) and \(p'\) as in the statement, the formula for the adjoint
under a change of basis and the equality \(\det U=\pm1\) give}
\[
\adj(B')
 =U^{-1}\adj(B)U^{-t}
 =D(A)p'(p')^t.
\]
The vector \(p'\) is primitive.  If \(A'=U^tAU\) is a positive
Euler matrix, \eqref{eq:euler-degree} applied to \(A'\) gives
\(D(A')=D(A).\)
For the parity assertion, extend \(p\) to an integral basis.  In that
basis \(B\) is \(B_0\oplus0,\) where \(B_0\) represents the induced
form on \(\mathbb Z^4/\mathbb Zp,\) and \(\det B_0=D(A).\)
This rank-three lattice is even, so \(B_0\) modulo \(2\) is an
alternating matrix of odd size.  Its determinant is zero modulo \(2,\)
and \(D(A)\) is even.
\end{proof}
\begingroup\begingroup
\subsection{Euler degree of a strong minifold}

In this subsection, we identify the Euler degree of a strong three-dimensional
minifold with its anticanonical degree.

\begin{proposition}\label{prop:strong-minifold-euler-degree}
Let \(X\) be a strong three-dimensional minifold, and let
\(\mathcal E=(E_1,E_2,E_3,E_4)\) be a full exceptional collection in
\(D^b(\coh X)\) generating a geometric helix. Let
\(A\) be its Euler matrix.  Then \(A\) is positive and satisfies
\(
        D(A)=(-K_X)^3.
\)
\end{proposition}

\begin{proof}

The matrix \(A\) is positive by
Proposition~\ref{prop:strong-is-noncommutative-minifold} and
Theorem~\ref{thm:geometric-euler-positivity}.  Let \(p\) be its primitive
radical vector, put \(q=Ap,\) and write
\[
        r=(\operatorname{rk}E_1,\ldots,\operatorname{rk}E_4)^t.
\]
Both \(q^t\) and \(r^t\) annihilate \(\operatorname{im}(T-I).\)
For \(r,\) this follows because tensoring by \(\omega_X^{-1}\) preserves
rank.  For \(q,\) it follows from \(Tp=p\) and the invariance of
\(\chi\): \(\chi(Tu,p)=\chi(u,p).\)
Maximal unipotence gives \(\dim\operatorname{im}(T-I)=3,\) so \(q\)
and \(r\) span the same rational line.
They are both primitive integral vectors: \(q=Ap\) is primitive because
\(A\) is unimodular.  The classes \([E_i]\) form a
\(\mathbb Z\)-basis of \(K_0(X),\) so writing \([\mathcal O_X]\) in this
basis and applying the rank homomorphism shows that
\(\gcd_i(\operatorname{rk}E_i)=1;\) hence, \(r\) is primitive.
Therefore, \(q=\pm r.\)

For classes \(\alpha,\beta\) with coordinate vectors \(u,v,\) we obtain
\(
        (u^tq)(v^tq)=\operatorname{rk}(\alpha)\operatorname{rk}(\beta).
\)
Lemma~\ref{lem:adjoint-leading-term} now gives
\[
 \chi(\alpha,T^m\beta)
 =\frac{D(A)}6\operatorname{rk}(\alpha)\operatorname{rk}(\beta)m^3
   +O(m^2).
\]
Hirzebruch--Riemann--Roch gives the cubic coefficient
\(\frac{(-K_X)^3}{6}\operatorname{rk}(\alpha)
\operatorname{rk}(\beta).\)  Taking
\(\alpha=\beta=[\mathcal O_X]\) and comparing
cubic coefficients proves the identity.
\end{proof}
\endgroup\endgroup

\section{The braid-group action on Euler matrices}
\label{sec:braid-action-euler-matrices}

{In this section we describe the action of the braid group \(B_4\) on Euler matrices and prove that it
preserves positivity and the Euler degree.}

Let \(A\) be an Euler matrix as in \eqref{eq:A}, let
\(\mathbf e=(e_1,e_2,e_3,e_4)\) be a basis of \(\mathbb Z^4,\)
and put \(\chi(u,v)=u^tAv.\)  For \(1\le i\le3,\) set
\(h_i=\chi(e_i,e_{i+1}).\)  Define \(\sigma_i\) and \(\sigma_i^{-1}\)
by replacing the pair \((e_i,e_{i+1})\) as follows and leaving
the other basis vectors fixed:
\begin{equation}\label{eq:braid-action-euler-bases}
\begin{aligned}
\sigma_i(e_i,e_{i+1})&=(h_ie_i-e_{i+1},e_i),\qquad
\sigma_i^{-1}(e_i,e_{i+1})&=(e_{i+1},h_ie_{i+1}-e_i).
\end{aligned}
\end{equation}

{For each \(i,\) \(\sigma_i\) and \(\sigma_i^{-1}\) are inverse.
The \(\sigma_i\) satisfy the braid relations and define an action of
\(B_4\) on Euler matrices.}  When \(A\) is the Euler matrix of an exceptional collection
{\(\mathcal E=(E_1,E_2,E_3,E_4)\) in a triangulated category \(\mathcal D\)},
{Equation~\eqref{eq:braid-action-euler-bases} gives the Euler matrices
of the left and right mutations of the consecutive objects.}

\begin{lemma}\label{lem:coordinate-mutations}
{The standard generators of \(B_4\) and their inverses act on the
parameters of \(A\) as follows:}
\begin{align}
\sigma_1(a,b,c,d,e,f)=(a,e,c,ad-f,ae-b,d),&&
\sigma_1^{-1}(a,b,c,d,e,f)=(a,ab-e,c,f,b,af-d),\notag\\
\sigma_2(a,b,c,d,e,f)=(ab-e,b,f,d,a,bf-c),&&
\sigma_2^{-1}(a,b,c,d,e,f)=(e,b,bc-f,d,be-a,c),\label{eq:elementary-mutations}\\
\sigma_3(a,b,c,d,e,f)=(a,bc-f,c,e,ce-d,b),&&
\sigma_3^{-1}(a,b,c,d,e,f)=(a,f,c,cd-e,d,cf-b).\notag
\end{align}
\end{lemma}

\begin{proof}
For example, \(\sigma_1\) replaces the ordered basis by
\(
        (ae_1-e_2,e_1,e_3,e_4).
\)
Computing the matrix of \(\chi\) in this basis gives the formula
for \(\sigma_1\) in \eqref{eq:elementary-mutations}; the other five
identities in that equation follow in the same way.
\end{proof}

Two Euler matrices are called \emph{mutation-equivalent} {if they lie in the
same orbit under this \(B_4\)-action}.

\begin{lemma}\label{lem:formal-mutation-invariants}
Let
\(W=\sigma_{i_m}^{\delta_m}\cdots\sigma_{i_1}^{\delta_1},\) where
\(\delta_j\in\{\pm1\}.\)  Starting from \(\mathbf e,\) apply successively
\(\sigma_{i_1}^{\delta_1},\ldots,\sigma_{i_m}^{\delta_m}\) using
\eqref{eq:braid-action-euler-bases}.  Let
\(U_W\in\operatorname{SL}_4(\mathbb Z)\) be the matrix whose columns are
the {new basis vectors} written in the coordinates of \(\mathbf e.\)
Let \(p\in\mathbb Z^4,\) and let \(p_W\) be the coordinate vector
of the same element of \(\mathbb Z^4\) in the {new basis}.
Then
\begin{equation}\label{eq:formal-mutation-change-of-basis}
 A_W=U_W^tAU_W,
 \qquad B_W=U_W^tBU_W,
 \qquad p_W=U_W^{-1}p.
\end{equation}
The operators \(T\) and \(C\) of \eqref{eq:euler-operators} are conjugated
by \(U_W,\) and
\(
        \Phi(A_W)=\Phi(A),
        \;
        \Psi(A_W)=\Psi(A).
\)

If \(A\) and \(A_W\) are positive, then \(D(A_W)=D(A).\)
\end{lemma}

\begin{proof}
The identities \eqref{eq:formal-mutation-change-of-basis} are the
change-of-basis formulas.  Substitution in \eqref{eq:euler-operators} gives
\(T_W=U_W^{-1}TU_W\) and \(C_W=U_W^{-1}CU_W.\)  Moreover,
\(\Omega_W:=A_W-A_W^t=U_W^t\Omega U_W.\)  For every skew-symmetric
matrix \(\Omega,\)
\(\operatorname{Pf}(U_W^t\Omega U_W)=\det(U_W)\operatorname{Pf}(\Omega).\)
Since \(\det U_W=1,\) definition \eqref{eq:Phi-def} gives
\(\Phi(A_W)=\Phi(A).\)  Since
\(\det B_W=\det B,\) equation~\eqref{eq:detB} then gives
\(\Psi(A_W)=\Psi(A).\)  The equality \(D(A_W)=D(A)\) follows from
Proposition~\ref{prop:euler-degree-arithmetic}.
\end{proof}

\begin{lemma}\label{lem:radical-mutation-inequalities}
Let \(A\) be a positive Euler matrix, let \(p\) be the primitive
radical vector given by Proposition~\ref{prop:positive-radical-vector}, and write
\(p=(-x,y,-z,w)^t,\) where \(x,y,z,w\in\mathbb Z_{>0}.\)  Put
\(\varepsilon=\Phi(A)/4\in\{\pm1\}.\)  Then
\begin{align}
 N&:=axy-x^2-y^2=czw-z^2-w^2>0,\label{eq:mutation-N}\\
 L&:=byz-y^2-z^2>0.\label{eq:mutation-L}
\end{align}
Moreover,
\begin{equation}\label{eq:mutation-square-identities}
        D(A)N=2(a-\varepsilon c)^2,
        \qquad
        D(A)L=2(d-\varepsilon b)^2.
\end{equation}
\end{lemma}

\begin{proof}
The relation \(Bp=0\) implies \(p^tAp=0.\)  Moreover, its third and fourth
components give
\[
 -z(Bp)_3+w(Bp)_4
 =\bigl(axy-x^2-y^2\bigr)-\bigl(czw-z^2-w^2\bigr)+p^tAp.
\]
This proves the equality of the two expressions for \(N\) in
\eqref{eq:mutation-N}.  Since
\(\adj(B)=D(A)pp^t,\) direct expansion gives
\begin{align*}
 D(A)N&=-a\adj(B)_{12}-\adj(B)_{11}-\adj(B)_{22}
        =2\Psi(A)-16+2a^2+2c^2-ac\Phi(A),\\
 D(A)L&=-b\adj(B)_{23}-\adj(B)_{22}-\adj(B)_{33}
        =2\Psi(A)-16+2b^2+2d^2-bd\Phi(A).
\end{align*}
Equation \eqref{eq:maximal-unipotence-relations} gives
\eqref{eq:mutation-square-identities}.  Hence, \(N,L\ge0.\)

If \(N=0,\) then \(x^2+y^2=axy.\)  {Write
\(x=gm\) and \(y=gn,\) where \(g=\gcd(x,y)\) and \(\gcd(m,n)=1.\)  Then}
\(m^2+n^2=amn.\)  Thus, \(m\mid n^2\) and \(n\mid m^2,\) so
\(m=n=1\) and \(a=2,\) contrary to positivity.  Therefore, \(N>0.\)
If \(L=0,\) write \(y=gm\) and \(z=gn,\) where
\(g=\gcd(y,z)\) and \(\gcd(m,n)=1.\)  Then
\(m^2+n^2=bmn,\) so \(m\mid n^2\) and \(n\mid m^2.\)  Hence,
\(m=n=1\) and \(b=2,\) again a contradiction.  Thus, \(L>0.\)
\end{proof}

\begin{proposition}\label{prop:mutation-positivity}
The action of \(B_4\) on Euler matrices preserves positivity.
\end{proposition}

\begin{proof}
{It is enough to consider one of the six elementary mutations in
Lemma~\ref{lem:coordinate-mutations}.}  Write it as
\(W=\sigma_i^\delta,\) where \(1\le i\le3\) and
\(\delta\in\{1,-1\}.\)  Every parameter of {\(A_W\)} is either a parameter of \(A\) or {one of
\(uv-t,ut-v,vt-u\) for a principal triple \((u,v,t)\)}.
Lemma~\ref{lem:vieta-complements-positive} therefore
shows that {every parameter of \(A_W\) is at least \(3\)}.  Maximal
unipotence is preserved by Lemma~\ref{lem:formal-mutation-invariants}.

Let \(p=(-x,y,-z,w)^t\) be the primitive radical vector.  Write
\(p_W=(-x_W,y_W,-z_W,w_W)^t,\) without assuming that the four
quantities are positive.  Formula~\eqref{eq:braid-action-euler-bases}
gives
\[
\begin{array}{c|c}
W&(x_W,y_W,z_W,w_W)\\
\hline
\sigma_1&(y,ay-x,z,w)\\
\sigma_1^{-1}&(ax-y,x,z,w)\\
\sigma_2&(x,z,bz-y,w)\\
\sigma_2^{-1}&(x,by-z,y,w)\\
\sigma_3&(x,y,w,cw-z)\\
\sigma_3^{-1}&(x,y,cz-w,z).
\end{array}
\]

All entries in the second column are positive by
\eqref{eq:mutation-N} and \eqref{eq:mutation-L}, since
\begin{align*}
 y(ax-y)&=N+x^2,&
 x(ay-x)&=N+y^2,&
 z(by-z)&=L+y^2,\\
 y(bz-y)&=L+z^2,&
 w(cz-w)&=N+z^2,&
 z(cw-z)&=N+w^2.
\end{align*}
{The adjoint satisfies} \(\adj(B_W)=D(A)p_Wp_W^t.\)
For the minor \(D_{ij}^W\) {formed by deleting} row \(i\) and column \(j,\)
this gives
\(
        D_{ij}^W=(-1)^{i+j}D(A)(p_W)_i(p_W)_j>0
\)
because the coordinates of \(p_W\) have alternating signs.
{Thus, condition \textup{(ii)} is preserved as well, and condition
\textup{(iv)} follows from Proposition~\ref{prop:strict-principal-deltas}.}
\end{proof}
\section{Euler--Fricke correspondence and modular classification}
\label{sec:euler-fricke}

\begingroup\begingroup
In this section we classify positive Euler matrices up to mutations,
thus proving the classification statement of
Theorem~\ref{thm:intro-classification} stated in the Introduction.
We associate to each positive Euler matrix a unique \(\SL_2\)-character of the
twice-punctured torus and identify mutation orbits with pure
mapping-class-group orbits, see
Propositions~\ref{prop:euler-fricke-braid-equivariance}--\ref{prop:euler-fricke-dictionary}.
Using Lemma~\ref{lem:integral-traces-modular} and
Theorem~\ref{thm:positive-fricke-character}, we realize these characters by representations
\(F_3\to\SL_2(\mathbb Z).\) Their images in \(\PSL_2(\mathbb Z)\)
are torsion-free subgroups of index twelve, and the corresponding quotients
of \(\mathbb H\) have genus one and two cusps, see
Proposition~\ref{prop:index-twelve}.
By Proposition~\ref{prop:modular-orbit-injectivity}, two positive Euler matrices
are mutation-equivalent if the corresponding subgroups are conjugate in
\(\PSL_2(\mathbb Z).\)
The five conjugacy classes of such subgroups in
\cite[Section~6]{SebbarBesrour2021} give at most five mutation orbits.
Five positive matrices with distinct Euler degrees complete the
classification in Theorem~\ref{thm:main-reduction}.
\endgroup\endgroup

\subsection{Character varieties}

\begingroup\begingroup
Let \(\Gamma\) be a finitely generated group.  Denote by
\(\operatorname{Hom}(\Gamma,\SL_2(\kk))\) the affine representation
variety whose \(\kk\)-points are the homomorphisms
\(\Gamma\to\SL_2(\kk).\) Indeed, evaluation on generators
\(\gamma_1,\ldots,\gamma_r\) identifies this set with the tuples
\((M_1,\ldots,M_r)\in\SL_2(\kk)^r\) satisfying
\(w(M_1,\ldots,M_r)=I\) for every
relation \(w(\gamma_1,\ldots,\gamma_r)=1\) in \(\Gamma.\)
Hence, this subset is closed and affine, cf.
\cite[Section~5]{Sikora2012}. The \(\SL_2\)-character variety of
\(\Gamma\) is
\[
 \mathcal X(\Gamma)
 :=\operatorname{Hom}(\Gamma,\SL_2(\kk))
   \mathbin{/\mkern-6mu/}\SL_2(\kk).
\]
Here the affine quotient is taken with respect to the conjugation action
of \(\SL_2(\kk).\) Its coordinate algebra is
\[
 \kk[\mathcal X(\Gamma)]
 =\kk[\operatorname{Hom}(\Gamma,\SL_2(\kk))]
      ^{\SL_2(\kk)}.
\]
Let \(q_\Gamma\colon\operatorname{Hom}(\Gamma,\SL_2(\kk))\to\mathcal X(\Gamma)\)
denote the quotient morphism, and put \([\rho]=q_\Gamma(\rho).\)
Each fibre of \(q_\Gamma\) over a \(\kk\)-point contains a unique
closed orbit. An orbit is closed if and only if the corresponding
representation is completely reducible, cf. \cite[Theorem~30]{Sikora2012}.
Thus, the \(\kk\)-points of \(\mathcal X(\Gamma)\)
correspond to \(\SL_2(\kk)\)-conjugacy classes of completely
reducible representations, cf.
\cite[Section~11]{Sikora2012}.
\endgroup\endgroup

\subsection{The trace construction}
\label{subsec:trace-construction}

{Every \(\gamma\in\Gamma\)}
defines a regular function
\(
 [\rho]\mapsto\tr\rho(\gamma)
\)
on \(\mathcal X(\Gamma).\)  

Let \(F_3=\langle\alpha_1,\alpha_2,\alpha_3\rangle\) be the free group of
rank three.  For \([\rho]\in\mathcal X(F_3),\) define
\begin{equation}\label{eq:trace-stokes-matrix}
 A_\rho=
 \begin{pmatrix}
  1&\tr\rho(\alpha_1)&\tr\rho(\alpha_1\alpha_2)
      &\tr\rho(\alpha_1\alpha_2\alpha_3)\\
  0&1&\tr\rho(\alpha_2)&\tr\rho(\alpha_2\alpha_3)\\
  0&0&1&\tr\rho(\alpha_3)\\
  0&0&0&1
 \end{pmatrix}.
\end{equation}
Let \(\mathcal V(4)\simeq\mathbb A^6_{\kk}\) be the affine space of Euler
matrices \eqref{eq:A} whose six parameters lie in \(\kk\).
\begingroup
The six trace functions in \eqref{eq:trace-stokes-matrix} define a morphism
\begin{equation}\label{eq:theta-map}
 \Theta\colon\mathcal X(F_3)\longrightarrow\mathcal V(4),
 \qquad [\rho]\longmapsto A_\rho.
\end{equation}
\endgroup

Set \(t(\rho)=\tr\rho(\alpha_1\alpha_3).\)  The polynomials
\(\Phi(A)\) and \(\Psi(A)\) are defined in
\eqref{eq:Phi-def} and \eqref{eq:Psi-def}.

\begin{proposition}
\label{prop:full-euler-fricke-cover}
The map
\(
 \widetilde\Theta=(\Theta,t)\colon\mathcal X(F_3)
 \longrightarrow\mathcal V(4)\times\mathbb A^1_{\kk}
\)
identifies \(\mathcal X(F_3)\) with the hypersurface
\begin{equation}\label{eq:full-euler-fricke-cover}
 \left\{(A,t)\in\mathcal V(4)\times\mathbb A^1_{\kk}
 \ \middle|\
 t^2-\Phi(A)t+\Psi(A)-4=0\right\}.
\end{equation}
Equivalently, we have
\(
 \kk[\mathcal X(F_3)]
 \simeq
 \kk[a,b,c,d,e,f,t]
 \big/\bigl(t^2-\Phi(A)t+\Psi(A)-4\bigr)
\) and, hence,
the projection \(\Theta\) in \eqref{eq:theta-map} is finite and flat
of degree two.\end{proposition}

Indeed, the trace function $t(\rho)$,
together with the six entries of \(A_\rho,\) generates the coordinate
{algebra} of \(\mathcal X(F_3),\) subject to the single relation
\eqref{eq:full-euler-fricke-cover};
see \cite[equations~(5.1.1)--(5.1.3) and Proposition~5.1.1]{Goldman2009}.
This relation goes back to \cite{Vogt1889}.  We use its formulation in
\cite[Propositions~5.9 and~6.8]{FanWhang2024}, where the matrices in \(\mathcal V(4)\) are called Stokes matrices.

\subsection{The braid-group action}

Define automorphisms \(\beta_1,\beta_2,\beta_3\) of \(F_3\) by
\begin{equation}\label{eq:free-group-braid-automorphisms}
\begin{array}{c|ccc}
 &\alpha_1&\alpha_2&\alpha_3\\ \hline
\beta_1&\alpha_1&\alpha_1\alpha_2&\alpha_3\\
\beta_2&\alpha_1\alpha_2^{-1}&\alpha_2&\alpha_2\alpha_3\\
\beta_3&\alpha_1&\alpha_2\alpha_3^{-1}&\alpha_3.
\end{array}
\end{equation}
Direct substitution shows that the \(\beta_i\) satisfy the
relations of the braid group \eqref{def-braid-group}; hence, \(\sigma_i\mapsto\beta_i\) extends to a homomorphism
\(\beta\colon B_4\to\operatorname{Aut}(F_3).\)

\begingroup\begingroup
We define an action of the braid group \(B_4\) on
\(\mathcal X(F_3)\) by setting, for \(i=1,2,3,\)
\begin{equation}\label{eq:character-braid-action}
 \sigma_i\cdot\rho=\rho\circ\beta_i,\qquad
 \sigma_i\cdot[\rho]=[\rho\circ\beta_i].
\end{equation}
\endgroup\endgroup
\begin{proposition}\label{prop:euler-fricke-braid-equivariance}
For the $B_4$-action on
\(\mathcal X(F_3)\) as in  \eqref{eq:character-braid-action} and the \(B_4\)-action on \(\mathcal V(4)\) given by
\eqref{eq:elementary-mutations}, the map
\(\Theta\) is \(B_4\)-equivariant:
\begin{equation}\label{eq:trace-braid-equivariance}
 \Theta(\sigma_i\cdot[\rho])
 =\sigma_i\bigl(\Theta([\rho])\bigr),
 \qquad i=1,2,3.
\end{equation}
\end{proposition}

\begin{proof}
Write \(A_\rho=A(a,b,c,d,e,f).\)  The identity
\(M^{-1}=\tr(M)I-M\) for \(M\in\SL_2(\kk)\) gives
\[
\begin{aligned}
 A_{\rho\circ\beta_1}&=A(a,e,c,ad-f,ae-b,d),\\
 A_{\rho\circ\beta_2}&=A(ab-e,b,f,d,a,bf-c),\\
 A_{\rho\circ\beta_3}&=A(a,bc-f,c,e,ce-d,b).
\end{aligned}
\]
{These are precisely the formulas for the action of
\(\sigma_1,\sigma_2,\sigma_3\) in Lemma~\ref{lem:coordinate-mutations}.}
\end{proof}

\subsection{The twice-punctured torus and its boundary classes}

Let \(D\) be an oriented closed disc and let
\(\lambda_1,\ldots,\lambda_4\) be four distinct points in its interior.
Let \(\varpi\colon\Sigma\to D\) be the connected double cover branched
over these four points.  Riemann--Hurwitz gives
\[
        \chi(\Sigma)=2\chi(D)-4=-2.
\]
The monodromy around each of the four branch points exchanges the two sheets.
The monodromy along \(\partial D\) is the product of these four permutations
and is therefore trivial.  Hence
\(\varpi^{-1}(\partial D)\) has two components.  If \(g\) is the genus of
\(\Sigma,\) then \(-2=2-2g-2,\) so \(g=1.\)  Thus,
\(\Sigma=\Sigma_{1,2},\) and its interior
\(\Sigma_{1,2}^{\circ}\) is a torus with two punctures.

\begingroup\begingroup
Choose an embedded arc in \(D\) passing through
\(\lambda_1,\ldots,\lambda_4\) in this order, and let \(I_i\) be its subarc
from \(\lambda_i\) to \(\lambda_{i+1}.\)  Choose a base point \(*\in\Sigma_{1,2}^{\circ},\)
orient the lifted curves \(\varpi^{-1}(I_i)\), and choose paths from
\(*\) to these curves so that
\((\alpha_1,\alpha_2,\alpha_3)\) is the sequence of generators of \(\pi_1(\Sigma_{1,2},*)\)
specified in \cite[Definition~4.8 and Figure~1]{FanWhang2024}.  These generators give an isomorphism
\begin{equation}\label{eq:surface-free-group}
 \pi_1(\Sigma_{1,2})
 \simeq F_3.
\end{equation}
\endgroup\endgroup
Set
\begin{equation}\label{eq:goldman-generators-boundary}
 U=\alpha_1,
 \qquad X=\alpha_2^{-1},
 \qquad Y=\alpha_2\alpha_3,
 \qquad K_1=UXY,
 \qquad K_2=UYX.
\end{equation}
The elements \(U,X,Y\) also form a basis of $\pi_1(\Sigma_{1,2})$, since
\(
\alpha_1=U,
\alpha_2=X^{-1},
\alpha_3=XY
.\)
The words \(K_1\) and \(K_2\) represent the two boundary components $\partial \Sigma$, up to
conjugacy and inversion.  In particular,
\[
 K_1=\alpha_1\alpha_3,
 \qquad
 K_2=\alpha_1\alpha_2\alpha_3\alpha_2^{-1}.
\]
Write \(\mathcal X(\Sigma_{1,2})\) for the \(\SL_2\)-character variety
of \(\pi_1(\Sigma_{1,2}).\)  The basis \eqref{eq:surface-free-group}
identifies it with \(\mathcal X(F_3).\)

Proposition~6.8 of \cite{FanWhang2024} proves that, if
\(A_\rho=A(a,b,c,d,e,f),\) then
\begin{equation}\label{eq:boundary-traces}
 \tr\rho(K_1)+\tr\rho(K_2)=\Phi(A_\rho),
 \qquad
 \tr\rho(K_1)\tr\rho(K_2)=\Psi(A_\rho)-4
\end{equation}
Since
\(t=\tr\rho(\alpha_1\alpha_3)=\tr\rho(K_1),\) the other root of the
quadratic equation in \eqref{eq:full-euler-fricke-cover} is
\(\tr\rho(K_2).\)

\begingroup\begingroup
Let \(\widehat\Sigma\) be the torus obtained by filling the two punctures,
and let \(p_1,p_2\) be the added points.  Label them so that, for \(i=1,2,\)
a small loop going once around \(p_i\) represents \(K_i\) up to conjugacy
and inversion.
The pure mapping class group
\(\operatorname{PMod}^{+}(\Sigma_{1,2}^{\circ})\) consists of
orientation-preserving self-homeomorphisms of \(\widehat\Sigma\) fixing
each \(p_i,\) modulo isotopies fixing them throughout.  A mapping class
induces an element of \(\operatorname{Out}(\pi_1(\Sigma_{1,2}^{\circ},*)).\)
Precomposition defines a right action of
\(\operatorname{PMod}^{+}(\Sigma_{1,2}^{\circ})\) on \(\mathcal X(F_3)\):
\(
 [\rho]\cdot[f]=[\rho\circ f_*].
\)
Here \(f_*\in\operatorname{Aut}(F_3)\) is any representative of
the outer automorphism induced by \(f;\) the right-hand side is
independent of this choice because inner automorphisms act trivially
on characters.
\endgroup\endgroup
\begingroup
{For \(i=1,2,3,\) let \(H_i\) be the positive half-twist interchanging
\(\lambda_i\) and \(\lambda_{i+1},\) supported in a small neighbourhood of \(I_i.\)}  {Among its two
lifts to \(\Sigma_{1,2},\) let \(\widetilde H_i\) be the one equal to the
identity near \(\partial\Sigma_{1,2}.\)  By
\cite[Proposition~4.9\textup{(1)}]{FanWhang2024}, \(\widetilde H_i\) is the
Dehn twist about \(\varpi^{-1}(I_i).\)  The assignments
\(\sigma_i\mapsto[\widetilde H_i]\) define a homomorphism}
\begin{equation}\label{eq:braid-lifting-map}
        B_4\longrightarrow\operatorname{PMod}^{+}(\Sigma_{1,2}^{\circ}).
\end{equation}
{The Birman--Hilden theorem gives the following exact sequence}; see
\cite{BirmanHilden1973},
\cite[Sections~9.2 and~9.4.1]{FarbMargalit2012}, and the remark following
\cite[Proposition~4.9]{FanWhang2024}.
\endgroup\begin{proposition}\label{prop:braid-pmod-quotient}
The homomorphism \eqref{eq:braid-lifting-map} induces an exact sequence
\begin{equation}\label{eq:braid-pmod-exact}
 1\longrightarrow Z(B_4)\longrightarrow B_4
 \longrightarrow\operatorname{PMod}^{+}(\Sigma_{1,2}^{\circ})
 \longrightarrow1
\end{equation}
where \(Z(B_4)\) is generated by the full twist
\((\sigma_1\sigma_2\sigma_3)^4.\)
\end{proposition}

{By \cite[Proposition~4.9\textup{(2)}]{FanWhang2024}, the mapping class
\([\widetilde H_i]\) induces the outer automorphism represented by \(\beta_i.\)
Hence, the \(B_4\)-orbits on \(\mathcal X(F_3)\) coincide with the
\(\operatorname{PMod}^{+}(\Sigma_{1,2}^{\circ})\)-orbits.}

Let \(\jmath\colon\Sigma_{1,2}\to\Sigma_{1,2}\) be the deck involution
of \(\varpi.\)  It preserves orientation and interchanges the two
boundary components.  {In the basis \(U,X,Y,\) choose the following representative
\(\jmath_*\in\operatorname{Aut}(F_3)\) of the outer automorphism
induced by \(\jmath\):}
\begin{equation}\label{eq:elliptic-involution-free-group}
 \jmath_*(U)=U^{-1},
 \qquad
 \jmath_*(X)=X^{-1},
 \qquad
 \jmath_*(Y)=Y^{-1}.
\end{equation}
\begingroup
Let \(\iota_{\jmath}\) be the automorphism of \(\mathcal X(F_3)\) induced
by \(\jmath.\)  Substitution of
\eqref{eq:elliptic-involution-free-group} into
\eqref{eq:trace-stokes-matrix}, using
\(\tr(PQ^{-1})=\tr(P)\tr(Q)-\tr(PQ),\) shows that
\(\iota_{\jmath}\) fixes the six trace functions occurring in \(A_\rho\)
and interchanges \(\tr\rho(K_1)\) with \(\tr\rho(K_2)\):
\begin{equation}\label{eq:euler-fricke-deck-involution}
 \iota_{\jmath}([\rho])=[\rho\circ\jmath_*],
 \qquad
 \bigl(\tr\rho(K_1),\tr\rho(K_2)\bigr)
 \longmapsto
 \bigl(\tr\rho(K_2),\tr\rho(K_1)\bigr).
\end{equation}
Equations~\eqref{eq:euler-fricke-deck-involution} and
\eqref{eq:full-euler-fricke-cover} show that \(\iota_{\jmath}\) is the
non-identity involution of \(\mathcal X(F_3)\) over \(\mathcal V(4)\) given
in these coordinates by
\(
 (A,t)\longmapsto(A,\Phi(A)-t).
\)
\endgroup

\begin{proposition}\label{prop:euler-fricke-dictionary}
Let \(A\in\mathcal V(4)\) satisfy
\(\Phi(A)^2=16\) and \(\Psi(A)=8.\)  The fibre
\(\Theta^{-1}(A)\) has a unique closed point, denoted by
\([\rho_A].\)  Every representative \(\rho\) of this character satisfies
\[
 \tr\rho(K_1)=\tr\rho(K_2)
 =\frac{\Phi(A)}2\in\{-2,2\}.
\]
\end{proposition}

\begin{proof}
Under the stated assumptions, the equation of the fibre in
\eqref{eq:full-euler-fricke-cover} is
\[
 t^2-\Phi(A)t+4
 =\left(t-\frac{\Phi(A)}2\right)^2=0.
\]
\begingroup\begingroup
The scheme-theoretic fibre is
\(
 \Theta^{-1}(A)\simeq
 \operatorname{Spec}\!\left(
 \kk[t]/\bigl((t-\Phi(A)/2)^2\bigr)\right).
\)
It has length two and a unique closed point, at which
\(t=\Phi(A)/2.\)
\endgroup\endgroup
Equation~\eqref{eq:boundary-traces} gives the two stated traces.
\end{proof}

The hypersurface \eqref{eq:full-euler-fricke-cover} is defined over
\(\mathbb Q\). For a positive Euler matrix \(A\), the rational point
\((A,\Phi(A)/2)\) defines characters over both \(\kk\) and the complex numbers.
In the arguments below, \([\rho_A]\) denotes the corresponding complex character.

\begingroup\begingroup
\begin{corollary}\label{cor:all-traces-integral}
Let \(A\) be a positive Euler matrix, and let \(\rho\) represent
\([\rho_A].\)  Then \(\tr\rho(\gamma)\in\mathbb Z\) for every
\(\gamma\in F_3.\)
\end{corollary}
\begin{proof}
The six traces in \eqref{eq:trace-stokes-matrix} are the integral entries
of \(A,\) and Proposition~\ref{prop:euler-fricke-dictionary} gives
\(\tr\rho(\alpha_1\alpha_3)=\Phi(A)/2\in\{-2,2\}.\)
By \cite[Theorem~3.1]{Horowitz1972}, the trace of every word in the matrices
\(\rho(\alpha_1),\rho(\alpha_2),\rho(\alpha_3)\) and their inverses is a polynomial
with integer coefficients in these seven traces.
\end{proof}
\endgroup\endgroup

\subsection{Fricke representations}

Let
\(
 \mathbb H=\{x+iy\mid x,y\in\mathbb R,\ y>0\}
\)
be the upper half-plane
{with the Poincar\'e metric}
\(ds^2=(dx^2+dy^2)/y^2.\)  The group
\(\PSL_2(\mathbb R)\) acts on \(\mathbb H\) by orientation-preserving
{isometries}.

\begin{definition}\label{def:fricke-representation}
\begingroup\begingroup
A discrete faithful representation
\[
 \overline\rho\colon\pi_1(\Sigma_{1,2})
 \longrightarrow\PSL_2(\mathbb R)
\]
with image \(\Gamma\) is called a \emph{Fricke representation} if
\(Y=\Gamma\backslash\mathbb H\) has finite area and an orientation-preserving homeomorphism
\(f\colon\Sigma_{1,2}^{\circ}\to Y\) induces \(\overline\rho,\) using
the identification \(\pi_1(Y)\simeq\Gamma\) supplied by the covering
\(\mathbb H\to Y.\) We consider such representations up to conjugation
in \(\PSL_2(\mathbb R).\)
\endgroup\endgroup
\end{definition}

A non-trivial element of
\(\PSL_2(\mathbb R)\) is \emph{parabolic} if it has a unique fixed point on
\(\partial\mathbb H=\mathbb R\cup\{\infty\}.\)  An element
\(P\in\SL_2(\mathbb R)\) projects to a parabolic element if and only if
\(P\ne\pm I\) and \(\tr P=\pm2.\)

{Let \(U,X,Y\) be the basis of
\(\pi_1(\Sigma_{1,2})\) defined in
\eqref{eq:goldman-generators-boundary}.}  For a representation \(\rho,\) put
\[
 u=\tr\rho(U),\quad x=\tr\rho(X),\quad y=\tr\rho(Y),
 \quad v=\tr\rho(UX),\quad w=\tr\rho(UY),\quad z=\tr\rho(XY).
\]
The identity
\(
\tr(RS^{-1})=\tr(R)\tr(S)-\tr(RS)
\)
for \(R,S\in\SL_2\) gives
\begin{equation}\label{eq:goldman-coordinates}
 (u,x,y,v,w,z)=(a,b,f,ab-e,d,c).
\end{equation}
For matrices \(R,S\in\SL_2,\) write
\([R,S]=RSR^{-1}S^{-1}.\)  {The Cayley--Hamilton theorem gives}
\begingroup
\begin{equation}\label{eq:fricke-commutator-identity}
 \tr[R,S]
 =(\tr R)^2+(\tr S)^2+(\tr RS)^2
  -(\tr R)(\tr S)(\tr RS)-2.
\end{equation}
\endgroup
Using
\(
\Delta(r,s,t)=rst-r^2-s^2-t^2+4
,\) equation \eqref{eq:goldman-coordinates} yields
\begin{equation}\label{eq:delta-commutator-traces}
\begin{aligned}
 \tr\rho([U,Y])=2-\Delta(a,d,f),\quad
 \tr\rho([X,Y])=2-\Delta(b,c,f),\quad
 \tr\rho([U,X])=2-\Delta(a,b,e).
\end{aligned}
\end{equation}

\begin{lemma}\label{lem:integral-traces-modular}
{Let \(\widetilde\Gamma\subset\SL_2(\mathbb C)\) be a subgroup
whose natural action on \(\mathbb C^2\) has no invariant line.}
Assume that every element of \(\widetilde\Gamma\) has integral trace and
that \(\widetilde\Gamma\) contains a non-central element of trace
\(\pm2.\)  Then for some \(h\in\SL_2(\mathbb C)\) we have
\(
 h\widetilde\Gamma h^{-1}\subset\SL_2(\mathbb Z).
\)
\end{lemma}

\begin{proof}
Let
\(
 \mathcal O=\mathbb Z[\widetilde\Gamma]
 \subset M_2(\mathbb C)
\)
and
\(
 \mathcal B=\mathbb Q[\widetilde\Gamma].
\)
{By \cite[Proposition~2.2]{Bass1980}, the assumption
\(\tr(\gamma)\in\mathbb Z\) for every \(\gamma\in\widetilde\Gamma\)}
implies that \(\mathcal O\) is an order in the four-dimensional central
simple \(\mathbb Q\)-algebra \(\mathcal B.\)

Choose \(P\in\widetilde\Gamma\) with
\(
 \tr P=2\delta
,\) where \(\delta\in\{\pm1\},\) and \(P\ne\delta I.\)  {The
Cayley--Hamilton theorem gives}
\( P-\delta I\ne 0\) and
\((P-\delta I)^2=0.
\)
Thus, \(\mathcal B\) contains a non-zero nilpotent and is not a division
algebra.  Hence, \(\mathcal B\simeq M_2(\mathbb Q).\)  {Extend scalars from \(\mathbb Q\) to \(\mathbb C\) in a chosen
isomorphism \(\mathcal B\simeq M_2(\mathbb Q).\)  By the Skolem--Noether theorem, {this embedding}
\(\mathcal B\hookrightarrow M_2(\mathbb C)\) is conjugate to the original
inclusion.  After this conjugation, we may assume}
\(\mathcal O\subset M_2(\mathbb Q).\)

Let \(e_1,e_2\) be the basis of \(\mathbb Q^2\) and put
\(
 L=\mathcal Oe_1+\mathcal Oe_2.
\)
{The module \(L\) is a free \(\mathbb Z\)-module of rank two
spanning \(\mathbb Q^2,\) and it is}
preserved by \(\mathcal O.\) A \(\mathbb Z\)-basis of \(L\) gives
\(
 \mathcal O\subset\operatorname{End}_{\mathbb Z}(L)
 \simeq M_2(\mathbb Z).
\)
Since each element of \(\widetilde\Gamma\) has determinant one, this
conjugates \(\widetilde\Gamma\) into \(\SL_2(\mathbb Z).\)  Multiplying
the conjugating matrix by a scalar does not change the conjugation and
makes its determinant equal to one.
\end{proof}

{The \hyperref[app:half-turn-realization]{Appendix} constructs the representation
\(\rho_A^{\mathbb R}\) and proves Theorem~\ref{thm:positive-fricke-character}.}

\begin{theorem}\label{thm:positive-fricke-character}
\begingroup\begingroup
For every positive Euler matrix \(A,\) the character \([\rho_A]\) admits
a representative \(\rho_A^{\mathbb R}\colon F_3\to\SL_2(\mathbb R)\)
whose projectivization is a Fricke representation.  Its quotient is a
complete finite-area hyperbolic surface of genus one with two cusps.
The two peripheral matrices are non-central and satisfy
\[
 \tr\rho_A^{\mathbb R}(K_1)
 =\tr\rho_A^{\mathbb R}(K_2)
 =\frac{\Phi(A)}2\in\{-2,2\}.
\]
Every representative of \([\rho_A]\) is irreducible.
\endgroup\endgroup
\end{theorem}

\begin{theorem}\label{thm:positive-fuchsian-group}
\begingroup\begingroup
For a positive Euler matrix \(A,\) there is a unique Fricke representation
\[
 \overline\rho_A\colon
 \pi_1(\Sigma_{1,2})\longrightarrow\PSL_2(\mathbb R),
\]
up to \(\PSL_2(\mathbb R)\)-conjugation, admitting an
\(\SL_2(\mathbb R)\)-lift \(\rho\) with \(A_\rho=A.\)
Consequently its image \(\Gamma_A\) is determined up to the same conjugation.
\endgroup\endgroup
\end{theorem}

\begin{proof}
\begingroup\begingroup
Existence follows from Theorem~\ref{thm:positive-fricke-character}.  Let
\(\rho,\rho'\colon\pi_1(\Sigma_{1,2})\to\SL_2(\mathbb R)\)
be lifts of two Fricke representations satisfying
\(A_\rho=A_{\rho'}=A.\)  Choose orientation-preserving homeomorphisms
\(f\colon\Sigma_{1,2}^{\circ}\to Y\) and
\(f'\colon\Sigma_{1,2}^{\circ}\to Y'\) inducing the projectivizations of
\(\rho\) and \(\rho',\) respectively, on fundamental groups.
Proposition~\ref{prop:euler-fricke-dictionary} gives
\([\rho]=[\rho_A]=[\rho'].\)

Both representations are irreducible, so
\(\rho'=g\rho g^{-1}\) for some \(g\in\GL_2(\mathbb C).\)
Because \(\rho\) and \(\rho'\) take values in
\(\SL_2(\mathbb R),\) the matrix \(g^{-1}\overline g\) centralizes
\(\rho(F_3).\)  Irreducibility and Schur's lemma give
\(g^{-1}\overline g=\lambda I.\)  Then \(\lambda\overline\lambda=1;\)
choose \(\mu\in\mathbb C^*\) with \(\mu/\overline\mu=\lambda.\)
Replacing \(g\) by \(\mu g,\) we may take
\(g\in\GL_2(\mathbb R).\)
Write \(g=\begin{psmallmatrix}a&b\\c&d\end{psmallmatrix}.\)
If \(\det g>0,\) set
\(\widetilde g(z)=(az+b)/(cz+d);\) if \(\det g<0,\) set
\(\widetilde g(z)=(a\overline z+b)/(c\overline z+d).\)
In either case \(\widetilde g\) is an isometry of \(\mathbb H\) conjugating
\(\overline\rho(F_3)\) to \(\overline\rho'(F_3).\)
The isometry \(\widetilde g\) descends to an isometry
\(h\colon Y\to Y'.\) For each \(i=1,2,\) it sends the cusp
corresponding to \(\overline\rho(K_i)\) to the cusp corresponding to
\(\overline\rho'(K_i).\)
The equivariance of \(\widetilde g\) implies that
\(h_*\circ f_*\) and \(f'_*\) differ by an inner automorphism of
\(\pi_1(Y');\) here the induced homomorphisms are formed after
choosing basepoints and connecting paths.
Consequently,
\(
 \varphi=(f')^{-1}\circ h\circ f
 \quad\text{satisfies}\quad
 [\varphi_*]=1\in\operatorname{Out}
       (\pi_1(\Sigma_{1,2}^{\circ},*)).
\)
Moreover, \(\varphi\) extends to a homeomorphism
\(\widehat\varphi\colon\widehat\Sigma\to\widehat\Sigma\)
satisfying \(\widehat\varphi(p_i)=p_i\) for \(i=1,2.\)
By the injectivity in the extended Dehn--Nielsen--Baer theorem
\cite[Theorem~8.8]{FarbMargalit2012}, there is an isotopy
\(H_t\colon\widehat\Sigma\to\widehat\Sigma,\) \(0\le t\le1,\)
such that
\[
 H_0=\widehat\varphi,\qquad H_1=\operatorname{id}_{\widehat\Sigma},
 \qquad H_t(p_i)=p_i, \quad i=1,2,\quad 0\le t\le1.
\]
Thus, \(\varphi\) and, hence, \(h,\) preserves orientation, so
\(\det g>0.\) Multiplying \(g\) by the positive scalar
\((\det g)^{-1/2}\) puts it in \(\SL_2(\mathbb R).\)  Therefore, the two
Fricke representations are conjugate in \(\PSL_2(\mathbb R).\)
\endgroup\endgroup
\end{proof}

\subsection{{An integral conjugate}}

\begingroup\begingroup
Because \(F_3\) is free, every homomorphism
\(\overline\rho\colon F_3\to\PSL_2(\kk)\) lifts to
\(\SL_2(\kk).\)  For two lifts of the same \(\overline\rho,\)
there is a homomorphism \(\eta\colon F_3\to\{\pm I\}\) such that
\(\rho'(\gamma)=\eta(\gamma)\rho(\gamma)\) for every \(\gamma\in F_3.\)
More generally, suppose that the projectivizations of \(\rho\) and
\(\rho'\) are conjugate. Choose \(g\in\SL_2(\kk)\) such that
\(\overline\rho'=[g]\overline\rho[g]^{-1}.\) There is then a homomorphism
\(\eta\colon F_3\to\{\pm I\}\) such that
\(
 \rho'(\gamma)=\eta(\gamma)g\rho(\gamma)g^{-1}
\) for \(\gamma\in F_3.
\)
If both \(A_\rho\) and \(A_{\rho'}\) are positive, the positive traces
of \(\rho(\alpha_i)\) and \(\rho'(\alpha_i)\) give \(\eta(\alpha_i)=I\) for
\(i=1,2,3.\) Hence, \(\eta\) is trivial and the two
\(\SL_2(\kk)\)-characters coincide.
\endgroup\endgroup

\begin{corollary}\label{cor:positive-modular-group}
Let \(A\) be a positive Euler matrix.  The image of \(\overline\rho_A\) is conjugate in
\(\PSL_2(\mathbb R)\) to a subgroup of \(\PSL_2(\mathbb Z).\)
\end{corollary}

\begin{proof}
\begingroup\begingroup
The real representative \(\rho^{\mathbb R}\) from
Theorem~\ref{thm:positive-fricke-character} is irreducible and has a
non-central peripheral matrix of trace \(\pm2.\)  By
Corollary~\ref{cor:all-traces-integral} and
Lemma~\ref{lem:integral-traces-modular}, there is \(g\in\GL_2(\mathbb C)\)
such that
\(
 \rho^{\mathbb Z}:=g\rho^{\mathbb R}g^{-1}
 \)
 and
 \(
 \rho^{\mathbb Z}(F_3)\subset\SL_2(\mathbb Z).
\)
As in the proof of Theorem~\ref{thm:positive-fuchsian-group}, a scalar
multiple of \(g\) is real.
If \(\det g<0,\) replace \(g\) by \(\operatorname{diag}(-1,1)g;\)
this preserves integrality.  Scaling by \((\det g)^{-1/2}\) now gives a
conjugating matrix in \(\SL_2(\mathbb R).\)
\endgroup\endgroup
\end{proof}

\subsection{Modular images and mutation orbits}

For a positive Euler matrix \(A,\) we conjugate the representation
\(\overline\rho_A,\) as in Corollary~\ref{cor:positive-modular-group}, so
that its image lies in \(\PSL_2(\mathbb Z).\)  {We call this
representation} a \emph{modular representative} of \(A,\) and
its image \(\Gamma\subset\PSL_2(\mathbb Z)\) is called a \emph{modular
image} of \(A.\)

\begin{proposition}\label{prop:index-twelve}
Every modular image \(\Gamma\) of a positive Euler matrix is torsion-free,
has index twelve in \(\PSL_2(\mathbb Z),\) and the quotient
\(\Gamma\backslash\mathbb H\) is a genus-one surface with two cusps.
\end{proposition}

\begin{proof}
Theorem~\ref{thm:positive-fuchsian-group} shows that \(\Gamma\) is
torsion-free and that its quotient has genus one and two cusps.
The Poincar\'e metric fixed above has curvature \(-1.\)
Gauss--Bonnet theorem therefore gives the area of a complete finite-area
hyperbolic surface of genus \(g\) with \(n\) cusps as
\(
2\pi(2g-2+n)
,\) so
\[
 \Area(\Gamma\backslash\mathbb H)=2\pi(2\cdot1-2+2)=4\pi.
\]
The modular orbifold
\(\PSL_2(\mathbb Z)\backslash\mathbb H\) has genus zero, one cusp, and
two orbifold points with stabilizers of orders \(2\) and \(3.\)  Its orbifold Euler
characteristic is therefore
\(2-1-(1-1/2)-(1-1/3)=-1/6,\) and its area is \(\pi/3.\)
The orbifold covering
\(
\Gamma\backslash\mathbb H\longrightarrow
\PSL_2(\mathbb Z)\backslash\mathbb H
\)
has degree \([\PSL_2(\mathbb Z):\Gamma];\) integration of the
area form over this covering gives
\(
 [\PSL_2(\mathbb Z):\Gamma]=\frac{4\pi}{\pi/3}=12.
\)
\end{proof}

Let \(\mathscr M\) be the set of \(\PSL_2(\mathbb Z)\)-conjugacy classes
of torsion-free index-twelve subgroups
\(\Gamma\subset\PSL_2(\mathbb Z)\) for which
\(\Gamma\backslash\mathbb H\) has genus one and two cusps.
{Choosing one modular image for each \(B_4\)-orbit of positive Euler
matrices and taking its \(\PSL_2(\mathbb Z)\)-conjugacy class defines a map
from the set of these orbits to \(\mathscr M.\)
Proposition~\ref{prop:modular-orbit-injectivity} shows that this map is injective.}

\begin{proposition}\label{prop:modular-orbit-injectivity}
{Let \(A,A'\) be positive Euler matrices.  If they have modular
images \(\Gamma,\Gamma'\subset\PSL_2(\mathbb Z)\) that are conjugate in
\(\PSL_2(\mathbb Z),\) then \(A\) and \(A'\) are mutation-equivalent.}
\end{proposition}

\begin{proof}
Let \(A,A'\) be positive.  Choose modular representatives
\(\overline\rho,\overline\rho'\) and \(\SL_2(\mathbb R)\)-lifts
\(\rho,\rho'\) so that
\(A_\rho=A,\) \(A_{\rho'}=A',\) and, after conjugation, both projective
images equal \(\Gamma.\)  The two representations
give orientation-preserving homeomorphisms
\(
 f,f'\colon\Sigma_{1,2}^{\circ}\longrightarrow
 \Gamma\backslash\mathbb H.
\)
\begingroup\begingroup
The homeomorphism \((f')^{-1}\circ f\) preserves orientation.
If its extension to \(\widehat\Sigma\) interchanges \(p_1\) and \(p_2,\)
replace \((f',\rho')\) by
\((f'\circ\jmath^{\circ},\rho'\circ\jmath_*),\) where
\(\jmath^{\circ}=\jmath|_{\Sigma_{1,2}^{\circ}}.\)
The extension of \(\jmath^{\circ}\) interchanges \(p_1\) and \(p_2,\) and
\eqref{eq:euler-fricke-deck-involution} gives
\(A_{\rho'\circ\jmath_*}=A_{\rho'}=A'.\)
After this replacement, the extension of \((f')^{-1}\circ f\)
to \(\widehat\Sigma\) fixes \(p_1\) and \(p_2\) individually.
Its isotopy class therefore belongs to
\(\operatorname{PMod}^{+}(\Sigma_{1,2}^{\circ}).\)
\endgroup\endgroup

\begingroup\begingroup
Choose \(W\in B_4\) such that the projectivizations of
\(\rho_W:=W\cdot\rho\) and \(\rho'\) are conjugate in
\(\PSL_2(\mathbb R).\)
Equation~\eqref{eq:trace-braid-equivariance} gives \(A_{\rho_W}=W(A),\)
and Proposition~\ref{prop:mutation-positivity} shows that \(W(A)\) is
positive. Choose \(g\in\SL_2(\mathbb R)\) such that
\(\overline\rho'=[g]\overline\rho_W[g]^{-1}.\) There is a homomorphism
\(\eta\colon F_3\to\{\pm I\}\) satisfying
\(
 \rho'(\gamma)=\eta(\gamma)g\rho_W(\gamma)g^{-1}
 \)
for \(\gamma\in F_3.
\)
Positivity of \(W(A)\) and \(A'\) gives \(\eta(\alpha_i)=I\) for
\(i=1,2,3,\) so \(\eta\) is trivial. Hence
\(A_{\rho_W}=A_{\rho'},\) that is, \(W(A)=A'.\)
\endgroup\endgroup
\end{proof}

Section~6 of \cite{SebbarBesrour2021} enumerates the
torsion-free subgroups of index twelve in \(\PSL_2(\mathbb Z)\) and gives
exactly five \(\PSL_2(\mathbb Z)\)-conjugacy classes whose quotients have
genus one and two cusps.
Proposition~\ref{prop:modular-orbit-injectivity} therefore gives the bound
\begin{equation}\label{eq:at-most-five-positive-orbits}
 \#\{\text{mutation orbits of positive Euler matrices}\}\le5.
\end{equation}

\begin{samepage}
\begin{theorem}\label{thm:main-reduction}
\begingroup\begingroup
The positive Euler matrices form exactly five mutation-equivalence classes,
and the Euler degree \(D(A)\) distinguishes them. For every positive Euler
matrix there are unique positive integers \(s,\iota,\) with \(s\)
square-free, such that \(D(A)=2s\iota^2.\) The table gives the Euler
degree, the integers \(s,\iota,\) a representative, and its primitive
radical vector:
\begin{equation}\label{eq:five-positive-representatives}
\begin{array}{c|c|c|c|c}
D(A)&s&\iota&(a,b,c,d,e,f)&p_D\\
\hline
22&11&1&(4,4,7,7,3,20)&(-3,2,-3,1)^t\\
40&5&2&(5,5,5,5,3,18)&(-2,1,-2,1)^t\\
54&3&3&(4,5,5,4,4,11)&(-2,1,-1,1)^t\\
64&2&4&(4,4,4,4,6,6)&(-1,1,-1,1)^t\\
72&1&6&(3,6,3,6,7,7)&(-1,1,-1,1)^t.
\end{array}
\end{equation}
Moreover, every positive Euler matrix \(A\) satisfies \(\Phi(A)=-4,\)
and every representative \(\rho\) of \([\rho_A]\) satisfies
\(
\tr\rho(K_1)=\tr\rho(K_2)=-2.
\)
\endgroup\endgroup
\end{theorem}
\end{samepage}
\begin{proof}
{For each row of \eqref{eq:five-positive-representatives}, let
\(A_D\) be the Euler matrix with the given parameters, let \(p_D\) be
the given primitive vector, and put \(B_D=A_D+A_D^t.\)}
Direct computation gives
\begin{equation}\label{eq:five-representative-identities}
 \Phi(A_D)=-4,
 \qquad
 \Psi(A_D)=8,
 \qquad
 B_Dp_D=0,
 \qquad
 \adj(B_D)=D\,p_Dp_D^t.
\end{equation}
\begingroup\begingroup
All parameters of \(A_D\) are at least \(3.\)  The last identity in
\eqref{eq:five-representative-identities} and the alternating signs of
\(p_D\) give
\[
 D_{ij}(B_D)=(-1)^{i+j}D(p_D)_i(p_D)_j>0.
\]
In particular \(\operatorname{rank}B_D=3.\)  Together with
\(\Phi(A_D)^2=16\) and \(\Psi(A_D)=8,\) the equation
\eqref{eq:coxeter-characteristic} gives the characteristic polynomial
\((\lambda-1)^4,\) while \eqref{eq:fixed-subspace-radical} gives
\(\operatorname{rank}(C-I)=3.\)  Thus, \(C\) has one Jordan block and
satisfies \(\mathrm{(MU)}.\)
Condition \textup{(iv)} of Definition~\ref{def:positive-euler-matrix}
now follows from Proposition~\ref{prop:strict-principal-deltas}.
Hence, \(A_D\) is positive.
\endgroup\endgroup
The equality
\(
\adj(B_D)=D\,p_Dp_D^t
\)
and primitivity of \(p_D\) show, by the definition
\eqref{eq:euler-degree}, that \(D(A_D)=D.\)  Euler degree is invariant under
mutations by Proposition~\ref{prop:euler-degree-arithmetic} and
Lemma~\ref{lem:formal-mutation-invariants}.  The five matrices therefore
belong to five distinct mutation orbits.  Together with
\eqref{eq:at-most-five-positive-orbits}, this shows that these are all the
orbits, and the
five distinct values of \(D(A)\) distinguish them.  Finally, let \(A\) be
any positive Euler matrix.  The classification just proved gives a word
\(W\in B_4\) and one of the five matrices \(A_D\) such that
\(W(A)=A_D.\)  Lemma~\ref{lem:formal-mutation-invariants} and
\eqref{eq:five-representative-identities} give
\(
 \Phi(A)=\Phi(W(A))=\Phi(A_D)=-4.
\)
The trace equalities follow from
Proposition~\ref{prop:euler-fricke-dictionary}.
\end{proof}

Theorem~\ref{thm:main-reduction} proves the classification statement of
Theorem~\ref{thm:intro-classification} from the Introduction.
Minifolds realizing the five mutation classes are described in
Subsections~\ref{subsec:commutative-minifolds} and~\ref{subsec:m72-construction}.
\section{Cusp widths and Euler degrees}
\label{sec:further-modular-data}
{In this section we study properties of the five conjugacy
classes of modular subgroups attached to the matrices in
Theorem~\ref{thm:main-reduction}.}

\begingroup
Let \({\mathsf s},{\mathsf u}\in\PSL_2(\mathbb Z)\) be
\(\begin{psmallmatrix}0&-1\\1&0\end{psmallmatrix}\) and
\(\begin{psmallmatrix}1&1\\0&1\end{psmallmatrix},\) respectively, and
put \({\mathsf r}={\mathsf s}{\mathsf u}.\)  Then
\[
 \PSL_2(\mathbb Z)=\langle {\mathsf s},{\mathsf r}\mid
 {\mathsf s}^2={\mathsf r}^3=1\rangle,
 \qquad {\mathsf s}{\mathsf r}={\mathsf u}.
\]
\endgroup
A cusp of a finite-index subgroup
\(\Gamma\subset\PSL_2(\mathbb Z)\) is a
\(\Gamma\)-orbit in \(\mathbb P^1(\mathbb Q).\)  If a cusp is
represented by \(g\infty\) for \(g\in\PSL_2(\mathbb Z),\) its
\emph{width} is
 \begin{equation}\label{eq:cusp-width-definition}
 w_\Gamma(g\infty)
 =\min\{m\ge1\mid g{\mathsf u}^mg^{-1}\in\Gamma\}.
\end{equation}
This integer is independent of the choice of \(g.\)  Equivalently, for
\(c=g\infty,\) there is an equality
\(
 w_\Gamma(c)
 =[\operatorname{Stab}_{\PSL_2(\mathbb Z)}(c):
   \operatorname{Stab}_\Gamma(c)],
\)
where \(\operatorname{Stab}\) denotes the point stabilizer.

{For a torsion-free subgroup \(\Gamma\) of index twelve, the group
\(\PSL_2(\mathbb Z)\) acts from the right on the twelve cosets in
\(\Gamma\backslash\PSL_2(\mathbb Z).\)}  {Right multiplication by
\(\mathsf s\) and \(\mathsf r\) gives permutations of cycle types \(2^6\) and \(3^4,\) respectively,
which together generate a transitive action.  The cycles of right
multiplication by \(\mathsf s\mathsf r=\mathsf u\)
correspond to the cusps, and}
their lengths are the cusp widths defined in
\eqref{eq:cusp-width-definition}.  In particular, the sum of the cusp
widths is equal to $12$.  Among these groups, those with genus-one quotient
and two cusps form five \(\PSL_2(\mathbb Z)\)-conjugacy classes,
containing twenty-eight subgroups in total.  The five
classes have the unordered pairs of cusp widths
\cite[Section~6]{SebbarBesrour2021}:
\begin{equation}\label{eq:index-twelve-cusp-width-pairs}
        (1,11),\qquad(2,10),\qquad(3,9),\qquad(4,8),
        \qquad(6,6).
\end{equation}

\begingroup\begingroup
For a positive Euler matrix \(A,\) let \(s,\iota\) be the integers in
Theorem~\ref{thm:main-reduction}.
\endgroup\endgroup

\begin{proposition}\label{prop:cusp-width-euler-degree}
Let \(A\) be a positive Euler matrix, and let
\(w_1\le w_2\) be the two cusp widths of the image of a modular
representative of \(A.\)  Then
\begin{equation}\label{eq:euler-degree-cusp-widths}
 D(A)=2w_1w_2,
 \qquad
 \iota=\gcd(w_1,w_2),
 \qquad
 {s=\frac{w_1w_2}{\gcd(w_1,w_2)^2}}{.}
\end{equation}
For the pairs in
\eqref{eq:index-twelve-cusp-width-pairs}, \(w_1\mid w_2;\) hence
\({(w_1,w_2)=(\iota,s\iota)}.\)
{The five possibilities are listed below; the normalizers are given
up to conjugacy in \(\PSL_2(\mathbb Z).\)}
\begingroup\begingroup
\begin{center}
\small
\normalfont
\setlength{\tabcolsep}{4pt}
\begin{tabular}{@{}ccccccc@{}}
\toprule
\(D(A)\) & \((s,\iota)\) & cusp widths & level & \(N(\Gamma)\)
 & \(N(\Gamma)/\Gamma\) & \(\#\text{conjugates}\)\\
\midrule
22 & \((11,1)\) & \((1,11)\) & 11
   & \(\Gamma_0(11)\) & \(1\) & 12\\
40 & \((5,2)\) & \((2,10)\) & 10
   & \(\Gamma_0(5)\) & \(\mathbb Z/2\mathbb Z\) & 6\\
54 & \((3,3)\) & \((3,9)\) & 9
   & \(\Gamma_0(3)\) & \(\mathbb Z/3\mathbb Z\) & 4\\
64 & \((2,4)\) & \((4,8)\) & 8
   & \(\Gamma_0(2)\) & \(\mathbb Z/4\mathbb Z\) & 3\\
72 & \((1,6)\) & \((6,6)\) & 6
   & \(\Gamma_{(3)}\) & \((\mathbb Z/2\mathbb Z)\times (\mathbb Z/2\mathbb Z)\) & 3\\
\bottomrule
\end{tabular}
\end{center}
\endgroup\endgroup

Here \(\Gamma_0(N)\) is the subgroup represented by
matrices whose lower-left entry is divisible by \(N.\)  The level is
the least \(N\) for which the group contains the principal congruence
subgroup \(\Gamma(N),\) and
\(\Gamma_{(3)}=\ker(\PSL_2(\mathbb Z)\longrightarrow \mathbb Z/3\mathbb Z),\) where
\(\mathbb Z/3\mathbb Z\) is written additively and \({\mathsf s}\mapsto0,\)
\({\mathsf r}\mapsto1,\) and \(N(\Gamma)\) is the normalizer of
\(\Gamma\) in \(\PSL_2(\mathbb Z).\)
The last column gives the number of subgroups {in the conjugacy class
in that row}.
\end{proposition}

\begin{proof}
We first match the five Euler matrices in
Theorem~\ref{thm:main-reduction} with the five cusp-width pairs.  For the
matrix \(A_D,\) define a representation \(\rho_D\) by putting
\(
P_D=\rho_D(\alpha_1),
Q_D=\rho_D(\alpha_2)
,\) and \(R_D=\rho_D(\alpha_3).\)  {The five triples below belong to}
\(\SL_2(\mathbb Z)^3\):
\[
\begin{array}{c|ccc}
D&P_D&Q_D&R_D\\
\hline
22&
\begin{psmallmatrix}4&-1\\1&0\end{psmallmatrix}&
\begin{psmallmatrix}1&1\\2&3\end{psmallmatrix}&
\begin{psmallmatrix}0&-1\\1&7\end{psmallmatrix}\\[2mm]
40&
\begin{psmallmatrix}5&-1\\1&0\end{psmallmatrix}&
\begin{psmallmatrix}1&1\\3&4\end{psmallmatrix}&
\begin{psmallmatrix}0&-1\\1&5\end{psmallmatrix}\\[2mm]
54&
\begin{psmallmatrix}4&-1\\1&0\end{psmallmatrix}&
\begin{psmallmatrix}2&1\\5&3\end{psmallmatrix}&
\begin{psmallmatrix}0&-1\\1&5\end{psmallmatrix}\\[2mm]
64&
\begin{psmallmatrix}4&-1\\1&0\end{psmallmatrix}&
\begin{psmallmatrix}2&1\\3&2\end{psmallmatrix}&
\begin{psmallmatrix}0&-1\\1&4\end{psmallmatrix}\\[2mm]
72&
\begin{psmallmatrix}3&-1\\1&0\end{psmallmatrix}&
\begin{psmallmatrix}3&2\\4&3\end{psmallmatrix}&
\begin{psmallmatrix}0&-1\\1&3\end{psmallmatrix}.
\end{array}
\]
{Let \(\overline\rho_D\) be the projectivization of \(\rho_D,\) and
put
\(
 \Gamma_D=\overline\rho_D(F_3)\subset\PSL_2(\mathbb Z).
\)}
For each \(\rho_D\),
\(A_{\rho_D}=A_D.\)  Hence, Proposition~\ref{prop:euler-fricke-dictionary}
identifies each triple with a representative of \([\rho_{A_D}].\)
\begingroup\begingroup
Since \(A_D\) is positive, Theorem~\ref{thm:positive-fricke-character}
implies that the same character has a Fricke representative and that every
representative of it is irreducible.  Let
\(\rho_D^{\mathrm F}\colon F_3\to\SL_2(\mathbb R)\) be a Fricke
representative of this character.  Since \(\rho_D\) and
\(\rho_D^{\mathrm F}\) are irreducible and have the same character,
there exists \(g_D\in\SL_2(\mathbb C)\) such that
\(\rho_D=g_D\rho_D^{\mathrm F}g_D^{-1}.\)
As in the proof of Theorem~\ref{thm:positive-fuchsian-group}, a scalar multiple
of \(g_D\) is real, so the corresponding isometry identifies the two quotients
of \(\mathbb H\).
Their projectivizations therefore have the same kernel, so
\(\overline\rho_D\) is faithful.
\endgroup\endgroup

{By \eqref{eq:goldman-generators-boundary}, the images under
\(\rho_D\) of the two primitive cusp loops are}
\begingroup\begingroup\begingroup
\[
 K_{1,D}:=\rho_D(K_1)=P_DR_D,
 \qquad
 K_{2,D}:=\rho_D(K_2)=P_DQ_DR_DQ_D^{-1}.
\]
\endgroup\endgroup\endgroup
Theorem~\ref{thm:main-reduction} gives
\(\tr K_{1,D}=\tr K_{2,D}=-2.\)
\begingroup\begingroup
Let \(c_{i,D}\in\mathbb P^1(\mathbb Q)\) be the fixed point of
the projective class \([K_{i,D}].\)  The conjugacy by
\([g_D]\in\PSL_2(\mathbb C)\) and the fact that each \(K_i\) goes once
around the corresponding puncture give
\(
 \operatorname{Stab}_{\Gamma_D}(c_{i,D})
 =\langle[K_{i,D}]\rangle ,\)
 where \(i=1,2.\)
More generally, let \(\Gamma\subset\PSL_2(\mathbb Z)\) be a subgroup
of finite index, let \(c\in\mathbb P^1(\mathbb Q),\) and suppose that
\(M\in\SL_2(\mathbb Z)\) has trace \(-2\) and that \([M]\) generates
\(\operatorname{Stab}_{\Gamma}(c).\)  Then
\(w_\Gamma(c)\) is the greatest common divisor of the entries of \(M+I.\)
\endgroup\endgroup
Indeed, if this width is \(w,\) then
\[
 M=-h
 \begin{pmatrix}1&\pm w\\0&1\end{pmatrix}
 h^{-1}
 \qquad\text{for some }h\in\SL_2(\mathbb Z),
\]
and left and right multiplication by unimodular matrices preserves the
greatest common divisor of the entries.  For
\(D=22,40,54,64,72,\) respectively, put
\((w_1,w_2)=(1,11),(2,10),(3,9),(4,8),(6,6).\)
Direct multiplication now gives
\[
 K_{1,D}=-\begin{pmatrix}1&w_2\\0&1\end{pmatrix},
 \qquad
 \gcd_{i,j}\bigl|(K_{2,D}+I)_{ij}\bigr|=w_1.
\]
{Thus, the five matrices \(A_D\) in
Theorem~\ref{thm:main-reduction} have the five cusp-width pairs in
\eqref{eq:index-twelve-cusp-width-pairs}.}  By
\eqref{eq:character-braid-action}, the \(B_4\)-action precomposes a
representation with an automorphism of \(F_3,\) so it does not change its
image.  By
Proposition~\ref{prop:modular-orbit-injectivity}, distinct mutation
orbits determine disjoint sets of
{
\(\PSL_2(\mathbb Z)\)-conjugacy classes of subgroups
\(\overline\rho(F_3),\) where \(\overline\rho\) is a modular representative.}
There are five
mutation orbits and five such conjugacy classes, and the triples
\((P_D,Q_D,R_D),\) for \(D=22,40,54,64,72,\) {give representatives
of all five}.  {Hence, the modular images of any two matrices in the
same mutation orbit are conjugate in \(\PSL_2(\mathbb Z),\) and images from
different mutation orbits are not conjugate.}

The equalities in \eqref{eq:euler-degree-cusp-widths}, as well as the
equivalent ordered form of the cusp-width pair, follow from
\eqref{eq:index-twelve-cusp-width-pairs} and the uniqueness of the
{square-free decomposition \(D(A)=2s\iota^2.\)}
{The chains of supergroups in
\cite[Table~2, rows \(11A^1,10A^1,9A^1,8A^1,6B^1\)]{CumminsPauli2003},
together with the standard names in \cite[Table~4]{CumminsPauli2003},
identify the normalizers} {up to conjugacy.}  {The \(L\)-column of Table~2 gives the
last column; equivalently, these numbers follow from the orbit--stabilizer
formula.}
\end{proof}

{For the four commutative three-dimensional minifolds in
Theorem~\ref{thm:threefold-minifolds}, the smaller cusp width \(w_1\) in
\eqref{eq:index-twelve-cusp-width-pairs} is the Fano index: it is
\(1,2,3,4\) for \(X_{22},V_5,Q^3,\mathbb P^3,\) respectively.}

\begin{remark}\label{rem:golyshev-modularity}
For these four minifolds, put
\(
        d=w_1=\iota,
        \;
        {N=\frac{w_2}{w_1}=s}.
\)
In the order \(X_{22}, V_5, Q^3, \mathbb P^3,\) one obtains
\[
        (N,d)=(11,1),(5,2),(3,3),(2,4),
        \qquad
        (-K_X)^3=D(A_X)=2d^2N.
\]
{Golyshev assigns these four pairs \((N,d)\) to the Fano
minifolds and constructs their quantum differential equations from the
modular curves \(X_0(N)/W_N\){, see}
\cite[Sections~3, 6.2, and~6.5--6.7]{Golyshev2007}.  If
\(Q=e^{2\pi i\tau}\) is the parameter used in this construction, the eta-product}
formula of \cite[Sections~5.5--5.6]{Golyshev2007} becomes
\begin{equation}\label{eq:golyshev-eta-product}
 I_{N,d}(\tau)
 =\eta(d\tau)^2\eta(Nd\tau)^2
 =\eta(w_1\tau)^2\eta(w_2\tau)^2.
\end{equation}
{Here \(N\) is the integer in the notation \(X_0(N),\) whereas
the index-twelve subgroup with cusp widths \((d,Nd)\) has level \(Nd.\)}

For the same four threefolds, Golyshev identifies the symmetrized Euler
pairing of the exceptional collection with the intersection pairing of
the {vanishing cycles in the modular Picard--Fuchs system for that
threefold}
\cite[Section~2.2 and Sections~3--6]{Golyshev2008}.
The identities \(w_1=d\) and \(w_2=Nd\) give the same pairs
\((N,d).\)  Substituting \((N,d)=(1,6)\) into
\eqref{eq:golyshev-eta-product} gives \(\eta(6\tau)^4,\) a modular form of
level \(36.\)
For the subsequent mirror-theoretic approach to
the classification of Fano manifolds, see \cite{CoatesEtAl2013}.
\end{remark}

{For a subgroup \(\Gamma\) in any of the five conjugacy classes
listed in Proposition~\ref{prop:cusp-width-euler-degree}, the number of its
conjugates in \(\PSL_2(\mathbb Z)\) is}
\begin{equation}\label{eq:twenty-eight-subgroups}
 \#\{g\Gamma g^{-1}\mid g\in\PSL_2(\mathbb Z)\}
 =[\PSL_2(\mathbb Z):N(\Gamma)]
 =\frac{[\PSL_2(\mathbb Z):\Gamma]}{[N(\Gamma):\Gamma]}
 =\frac{12}{|N(\Gamma)/\Gamma|}.
\end{equation}
Consequently, the twenty-eight subgroups enumerated by Sebbar and Besrour
decompose into the five conjugacy classes as
\(
        28=12+6+4+3+3.
\)
\begingroup\begingroup
Under the right action of \(\PSL_2(\mathbb Z)\) on
\(\Gamma\backslash\PSL_2(\mathbb Z),\) the stabilizer of
\(\Gamma g\) is \(g^{-1}\Gamma g.\) Hence, the transitive coset action,
without a distinguished coset, determines the
\(\PSL_2(\mathbb Z)\)-conjugacy class of \(\Gamma.\)
\endgroup\endgroup
\section{{Non-commutative deformations and the fifth minifold}}
\label{sec:noncommutative-consequences}

\subsection{Non-commutative deformations of commutative varieties}


\begingroup\begingroup
Pym \cite[Theorem~1.1 and Table~1]{Pym2015}
classifies flat graded Calabi--Yau deformations of the polynomial
algebra \(\kk[x_0,x_1,x_2,x_3]\) into six irreducible families,
denoted by
\(
  \mathsf L(1,1,1,1),\;
  \mathsf L(1,1,2),\;
  \mathsf R(2,2),\;
  \mathsf R(1,3),\;
  \mathsf S(2,3),\;
  \mathsf E(3).
\)
Their generic members are, respectively, Calabi--Yau skew-polynomial
algebras, generalized Laurent polynomial algebras, four-dimensional
Sklyanin algebras, central extensions of three-dimensional Sklyanin
algebras, graded Ore extensions of \(\kk[x_1,x_2,x_3]\) by
divergence-free derivations, and the
\(\operatorname{Aff}(\kk)\)-quantization of
\(\operatorname{Sym}^3(\mathbb P^1);\) see
\cite[Sections~3.1--3.6]{Pym2015}.  The four-dimensional
Sklyanin algebras were introduced in \cite{Sklyanin1982} and were proved
in \cite{SmithStafford1992} to be regular and noetherian and to have
Hilbert series \((1-t)^{-4}.\)
\endgroup\endgroup

\begingroup\begingroup
For \(X\in\{Q^3,V_5\},\) the \(\operatorname{Aff}(\kk)\)-actions described in
\cite[Example~5.1 and Theorems~5.2 and~7.1]{LorayPereiraTouzet2013}
lift to \(\mathcal O_X(1).\) The deformation formula of
\cite{CollGerstenhaberGiaquinto1989} then gives flat graded non-commutative
deformations of the homogeneous coordinate algebra of \(X,\) see
\cite[Lemmas~3.3 and~4.1]{Pym2015}.
For \(Q^3,\) the quantum flag construction gives another graded
\(q\)-deformation, see \cite{LakshmibaiReshetikhin1992}.
Applying the deformation formula to the \(\operatorname{Aff}(\kk)\)-action
on the Mukai--Umemura threefold gives a further example, see
\cite[Example~4.4]{Pym2015}.
We do not prove that these deformations satisfy the minifold conditions
of Definition~\ref{def:orlov-minifold}.
\endgroup\endgroup

\subsection{The construction of minifolds of degree \texorpdfstring{$72$}{72}}
\label{subsec:m72-construction}

  Let
\(V,\overline V,W,\overline W\) be three-dimensional vector spaces over
\(\kk.\) Denote by
\(
 U=\operatorname{Sym}^2W,
\) and
\(
 \overline U=\operatorname{Sym}^2\overline W
\)
the symmetric squares.
Choose bases \(x_1,x_2,x_3\) of \(V,\)
\(\overline x_1,\overline x_2,\overline x_3\) of \(\overline V,\) and
\(u^{ij}=u^{ji},\) \(\overline u^{ij}=\overline u^{ji}\) of
\(U,\) \(\overline U,\) respectively, with \(1\le i\le j\le3.\)
Fix \(p,q\in\kk^\times\) and suppress tensor-product signs in monomials.
With \((i,j,k)=\sigma(1,2,3)\) in each sum, define
\(\mathsf{\Phi}_{p,q}\in V\otimes U\otimes\overline V\otimes\overline U\) by

\begin{equation}\label{eq:m72-potential}
\begin{aligned}
\mathsf{\Phi}_{p,q}={}&
 \sum_{\sigma\in\mathfrak S_3}
 x_i u^{jj}\overline x_i\overline u^{kk}
-\sum_{\substack{\sigma\in\mathfrak S_3\\
                   \operatorname{sgn}(\sigma)=1}}
 p x_i u^{ij}\overline x_j\overline u^{kk}
-\sum_{\substack{\sigma\in\mathfrak S_3\\
                   \operatorname{sgn}(\sigma)=-1}}
 q x_i u^{ij}\overline x_j\overline u^{kk}
\\
&-\sum_{\substack{\sigma\in\mathfrak S_3\\
                   \operatorname{sgn}(\sigma)=1}}
 q x_i u^{kk}\overline x_j\overline u^{ij}
-\sum_{\substack{\sigma\in\mathfrak S_3\\
                   \operatorname{sgn}(\sigma)=-1}}
 p x_i u^{kk}\overline x_j\overline u^{ij}
\\
&+\sum_{\sigma\in\mathfrak S_3}
 pq x_i u^{ik}\overline x_j\overline u^{jk}
+\sum_{\sigma\in\mathfrak S_3}
 pq x_i u^{jk}\overline x_j\overline u^{ik}
+\sum_{\substack{\sigma\in\mathfrak S_3\\j<k}}
 pq x_i u^{jk}\overline x_i\overline u^{jk}.
\end{aligned}
\end{equation}

\begin{remark}\label{rem:m72-parameter}
\begingroup\begingroup
For \(a\in\kk^\times,\) let \(g_a\in\operatorname{GL}(U)\) fix
\(u^{11},u^{22},u^{33}\) and send \(u^{ij}\) to \(a u^{ij}\) for
\(i\ne j;\) define \(\overline g_a\) in the same way on \(\overline U.\)
Then
\(
 (1_V\otimes g_a\otimes1_{\overline V}\otimes\overline g_a)
 (\mathsf{\Phi}_{p,q})=\mathsf{\Phi}_{ap,aq}.
\)
These maps induce an isomorphism of \(\mathbb Z\)-algebras
\(\mathcal L(p,q)\simeq\mathcal L(ap,aq)\) fixing every object, where
\(\mathcal L(p,q)\) is defined in \eqref{eq:m72-z-algebra} below.
Thus, the isomorphism class of \(\mathcal L(p,q)\) depends only on \(p/q.\)
\endgroup\endgroup
\end{remark}

Choose the basis-preserving identifications
\(V\simeq\overline V\) and \(U\simeq\overline U.\)  They define the
involution
\(
 x_i u^{jk}\overline x_l\overline u^{mn}
 \longmapsto
 x_l u^{mn}\overline x_i\overline u^{jk}
\)
of \(V\otimes U\otimes\overline V\otimes\overline U.\)

\begin{lemma}\label{lem:m72-potential-involution}
The potential \(\mathsf{\Phi}_{p,q}\) is invariant under this involution.
\end{lemma}

\begin{proof}
\begingroup
The first sum in \eqref{eq:m72-potential} is preserved by interchanging
\(j\) and \(k.\)  The second and fifth sums are exchanged, as are the
third and fourth, by interchanging \(i\) and \(j.\)
The sixth and seventh sums are each preserved by the same interchange;
the final summands are fixed individually.
\endgroup
\end{proof}

Contracting consecutive pairs of factors in the cyclic order of
\(\mathsf{\Phi}_{p,q}\) defines four maps
\begin{equation}\label{eq:m72-second-derivatives}
\begin{aligned}
\delta_{\overline V,\overline U}&:\overline V^*\otimes\overline U^*
       \longrightarrow V\otimes U,
&\delta_{\overline U,V}&:\overline U^*\otimes V^*
       \longrightarrow U\otimes\overline V,\\
\delta_{V,U}&:V^*\otimes U^*
       \longrightarrow\overline V\otimes\overline U,
&\delta_{U,\overline V}&:U^*\otimes\overline V^*
       \longrightarrow\overline U\otimes V.
\end{aligned}
\end{equation}

\begingroup\begingroup
In \(\delta_{A,B},\) the subscripts name the contracted factors; the
remaining factors are read in cyclic order.  For example, in the
contraction defining
\(\delta_{\overline U,V}(\overline\mu\otimes\xi),\) the monomial
\(x\otimes u\otimes\overline x\otimes\overline u\) contributes
\(\xi(x)\overline\mu(\overline u)\,u\otimes\overline x,\) multiplied
by its coefficient in \(\mathsf{\Phi}_{p,q}.\)
\endgroup\endgroup
\begin{lemma}\label{lem:m72-rank-seven}
Each map in \eqref{eq:m72-second-derivatives} has rank \(7.\)
\end{lemma}

\begin{proof}
By Lemma~\ref{lem:m72-potential-involution}, it is enough to
consider \(\delta_{\overline V,\overline U}\) and
\(\delta_{\overline U,V}\) from \eqref{eq:m72-second-derivatives}.
\begingroup
Define \(r_1,\ldots,r_7\in V\otimes U\) by
\begin{equation}\label{eq:m72-relation-vectors}
\begin{gathered}
r_1=x_1u^{33}-px_3u^{31},\quad
r_2=x_1u^{22}-qx_2u^{12},\quad
r_3=x_2u^{33}-qx_3u^{23},\\
r_4=x_2u^{11}-px_1u^{12},\quad
r_5=x_3u^{22}-px_2u^{23},\quad
r_6=x_3u^{11}-qx_1u^{31},\\
r_7=x_1u^{23}+x_2u^{31}+x_3u^{12}.
\end{gathered}
\end{equation}
\endgroup
Direct contraction in \eqref{eq:m72-potential} shows that
\(\operatorname{im}\delta_{\overline V,\overline U}
=\Span\{r_1,\ldots,r_7\},\)
\begingroup\begingroup
whereas \(\operatorname{im}\delta_{\overline U,V}\) is spanned by
\(s_1,\ldots,s_7\in U\otimes\overline V,\) where
\[
\begin{gathered}
s_1=u^{33}\overline x_1-qu^{31}\overline x_3,\quad
s_2=u^{22}\overline x_1-pu^{12}\overline x_2,\quad
s_3=u^{33}\overline x_2-pu^{23}\overline x_3,\\
s_4=u^{11}\overline x_2-qu^{12}\overline x_1,\quad
s_5=u^{22}\overline x_3-qu^{23}\overline x_2,\quad
s_6=u^{11}\overline x_3-pu^{31}\overline x_1,\\
s_7=u^{23}\overline x_1+u^{31}\overline x_2+u^{12}\overline x_3.
\end{gathered}
\]
\endgroup\endgroup
\begingroup
For \(1\le i\le7,\) {define
\(\overline r_i\in\overline V\otimes\overline U\) by placing bars on the
basis vectors in \(r_i\) from \eqref{eq:m72-relation-vectors}}.  Expanding
\eqref{eq:m72-potential} gives
\begin{equation}\label{eq:m72-potential-factorization}
 \mathsf{\Phi}_{p,q}=
 \sum_{k=1}^3(r_{2k-1}\otimes\overline r_{2k}
             +r_{2k}\otimes\overline r_{2k-1})
       +pq\,r_7\otimes\overline r_7.
\end{equation}
\begingroup\begingroup
Let \(\overline s_i\in\overline U\otimes V\) be the image of \(s_i\)
under the basis-preserving identifications.  For the cyclic permutation
\(
 c(x\otimes u\otimes\overline x\otimes\overline u)
 =u\otimes\overline x\otimes\overline u\otimes x
\)
one has
\[
 c(\mathsf{\Phi}_{p,q})=
 \sum_{k=1}^3(s_{2k-1}\otimes\overline s_{2k}
              +s_{2k}\otimes\overline s_{2k-1})
       +pq\,s_7\otimes\overline s_7.
\]
\endgroup\endgroup
In each list the first six vectors have distinct diagonal monomials,
whereas the seventh has none; hence, the seven vectors are independent.
The coefficient matrix in either expansion is
\(\operatorname{diag}(J,J,J,pq),\) where
\(J=\begin{psmallmatrix}0&1\\1&0\end{psmallmatrix},\) and has determinant
\(-pq\ne0.\)  Thus, each contraction has rank \(7\).
\endgroup
\end{proof}

Define the connected quadratic \(\mathbb Z\)-algebra
\(\mathcal L=\mathcal L(p,q)\) by
\begin{equation}\label{eq:m72-z-algebra}
\mathcal L_{j+1}^{j}=
\begin{cases}
V,&j\equiv0\pmod4,\\
U,&j\equiv1\pmod4,\\
\overline V,&j\equiv2\pmod4,\\
\overline U,&j\equiv3\pmod4,
\end{cases}
\qquad
\mathcal R_{j+1}^{j-1}=
\begin{cases}
\operatorname{im}\delta_{U,\overline V},
 &j\equiv0\pmod4,\\
\operatorname{im}\delta_{\overline V,\overline U},
 &j\equiv1\pmod4,\\
\operatorname{im}\delta_{\overline U,V},
 &j\equiv2\pmod4,\\
\operatorname{im}\delta_{V,U},
 &j\equiv3\pmod4.
\end{cases}
\end{equation}
Here
\(
 \mathcal R_{j+1}^{j-1}
 \subset
 \mathcal L_j^{j-1}\otimes\mathcal L_{j+1}^{j}
,\)
and \(\mathcal L\) is the tensor \(\mathbb Z\)-algebra on
its degree-one components modulo the ideal generated by these
relation spaces.


\begin{remark}\label{rem:m72-two-periodic}
The algebra \(\mathcal L\) is \(4\)-periodic by construction.
Lemma~\ref{lem:m72-potential-involution} and the chosen
identifications induce a family of isomorphisms
\(
 \theta_i^j\colon\mathcal L_i^j\xrightarrow{\sim}
 \mathcal L_{i+2}^{j+2}
\)
compatible with multiplication and preserving the relation spaces.  Thus,
\(\mathcal L\) is \(2\)-periodic.  Write \(\theta\) for this family and
\(\theta^n,\) \(n\in\mathbb Z,\) for its iterates, using the inverse
isomorphisms for \(n<0.\)
\end{remark}

Using the periodicity isomorphisms \(\theta_i^j,\) the spaces
\(
 \widetilde R_m=\mathcal L_0^{-2m}
,\) \(m\geq0,\) form a graded algebra
\(\widetilde R=\bigoplus_{m\geq0}\widetilde R_m\): for
\(a\in\widetilde R_m\) and \(b\in\widetilde R_n,\) multiplication is
\(a\star b=\theta^{-n}(a)b.\)  Here
\(\theta^{-n}(a)\in\mathcal L_{-2n}^{-2(m+n)},\) so the product belongs
to \(\widetilde R_{m+n}.\)

{The statements below appear in
\cite[Theorems~2--4]{Orlov2026}.}

\begingroup\begingroup\begingroup
\begin{theorem}\label{thm:m72-koszul}
The quadratic \(\mathbb Z\)-algebra \(\mathcal L\) is Koszul and AS-regular
of dimension \(4\) and Gorenstein parameter \(4.\)
\end{theorem}
\endgroup\endgroup\endgroup

\begin{proposition}\label{prop:m72-graded-algebra}
The graded algebra \(\widetilde R\) is quadratic, Koszul, noetherian,
AS-Gorenstein, and of polynomial growth.  For every \(m\geq0,\)
\begin{equation}\label{eq:m72-hilbert-polynomial}
 \dim_{\kk}\widetilde R_m
 =\frac32(m+1)^3-\frac12(m+1).
\end{equation}
\end{proposition}

\begingroup
By \cite[Theorem~5]{Orlov2026},
\(\qprf\!\operatorname{-}\mathcal L\) is the perfect category of a
three-dimensional non-commutative Fano minifold, denoted \(\mathsf M_{72};\)
we use its DG enhancement when referring to the non-commutative scheme.
\begingroup\begingroup
Put \(P_i=e_i\mathcal L\) and denote their images in
\(\qprf\!\operatorname{-}\mathcal L\) by the same letters.
With the right-module convention of Definition~\ref{def:z-algebra},
\(\operatorname{Hom}_{\mathcal L}(P_i,P_j)=\mathcal L_i^j.\)
Thus, \(\mathcal P=(P_3,P_2,P_1,P_0)\) is a full strong exceptional collection
in \(\qprf\!\operatorname{-}\mathcal L;\)
see \cite[after Theorem~2]{Orlov2026}.

For the exceptional object \(P_0,\) put
\(R=R(\mathsf M_{72},P_0)\) and \(T=S^{-1}[3].\)
The construction of \cite[Theorem~5]{Orlov2026} and the four-step
periodicity identify \(TP_i\) with \(P_{i-4},\) with the action on
morphisms given by \(\theta^{-2}.\)  In particular,
\(T^mP_0\simeq P_{-4m}.\)
Then
\[
 R_m=\operatorname{Hom}(P_0,T^mP_0)
 \simeq\mathcal L_0^{-4m}=\widetilde R_{2m},\qquad m\ge0.
\]
Under these identifications, the product \(T^n(f)\circ g,\) for
\(f\in R_m\) and \(g\in R_n,\) corresponds to \(\theta^{-2n}(f)g.\)
This is the product in the second Veronese algebra
\(\widetilde R^{(2)}=\bigoplus_{m\ge0}\widetilde R_{2m},\) so
\(R\simeq\widetilde R^{(2)}\) as graded algebras.
Substituting \(2m\) for \(m\) in
\eqref{eq:m72-hilbert-polynomial} gives
\[
        \dim_\kk R_m=12m^3+18m^2+8m+1.
\]

The Euler matrix of the exceptional collection \(\mathcal P\) is
\endgroup\endgroup
\[
 A_{\mathrm{pr}}=
 \begin{pmatrix}
 1&3&11&18\\
 0&1&6&11\\
 0&0&1&3\\
 0&0&0&1
 \end{pmatrix}.
\]
{Let \(\mathcal P^\vee\) be the right dual exceptional collection of
\(\mathcal P\) in \(\qprf\!\operatorname{-}\mathcal L.\)}
Put \(S_\pm=\operatorname{diag}(1,-1,1,-1).\) Inverting \(A_{\mathrm{pr}}\) gives the Euler matrix of the {right dual exceptional collection
\(\mathcal P^\vee\)}:
\[
 A_{M_{72}}=S_\pm A_{\mathrm{pr}}^{-1}S_\pm=
 \begin{pmatrix}
 1&3&7&6\\
 0&1&6&7\\
 0&0&1&3\\
 0&0&0&1
 \end{pmatrix}.
\]
This matrix is positive and has Euler degree \(72.\)  It therefore
{lies in the mutation class of Euler degree \(72\) in}
Theorem~\ref{thm:main-reduction}.
This completes the proof of Theorem~\ref{thm:intro-classification}
stated in the Introduction.
\endgroup
\appendix

\begingroup
\refstepcounter{section}
\section*{Appendix}
\label{app:half-turn-realization}

{In this section we construct four half-turns
\(j_1,\ldots,j_4\in\PSL_2(\mathbb R)\) associated with a positive Euler
matrix (Lemma~\ref{lem:positive-euler-half-turns}).
We prove that
\(\Gamma=\langle j_1j_2,j_2j_3,j_3j_4\rangle\) is discrete and
torsion-free and that \(\Gamma\backslash\mathbb H,\) with its quotient
metric, is complete, has finite area, and has genus one and
two cusps (Lemma~\ref{lem:four-half-turns}).
Corollary~\ref{cor:half-turn-fricke} then shows that, after
conjugation if necessary, the projectivization of
\(\rho\colon F_3\to\SL_2(\mathbb R)\) from
Lemma~\ref{lem:positive-euler-half-turns} is a Fricke representation.}

{A subgroup of \(\PSL_2(\mathbb R)\) is called
\emph{non-elementary} if it has no finite orbit in
\(\mathbb H\cup\partial\mathbb H.\)}
On \(\mathfrak{sl}_2(\mathbb R),\) consider the Lorentzian form
\begin{equation}\label{eq:lorentz-form-sl2}
        \langle R,S\rangle=-\frac12\tr(RS)
\end{equation}
of signature \((1,2).\)  {Each of the two connected components of}
\(
 \mathscr H=\{J\in\mathfrak{sl}_2(\mathbb R)\mid\langle J,J\rangle=1\}
\)
{is a hyperboloid model of \(\mathbb H.\)  Under the standard
\(\PSL_2(\mathbb R)\)-equivariant identification, a matrix
\(J\in\mathscr H\) represents a point \(x_J\in\mathbb H.\)  Since
\(J^2=-I,\) its class \([J]\in\PSL_2(\mathbb R)\) is the order-two
orientation-preserving isometry fixing \(x_J,\) that is, the half-turn
about \(x_J.\)  The matrices \(J\) and \(-J\) represent the same
half-turn.}

\begin{lemma}\label{lem:positive-euler-half-turns}
Let \(A=A(a,b,c,d,e,f)\) be a positive Euler matrix and put \(B=A+A^t.\)
There are \(J_1,\ldots,J_4\) in one component of \(\mathscr H\) such that
\begin{equation}\label{eq:half-turn-gram}
        -\tr(J_iJ_j)=B_{ij}.
\end{equation}
The assignment
\begin{equation}\label{eq:half-turn-representation}
        \rho(\alpha_i)=-J_iJ_{i+1},\qquad 1\le i\le3,
\end{equation}
defines an irreducible representation with \(A_\rho=A.\)  Write
\(
 j_i=[J_i],\;
 \Gamma=\overline\rho(F_3),\;
 H=\langle j_1,j_2,j_3,j_4\rangle.
\)
Both \(\Gamma\) and \(H\) are discrete and non-elementary, and
\({\kappa}=j_1j_2j_3j_4\) is a parabolic element.
\end{lemma}

\begin{proof}
The first three leading principal minors of \(B/2\) are
\[
        1,\qquad 1-\frac{a^2}{4}<0,\qquad
        \frac{\Delta(a,b,e)}4>0.
\]
Together with \(\det B=0,\) these signs give rank three and inertia
\((1,2)\) on \(\mathbb R^4/\ker B.\)  {If \(\overline e_i\)
denotes the image in this quotient of the \(i\)-th standard basis vector of
\(\mathbb R^4,\) an isometry with the Lorentzian space
\eqref{eq:lorentz-form-sl2} sends \(\overline e_i\) to a matrix \(J_i\)
satisfying \eqref{eq:half-turn-gram}.}
They have norm one.  Moreover,
\(\langle J_i,J_j\rangle=B_{ij}/2>0\) for \(i\ne j,\) so all \(J_i\)
lie in the same connected component of \(\mathscr H.\)

Put \(M_i=-J_iJ_{i+1}.\)  The identities \(J_i^2=-I\) give
\[
        M_i\cdots M_{j-1}=-J_iJ_j, \qquad 1\le i<j\le4,
\]
so the six traces defining \(A_\rho\) are exactly the entries of \(A.\)
Let \(t_i=\tr\rho(K_i).\)  By \eqref{eq:boundary-traces} and
\eqref{eq:maximal-unipotence-relations},
\(t_1+t_2=\Phi(A),\) \(t_1t_2=4,\) and \(\Phi(A)^2=16.\)
Thus, \(t_1=t_2=\Phi(A)/2.\)
Since
\(\rho(K_1)=M_1M_3=J_1J_2J_3J_4,\) we obtain
\begin{equation}\label{eq:half-turn-pfaffian-trace}
        \tr(J_1J_2J_3J_4)=\frac{\Phi(A)}2.
\end{equation}
Set \(\varepsilon=\Phi(A)/4\in\{\pm1\}.\)  Then
\eqref{eq:half-turn-pfaffian-trace} and \(t_1=t_2\) give
\begin{equation}\label{eq:unknown-sign-boundary-traces}
        \tr\rho(K_1)=\tr\rho(K_2)=2\varepsilon.
\end{equation}

{The Fricke commutator formula and
Proposition~\ref{prop:strict-principal-deltas} give}
\(
        \tau:=\tr[M_1,M_2]=2-\Delta(a,b,e)<-2,
\)
{so \(\rho\) is irreducible.  Indeed, if \(\rho\) were reducible,
then \(M_1\) and \(M_2\) could be simultaneously upper triangularized;
their commutator would then have both diagonal entries equal to one and
hence, trace \(2.\)}  If \(\rho(K_1)=\varepsilon I,\) then
\(M_3=\varepsilon M_1^{-1}\) and
\(
 \rho(K_2)=M_1M_2M_3M_2^{-1}=\varepsilon[M_1,M_2],
\)
contradicting \(\tr\rho(K_2)=2\varepsilon.\)  Hence, \(\rho(K_1)\) is
non-central, and \({\kappa}\) is parabolic.

\begingroup\begingroup
Since \(A_\rho=A,\) Proposition~\ref{prop:euler-fricke-dictionary}
identifies \([\rho]\) with \([\rho_A].\)  Corollary~\ref{cor:all-traces-integral}
gives integral traces for every element of \(\rho(F_3).\)
Lemma~\ref{lem:integral-traces-modular} conjugates this group into
\(\SL_2(\mathbb Z),\) so \(\Gamma\) is discrete.
\endgroup\endgroup
{The fixed
points in \(\mathbb H\) of the half-turns \(j_1,j_2,j_3\) are not
collinear, since the determinant of their \(3\times3\) Gram matrix is
\(\Delta(a,b,e)/4>0.\)  Hence, the
axis of the hyperbolic element \(j_1j_2\) and the axis of the hyperbolic
element \(j_2j_3\) are distinct and intersect at the fixed point of
\(j_2.\)  Thus, \(\Gamma\) is non-elementary.}
For \(i<j,\)
\(
 j_i j_j=(j_i j_{i+1})\cdots(j_{j-1}j_j)\in\Gamma
.\)
Hence, every product of an even number of the generators \(j_i\) lies in
\(\Gamma,\) and \(H=\Gamma\cup\Gamma j_1.\)  It follows that \(H\) is discrete and
non-elementary as well.
\end{proof}

\begin{lemma}\label{lem:four-half-turns}
Let \(J_1,\ldots,J_4\) belong to one component of \(\mathscr H,\) and put
\(j_i=[J_i].\)  Suppose that \(H=\langle j_1,j_2,j_3,j_4\rangle\) is
discrete and non-elementary and that \({\kappa}=j_1j_2j_3j_4\) is
parabolic.  Then the homomorphism
\[
 (\mathbb Z/2\mathbb Z)*(\mathbb Z/2\mathbb Z)*(\mathbb Z/2\mathbb Z)*(\mathbb Z/2\mathbb Z)\longrightarrow H,\qquad s_i\longmapsto j_i,
\]
is an isomorphism.
\begingroup\begingroup
Let \(\xi\in\partial\mathbb H\) be the unique fixed point of
\(\kappa.\)  The orbifold \(H\backslash\mathbb H\) has finite area,
genus zero, four cone points with stabilizers of order two and one cusp, represented by
\(\xi,\) and
\(
 \operatorname{Stab}_H(\xi)=\langle\kappa\rangle.
\)
\endgroup\endgroup
The homomorphism \(\ell\colon H\to \mathbb Z/2\mathbb Z,\) \(\ell(j_i)=1,\) has the
torsion-free kernel \(\Gamma:=\ker\ell,\) with basis
\[
        h_1=j_1j_2,\qquad h_2=j_2j_3,\qquad h_3=j_3j_4.
\]
\begingroup\begingroup
The complete finite-area surface \(\Gamma\backslash\mathbb H\) has genus
one and two cusps, represented by the distinct \(\Gamma\)-orbits of
\(\xi\) and \(j_1\xi.\)  Their stabilizers in \(\Gamma\) are
\begin{equation}\label{eq:half-turn-cusp-stabilizers}
 \operatorname{Stab}_\Gamma(\xi)=\langle\kappa\rangle,
 \qquad
 \operatorname{Stab}_\Gamma(j_1\xi)
 =\langle j_1\kappa^{-1}j_1\rangle.
\end{equation}
\endgroup\endgroup
\end{lemma}

\begin{proof}
{The standard presentation of a finitely generated
non-elementary Fuchsian group is given in
\cite[Section~4, equation~(11)]{BGLS2010}.  Let \(g\) be
the genus of \(H\backslash\mathbb H,\) let \(m_1,\ldots,m_r\) be the orders of the stabilizers
at its cone points, and {let \(s\) be the number of cusps, and let \(b\) be the
number of funnel ends of \(H\backslash\mathbb H.\)}
Use the presentation
\[
\begin{aligned}
H=\bigl\langle&u_1,v_1,\ldots,u_g,v_g,x_1,\ldots,x_r,
 c_1,\ldots,c_s,d_1,\ldots,d_b\bigm|\\
 &x_i^{m_i}=1\ (1\le i\le r),\quad
 x_1\cdots x_r c_1\cdots c_s d_1\cdots d_b
 \prod_{i=1}^{g}[u_i,v_i]=1\bigr\rangle,
\end{aligned}
\]
where \(c_j\) is represented by a loop around the \(j\)-th cusp.
Since \(\kappa\) is parabolic, \(H\backslash\mathbb H\) has at least one cusp, so \(s\ge1.\)
The product relation gives
\[
 c_s=(x_1\cdots x_r c_1\cdots c_{s-1})^{-1}
 \left(d_1\cdots d_b\prod_{i=1}^{g}[u_i,v_i]\right)^{-1}.
\]
Eliminating \(c_s\) gives}
\[
        H\simeq F_{2g+s+b-1}*(\mathbb Z/m_1\mathbb Z)*\cdots*(\mathbb Z/m_r\mathbb Z),
\]
and hence
\begin{equation}\label{eq:fuchsian-abelianization}
 H^{\mathrm{ab}}\simeq
 \mathbb Z^{2g+s+b-1}\oplus \mathbb Z/m_1\mathbb Z\oplus\cdots\oplus \mathbb Z/m_r\mathbb Z.
\end{equation}
Since \(H\) is generated by four involutions, its abelianization is a
quotient of \((\mathbb Z/2\mathbb Z)^4.\)  Comparing this with
\eqref{eq:fuchsian-abelianization} gives
\(g=0,\) \(s=1,\) \(b=0,\) and every \(m_i=2.\)
{Since \(H\) is finitely generated and \(b=0,\) the quotient has no
funnel ends and has finite area.}
Thus, \(H\simeq (\mathbb Z/2\mathbb Z)^{*r}.\)  Non-elementarity gives \(r\ge3.\)
Since \(H^{\mathrm{ab}}\simeq(\mathbb Z/2\mathbb Z)^r\) is generated by
\([j_1],\ldots,[j_4],\) we have \(r\le4.\)

Let \(c\) generate a maximal parabolic subgroup {of \(H\)}.
There is only one cusp, so \({\kappa}\) is
{conjugate in \(H\)} to \(c^n\) for some \(n\ne0\){, see}
\cite[Theorem~10.3.2 and Corollary~10.3.3]{Beardon1983}.
Use the presentation
\(H=\langle x_1,\ldots,x_r,c\mid x_1^2=\cdots=x_r^2=1,
x_1\cdots x_r c=1\rangle.\)
Identify \(H^{\mathrm{ab}}\) with \((\mathbb Z/2\mathbb Z)^r\) by
sending \([x_i]\) to the \(i\)-th standard basis vector \(e_i.\)
The relation \(x_1\cdots x_rc=1\) gives \([c]=e_1+\cdots+e_r.\)
Suppose that \(r=3.\)  By the normal-form theorem for free products,
each \(j_i\) is conjugate to one of \(x_1,x_2,x_3,\) see
\cite[Theorem~6.2.5\textup{(iii)}]{Robinson1996}.
Since the four classes \([j_i]\) generate \(H^{\mathrm{ab}},\)
all three basis vectors occur among them, with multiplicities
\((2,1,1),\) up to permutation.  Thus, \([\kappa]\) has two nonzero
coordinates, whereas \([c^n]\) is either \(0\) or \(e_1+e_2+e_3.\)
This contradicts the conjugacy of \(\kappa\) and \(c^n,\) so \(r=4.\)

The group \((\mathbb Z/2\mathbb Z)^{*4}\) is residually finite
\cite[Theorem~4.1, Corollary~\textup{(ii)}]{Gruenberg1957} and finitely
generated, hence, Hopfian
\cite[Theorem~6.1.11]{Robinson1996}.  Since \(H\simeq (\mathbb Z/2\mathbb Z)^{*4},\) the
surjection \(s_i\mapsto j_i\) is therefore an isomorphism.
The cyclically reduced word \({\kappa}=j_1j_2j_3j_4\) is not a proper power:
the cyclic normal form of a proper power repeats a shorter word, whereas
each \(j_i\) occurs exactly once in \({\kappa}.\)  Hence, \(|n|=1,\) so
\(\operatorname{Stab}_H(\xi)=\langle\kappa\rangle.\)
To describe \(\ker\ell,\) put \(k_i=j_1j_i,\) \(i=2,3,4.\)
Substituting \(j_i=j_1k_i\) into the free-product presentation gives
\[
 H=\langle j_1,k_2,k_3,k_4\mid
       j_1^2=1,\ j_1k_ij_1=k_i^{-1}\ (i=2,3,4)\rangle
   \simeq F(k_2,k_3,k_4)\rtimes (\mathbb Z/2\mathbb Z).
\]
Here the nontrivial element of \(\mathbb Z/2\mathbb Z\) sends each generator \(k_i\)
to \(k_i^{-1}.\)  Thus, \(\Gamma=\ker\ell\) is free on \(k_2,k_3,k_4.\)
The identities
\(
        k_2=h_1, k_3=h_1h_2, k_4=h_1h_2h_3
\)
show that \(h_1,h_2,h_3\) form a free basis of~\(\Gamma.\)

The cusps above the unique cusp of \(H\backslash\mathbb H\)
are indexed by the double cosets
\(\Gamma\backslash H/\langle\kappa\rangle.\)  Since \(\Gamma\) is normal
of index two and \(\langle\kappa\rangle\subset\Gamma,\) there are exactly
two, represented by \(1\) and \(j_1.\)
{The corresponding \(\Gamma\)-orbits of parabolic fixed points are
represented by \(\xi\) and \(j_1\xi.\)  Intersecting
\(\operatorname{Stab}_H(\xi)=\langle\kappa\rangle\) and its
\(j_1\)-conjugate with \(\Gamma\) gives
\eqref{eq:half-turn-cusp-stabilizers}.}  The quotient
\(\Gamma\backslash\mathbb H\) is a complete finite-area surface,
since \(\Gamma\) is torsion-free and has index two in \(H.\)
Its fundamental group has rank three, so its genus \(g'\) satisfies
\(2g'+2-1=3;\) hence, \(g'=1.\)
\end{proof}
\begin{corollary}\label{cor:half-turn-fricke}
In the notation of Lemma~\ref{lem:positive-euler-half-turns}, after
conjugating \(\rho\) by a matrix in \(\GL_2(\mathbb R),\) if necessary,
its projectivization is a Fricke representation of
\(\Sigma_{1,2}^{\circ}.\)  This conjugation does not change \(A_\rho.\)
\end{corollary}
\begin{proof}
Put \(Y=\Gamma\backslash\mathbb H.\)  The projectivization of \(\rho\)
sends \(\alpha_i\) to \(h_i=j_ij_{i+1}.\)  By
Lemma~\ref{lem:four-half-turns}, the elements \(h_1,h_2,h_3\) form a free
basis of \(\Gamma;\) hence
\(\overline\rho\colon F_3\xrightarrow{\sim}\Gamma.\)  Moreover,
\[
 \overline\rho(K_1)=h_1h_3=\kappa,
 \qquad
 \overline\rho(K_2)=h_1h_2h_3h_2^{-1}
 =j_1\kappa^{-1}j_1.
\]
\begingroup\begingroup
Thus, \(\overline\rho\) sends the conjugacy classes of the cyclic subgroups
\(\langle K_1\rangle,\langle K_2\rangle\subset F_3\) to the two
\(\Gamma\)-conjugacy classes of maximal parabolic subgroups of
\(\Gamma.\)

Let \(\widehat Y=Y\cup\{q_1,q_2\}\) be the compactification obtained
by adding one point at each cusp.  Label \(q_i\) by the
\(\Gamma\)-conjugacy class of
\(\langle\overline\rho(K_i)\rangle.\)
Choose an orientation-preserving homeomorphism
\(\widehat f_0:\widehat\Sigma\to\widehat Y\) with
\(\widehat f_0(p_i)=q_i\) for \(i=1,2,\) and put
\(f_0=\widehat f_0|_{\Sigma_{1,2}^{\circ}}.\)
Choose a basepoint \(y_0\in Y\) and a lift to \(\mathbb H,\) so that
the covering \(\mathbb H\to Y\) identifies \(\pi_1(Y,y_0)\) with
\(\Gamma.\)  A path from \(y_0\) to \(f_0(*)\) then defines an
isomorphism \(f_{0*}:F_3\to\Gamma.\)  For the chosen \(f_0,\) the class
\[
 \eta=[f_{0*}^{-1}\circ\overline\rho]\in\operatorname{Out}(F_3)
\]
is independent of the basepoint, lift, and connecting path.
For each \(i=1,2,\) \(\eta\) fixes the conjugacy class of the cyclic
subgroup \(\langle K_i\rangle.\)
The extended Dehn--Nielsen--Baer theorem
\cite[Theorem~8.8]{FarbMargalit2012} therefore gives a
homeomorphism \(\widehat h:\widehat\Sigma\to\widehat\Sigma\)
fixing \(p_1\) and \(p_2\) individually, whose restriction
\(h=\widehat h|_{\Sigma_{1,2}^{\circ}}\) induces \(\eta.\)
Consequently, \(f=f_0\circ h\) induces \(\overline\rho,\) up to
conjugation in \(\Gamma.\)
\endgroup\endgroup

If \(f\) reverses orientation, put
\(r=\operatorname{diag}(1,-1).\)  The map
\(z\mapsto-\overline z\) conjugates the action of \(\Gamma\) to that of
\(r\Gamma r^{-1}\) and induces an orientation-reversing isometry
\(
 Y\longrightarrow r\Gamma r^{-1}\backslash\mathbb H.
\)
\begingroup\begingroup
Composing this isometry with \(f\) gives an orientation-preserving homeomorphism
\(
 \Sigma_{1,2}^{\circ}\longrightarrow r\Gamma r^{-1}\backslash\mathbb H
\)
whose induced homomorphism on fundamental groups is the projectivization
of \(r\rho r^{-1},\) up to conjugation in \(\PSL_2(\mathbb R).\)
\endgroup\endgroup
Conjugation preserves all traces,
so \(A_{r\rho r^{-1}}=A_\rho.\)
\end{proof}

\begin{proof}[Proof of Theorem~\ref{thm:positive-fricke-character}]
Let \(\rho\) be the representation constructed in
Lemma~\ref{lem:positive-euler-half-turns} by
\eqref{eq:half-turn-representation}.  Apply
Corollary~\ref{cor:half-turn-fricke}, replacing \(\rho\) by the
{conjugate from that corollary} if necessary.  Then \(A_\rho=A,\) and the
projectivization of \(\rho\) is a Fricke representation whose quotient is
a complete finite-area surface of genus one with two cusps.
Proposition~\ref{prop:euler-fricke-dictionary} gives
\([\rho]=[\rho_A],\) and \eqref{eq:unknown-sign-boundary-traces} gives
\[
 \tr\rho(K_1)=\tr\rho(K_2)=\frac{\Phi(A)}2\in\{-2,2\}.
\]
\begingroup\begingroup
Lemma~\ref{lem:four-half-turns} shows that, for each \(i=1,2,\)
\(\overline\rho(K_i)\) is a parabolic element with a unique
fixed point \(\xi_i\in\partial\mathbb H,\) and
\(
 \operatorname{Stab}_{\overline\rho(F_3)}(\xi_i)
 =\langle\overline\rho(K_i)\rangle.
\)
The points \(\xi_1\) and \(\xi_2\) belong to distinct
\(\overline\rho(F_3)\)-orbits.
\endgroup\endgroup
Hence, \(\rho(K_1)\) and \(\rho(K_2)\) are non-central.

Every representative \(\rho'\) of \([\rho_A]\) satisfies
\[
 \tr\rho'([U,Y])=2-\Delta(a,d,f)<-2.
\]
If \(\rho'\) were reducible, \(\rho'(U)\) and \(\rho'(Y)\) could be
simultaneously upper triangularized, and their commutator would have trace
\(2.\)  Thus, every representative of \([\rho_A]\) is irreducible.
\end{proof}
\endgroup

\end{document}